\documentclass[11pt,reqno]{amsart}
\usepackage{fullpage}
\usepackage{amsmath,amsthm,verbatim,amssymb,amsfonts,amscd, graphicx,bbm,bm}
\usepackage{graphics}
\usepackage{mathtools}
\usepackage{mathrsfs}
\usepackage{eufrak}
\usepackage{esint}
\usepackage{color}

\usepackage[pdfstartview=FitH, bookmarksnumbered=true,bookmarksopen=true, colorlinks=true, pdfborder=001, citecolor=blue, linkcolor=blue,urlcolor=blue]{hyperref}

\newtheorem{theorem}{Theorem}
\newtheorem{proposition}[theorem]{Proposition}
\newtheorem{lemma}[theorem]{Lemma}
\newtheorem{corollary}[theorem]{Corollary}

\theoremstyle{definition}
\newtheorem{remark}{Remark}

\newtheorem{question}{Question}

\newcommand{\cref}[1]{Corollary~\ref{c.#1}}

\numberwithin{equation}{section}
\numberwithin{theorem}{section}

\newcommand{\R}{\mathbb{R}}
\newcommand{\ol}{\overline}
\newcommand{\eps}{\varepsilon}
\newcommand{\bD}{\mathbb{D}}
\newcommand{\frakH}{\mathfrak{H}}
\newcommand{\cS}{\mathcal{S}}
\newcommand{\bS}{\mathbb{S}}
\newcommand{\bI}{\mathbb{I}}
\newcommand{\bK}{\mathbb{K}}
\newcommand{\cH}{\mathcal{H}}
\newcommand{\cHR}{\mathcal{H}_\mathbf{R}}

\title[Uniform convergence on transmission problems]{Uniform convergence for linear elastostatic systems with periodic high contrast inclusions}
\author{Xin Fu}
\address[X. Fu]{Yau Mathematical Sciences Center, Tsinghua University, Beijing 100084, P.R. China}
\email{fux20@mails.tsinghua.edu.cn}

\author{Wenjia Jing}
\address[W. Jing]{Yau Mathematical Sciences Center, Tsinghua University, Beijing 100084 and Yanqi Lake Beijing Institute of Mathematical Sciences and Applications, Beijing 101407, P.R. China}
\email{wjjing@tsinghua.edu.cn}

\date{\today}

\begin{document}
	\maketitle	

	\begin{abstract}
		We consider the Lam\'{e} system of linear elasticity with periodically distributed inclusions whose elastic parameters have high contrast compared to the background media. We develop a unified method based on layer potential techniques to quantify three convergence results when some parameters of the elastic inclusions are sent to extreme values. More precisely, we study the \emph{incompressible inclusions} limit where the bulk modulus of the inclusions tends to infinity, the \emph{soft inclusions} limit where both the bulk modulus and the shear modulus tend to zero, and the \emph{hard inclusions} limit where the shear modulus tends to infinity. 		Our method yields convergence rates that are independent of the periodicity of the inclusions array, and are sharper than some earlier results of this type. A key ingredient of the proof is the establishment of uniform spectra gaps for the elastic Neumann-Poincar\'{e} operator associated to the collection of periodic inclusions that are independent of the periodicity.
\smallskip


\smallskip

\noindent{\bf Mathematics subject classification (MSC 2020)}: 35B27, 35J08, 74G55
	\end{abstract}
	
	\section{Introduction}
	
	In this paper, we study the linear elastostatic behavior of a collection of high contrast inclusions embedded in a homogeneous background. Partial differential equations in high contrast media serve as natural models in physics and engineering, e.g.\;anomalous tissues in biomedical imaging \cite{manduca2001magnetic}, meta-materials exhibiting novel electro-magnetic or elastic properties \cite{sakoda2004optical,MR3769919}, etc. Mathematical studies of such PDEs are valuable both in theory and in practice; see \cite{ammari2009layer,ammari2007polarization,bonnetier2019homogenization} and references therein for more details.

	In the linear elastostatics setting, Lam\'e systems in perforated domains (the part outside the inclusions) with proper boundary conditions are widely used to model high contrast, namely, extremely soft or hard, inclusions; see \cite{MR548785,MR1195131,MR4240768}. In this paper, we rigorously justify several models of this type, by proving that they are the asymptotic limits of the transmission problem in the whole domain when the elastic parameters of the high contrast inclusions are sent to certain extreme values. More importantly, we quantify the convergence rates and study the dependence on the geometric setup of the inclusions.

	To fix ideas, let $\Omega \subset \mathbb{R}^d$ ($d=2,3$) model the domain occupied by the elastic material, and let $D\subset \Omega$ be the part occupied by the high contrast inclusions.

\noindent\emph{Geometric Setup}. For the domains $\Omega$ and $D$, we impose (part of) the following assumptions.

\noindent{(A1)}\;\;$\Omega\subset \R^d$ ($d=2,3$) is open bounded with connected Lipschitz boundary $\partial \Omega$. $D\subset \Omega$ is open with Lipschitz boundary $\partial D$, and $\partial D$ has a finite number $N$ of connected components. Moreover, $\Omega\setminus \ol D$ is connected with Lipschitz boundary $\partial \Omega \cup \partial D$. The connected components of $D$ are enumerated as $D_{i}$, $i=1,\cdots, N$.

\noindent{(A2)}\;\;$\Omega$ is as in (A1). Given $\eps \in (0,1)$, $D = D_\eps$ is part of an $\eps$-periodic array of small inclusions constructed as follows, in several steps.

Let $Y=(-\frac12,\frac12)^d$ be the unit cell,
	and let $\omega \subset Y$ be an open subset with connected Lipschitz boundary such that $\mathrm{dist}(\omega,\partial Y)>0$; for simplicity, assume $\omega$ is simply connected. $\omega$ is then the model inclusion in the unit scale, and $Y_{\rm f} = Y\setminus \ol \omega$ is the model environment in the unit scale. Given $\varepsilon>0$ and $\mathbf{n}\in \mathbb{Z}^d$, we denote $\eps(\mathbf{n}+Y)$ and $\eps(\mathbf{n}+\omega)$ by $Y_{\varepsilon}^{\mathbf{n}}$  and $\omega_{\varepsilon}^{\mathbf{n}}$, respectively. 
	Let $\Pi_{\varepsilon}$ be the set of lattice points $\mathbf{n}$ such that $\overline{Y_{\varepsilon}^{\mathbf{n}} }$ be contained in $\Omega$, i.e.,
	\begin{equation}
		\Pi_{\varepsilon}:=\left\{   \mathbf{n}\in \mathbb{Z}^d:  \overline{Y_{\varepsilon}^{\mathbf{n}} }\subset  \Omega  \right\}.
	\end{equation}
	Then the inclusions set $D = D_\eps$ and the background part $\Omega_\eps$ are defined by
	\begin{equation}
	\label{eq:Depsdef}
		D_{\varepsilon}:=\bigcup_{\mathbf{n}\in \Pi_{\varepsilon} } \omega_{\varepsilon}^{\mathbf{n}} , \quad \Omega_{\varepsilon}:=\Omega \setminus \overline{D_{\varepsilon}}.
	\end{equation}
For each fixed $\eps$, the number of connected components of $D_\eps$ is $N = |\Pi_\eps|$. Moreover, we define $Y_{\varepsilon}$ and $K_\eps$ by: 
	\begin{equation*}
		K_\eps = \Omega\setminus \Big(\bigcup_{\mathbf{n}\in \Pi_{\varepsilon} } \ol Y_{\varepsilon}^{\mathbf{n}}\Big), \qquad Y_\eps = \Omega\setminus \ol K_\eps.
	\end{equation*}
Intuitively, $Y_{\varepsilon}$ is the stack of $\varepsilon$-cells $Y^{\mathbf{n}}_{\varepsilon}$ contained in $\Omega$, and $K_\eps$ is the cushion area; see Fig. \ref{figure}. 
\begin{figure}[h]  
	\label{figure}
	\centering         
	\includegraphics[scale=0.1]{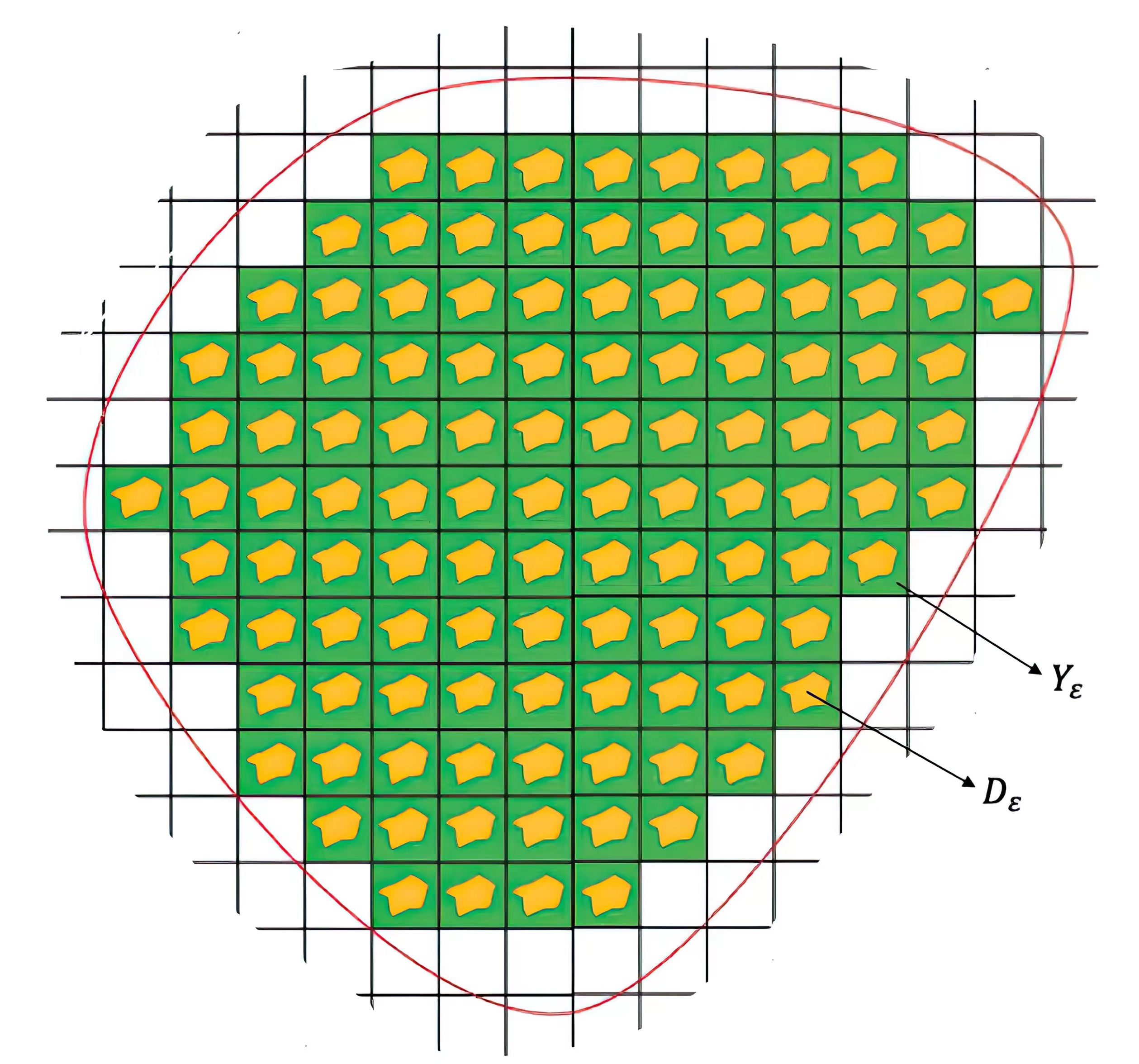}
	\caption{The domains $Y_{\varepsilon}$ and $D_{\varepsilon}$.}
\end{figure}

For each fixed $\eps > 0$, we see that $D= D_\eps$ constructed in (A2) satisfies the conditions in (A1). Throughout the paper we assume (A1) holds, and when we talk about uniform convergence, so the parameter $\eps > 0$ enters, we are considering a family of geometric configurations $(\Omega,D_\eps)$ satisfying (A2), and our aim is to derive results that are uniform in $\eps$.

A pair of real numbers $(\lambda,\mu)$ are called \emph{admissible} and  referred to as a Lam\'e pair, if they satisfy
\begin{equation}\label{elliptic condition}
		\mu>0\quad  \mathrm{and}\quad  d\lambda+2\mu>0.
	\end{equation}
For a Lam\'e pair $(\lambda,\mu)$, the elastostatic system (Lam\'{e} system) reads
	\begin{equation}\label{Lame system}
		\mathcal{L}_{\lambda,\mu}\mathbf{u}:= \mu \Delta \mathbf{u} +(\lambda+\mu)\nabla\mathrm{div}\, \mathbf{u},
	\end{equation}
	where $\mathbf{u} = (u^1,\dots,u^d)$ represents the displacement field. The admissibility condition guarantees that the Lam\'e operator is elliptic, and  
	physical laws ensure that the natural materials always admit this condition; see \cite{landau1986theory}. The Lam\'e operator can also be written as $\nabla \cdot \sigma(\mathbf{u})$ where
	\begin{equation*}
	\sigma(\mathbf{u}) := \lambda(\mathrm{div}\, \mathbf{u})\mathbb{I}_d + 2\mu \bD(\mathbf{u})
  \end{equation*}
  is the stress tensor. Here and below, $\mathbb{I}_d$ is the $d\times d$ identity matrix, $\bD$ denotes the symmetrized differential operator
	\begin{equation}
	\bD(\mathbf{u}) = \frac12( \nabla + \nabla^{\rm T}) \mathbf{u} = \frac12( \partial_i u^j  + \partial_j u^i)_{ij},
	\end{equation}
	and the superscript `${\rm T}$' denotes the transpose of a matrix. 
	The corresponding conormal derivative (boundary traction) at the boundary of a domain $E$ is
	\begin{equation}\label{conormal derivative}
	  \left. \frac{\partial \mathbf{u}}{\partial \nu_{(\lambda,\mu)}}\right|_{\partial E}:= \sigma(\mathbf{u}) \mathbf{N} = \lambda (\mathrm{div}\, \mathbf{u})\mathbf{N}+2\mu \bD(\mathbf{u})  \mathbf{N}\quad \mathrm{on} \ \partial E.
	\end{equation}
	Throughout the paper 
	$\mathbf{N}$ is the outward unit normal vector to $\partial E$.
	
	We work with the standard Sobolev space $H^1(\Omega)$ (vector versions) and its trace $H^{\frac12}(\partial E)$ at the boundary of a domain $E$, and the $L^2$ dual space $H^{-\frac12}(\partial E)$. We consider the space $\mathbf{R}$ of rigid motions in $\R^d$, defined by
	\begin{equation*}
		\mathbf{R}:=\big\{ \mathbf{r}=(r_1,\cdots,r_d)^{\rm T} : \bD(\mathbf{r}) = 0 \ \text{in $\R^d$} \big\}  ,
	\end{equation*}
 It is clear that $\mathbf{R}$ has dimension $d(d+1)/2$ and is spanned by   
 \begin{equation}\label{basis of R}
 	\mathbf{e}_1,\cdots,\mathbf{e}_d,\  x_j\mathbf{e}_i-x_i\mathbf{e}_j , \quad \mathrm{for} \ 1\leq i\neq j \leq d,
  \end{equation}
where $\mathbf{e}_i$ denotes the standard basis vector of $\mathbb{R}^d$. Those basis vectors are referred to as $\mathbf{r}_j$, $j=1,\dots,d(d+1)/2$. We define $H^{-\frac{1}{2}}_{\mathbf{R}}(\partial D)$ as the subspace of $H^{-\frac12}(\partial D)$ that is orthogonal to $\mathbf{R}$, i.e.,
		\begin{equation}\label{orthogonal subspace}
			 H^{-\frac{1}{2}}_{\mathbf{R}}(\partial D):=\left\{  \bm{\phi}\in H^{-\frac{1}{2}}(\partial D): \int_{\partial D_i} \bm{\phi}\cdot \mathbf{r}=0
			 \ \mathrm{for\ all}\ \mathbf{r}\in \mathbf{R } \ \mathrm{and}\ 1\leq i\leq N\right\}.
	\end{equation}
	The integral in the line above actually means the $H^{\frac12}$-$H^{-\frac12}$ pairing. Similarly, $H_{\mathbf{R}}^{\frac{1}{2}}(\partial D)$ is the subspace of $H^{\frac{1}{2}}(\partial D)$ orthogonal to $\mathbf{R}$. The space $H^{\frac12}_\mathbf{R}(\partial \Omega)$ and $H^{-\frac12}(\partial \Omega)$ are similarly defined.

	We consider the following transmission problem
	\begin{equation}\label{eq:transmissionproblem}
		\left\{
		\begin{aligned}
			&		\mathcal{L}_{\lambda,\mu}\mathbf{u}=0 && \mathrm{in}\ \Omega\setminus \ol D,\\
			&		\mathcal{L}_{\widetilde{\lambda},\widetilde{\mu}}\mathbf{u}=0 && \mathrm{in}\  D,\\
			& \mathbf{u}|_-=\mathbf{u}|_+ \quad \text{and} \quad \left. \frac{\partial \mathbf{u}}{\partial \nu_{(\widetilde{\lambda},\widetilde{\mu})}}\right|_-=\left. \frac{\partial \mathbf{u}}{\partial \nu_{(\lambda,\mu)}}\right|_+ && \mathrm{on} \ \partial D,\\
			& \left. \frac{\partial \mathbf{u}}{\partial \nu_{(\lambda,\mu)}}\right|_{\partial \Omega}=\mathbf{g} \in H^{-\frac{1}{2}}_{\mathbf{R}}(\partial \Omega) \quad \text{and} \quad  \mathbf{u}|_{\partial \Omega}\in H^{\frac{1}{2}}_{\mathbf{R}}(\partial \Omega). &&
		\end{aligned}\right.
	\end{equation}
	Here and in the rest of the paper, the background Lam\'e pair is fixed to be $(\lambda,\mu)$, and the one inside the inclusions is set to be $(\widetilde\lambda,\widetilde\mu)$. Both are assumed to be admissible. 
	The symbol $\rvert_+$ means the trace is taken from the exterior part of $D$, i.e, from $\Omega\setminus \ol D$, and $\rvert_-$ means the opposite. It is standard to verify (due to the transmission boundary condition at $\partial D$ above) that this problem is equivalent to 
	\begin{equation*}
	\left\{
	\begin{aligned}
	&\mathcal{L}_{\lambda(x),\mu(x)}\mathbf{u} = \nabla \cdot [\lambda(x) (\mathrm{div}\, \mathbf{u})\mathbb{I}_d + 2\mu(x) \bD(\mathbf{u})] = 0 \quad \text{in } \Omega,\\
	&\left.\frac{\partial \mathbf{u}}{\partial \nu_{(\lambda,\mu)}}\right|_{\partial \Omega}=\mathbf{g} \in H^{-\frac{1}{2}}_{\mathbf{R}}(\partial \Omega), \quad  \mathbf{u}|_{\partial \Omega}\in H^{\frac{1}{2}}_{\mathbf{R}}(\partial \Omega).
	\end{aligned}
	\right.
	\end{equation*}
	Here, $\lambda(x)$ is a piecewise constant function defined by $\lambda \mathbf{1}_{\Omega\setminus D} + \widetilde \lambda\mathbf{1}_{D}$, and $\mu(x)$ is defined in a similar manner. The existence and uniqueness of a weak solution $\mathbf{u}\in H^1(\Omega)$ to this problem follow from a standard application of the Lax-Milgram theorem, thanks to the admissibility of the Lam\'{e} pair $(\lambda(x),\mu(x))$. Note that the restriction imposed on the Neumann data $\mathbf{g}$ at $\partial \Omega$ is a necessary compatibility condition, and the restriction of $\mathbf{u}\rvert_{\partial \Omega}$ is for uniqueness. 
	
We take the form \eqref{eq:transmissionproblem} when considering transmission problems, and use layer potentials to obtain explicit representations for the solutions. In this way, we take the most advantage of the fact that the Lam\'e pair $(\lambda(x),\mu(x))$ are piecewise constants, and those explicit formulas show that the solutions enjoy many finer properties than merely in $H^1(\Omega)$.

Our first main result is a layer potential representation of the solution $\mathbf{u}$ to the transmission problem \eqref{eq:transmissionproblem}. The single-layer potential $\mathcal{S}^{\lambda,\mu}_D$, its boundary trace $\mathbb{S}^{\lambda,\mu}_D$ and the Neumann-Poincar\'{e} operator $\mathbb{K}^{\lambda,\mu,*}_D$ are defined in the next section; see \eqref{eq:singleS}, \eqref{eq:bSD} and \eqref{eq:KDstar}. Throughout the paper $\mathbf{G}\in H^1(\Omega)$ is the background solution defined by
		\begin{equation}\label{background solution}
			\mathcal{L}_{\lambda,\mu}\mathbf{G}=0 \quad \mathrm{in }\ \Omega, \quad \left. \frac{\partial \mathbf{G}}{\partial \nu_{(\lambda,\mu)}}\right|_{\partial \Omega}=\mathbf{g}, \quad \mathbf{G}|_{\partial \Omega}\in H^{\frac{1}{2}}_{\mathbf{R}}(\partial \Omega).
		\end{equation}
As it will be clear, the role of $\mathbf{G}$ is to lift the boundary condition in \eqref{eq:transmissionproblem} at $\partial \Omega$.

	\begin{theorem}\label{thm:reptrans}
		Assume {\upshape (A1)} holds. Then the unique solution $\mathbf{u}$ to \eqref{eq:transmissionproblem} is given by the single-layer potential:
		\begin{equation}
			\mathbf{u}=\left\{
			\begin{aligned}
				&\mathbf{G}+\mathcal{S}^{\lambda,\mu}_{D} \text{\boldmath $\phi$} &&\mathrm{in}\   \Omega\setminus \ol D \\
				&\mathcal{S}^{\widetilde{\lambda},\widetilde{\mu}}_{D}  \text{\boldmath $\psi$} &&\mathrm{in}\  D
			\end{aligned}	\right., 
		\end{equation}
		where $(\text{\boldmath $\psi$}, \text{\boldmath $\phi$})\in H^{-1/2}(\partial D)\times H_{\mathbf{R}}^{-1/2}(\partial D)$ solve the
		boundary integral equations 
			\begin{equation}\label{boundary integral equations}
				\left\{
			\begin{aligned}
				& \mathbf{G}|_{\partial D}+\mathbb{S}^{\lambda,\mu}_{D}\bm{\phi}=\mathbb{S}^{\widetilde{\lambda},\widetilde{\mu}}_{D}\bm{\psi}, \\
				&\left(-\frac{\mathbb{I}}{2}+\mathbb{K}^{\widetilde{\lambda},\widetilde{\mu},*}_{D}\right)\bm{\psi} - \left(\frac{\mathbb{I}}{2}+\mathbb{K}^{\lambda,\mu,*}_{D}\right)\bm{\phi} =\left. \frac{\partial \mathbf{G}}{\partial \nu_{(\lambda,\mu)}}\right|_{\partial D}.
			\end{aligned}\right.
		\end{equation}
		Moreover, the pair $(\bm{\phi},\bm{\psi})$ exists and is unique.
	\end{theorem}
	
	Layer potential representations for transmission problems are certainly not new. For instance, Escauriaza and Seo \cite{escauriaza1993regularity} obtained a formula for $\mathbf{u}$ 
			using classical single-layer potential; see also \cite{ammari2008method,ammari2007polarization,ammari2013strong}. 
			The extra assumption
			\begin{equation*}
				(\lambda -\widetilde{\lambda})(\mu-\widetilde{\mu})\geq 0 \quad \mathrm{and }\quad 0<\widetilde{\lambda},\widetilde{\mu}<\infty,
			\end{equation*}
			however, was needed in \cite{escauriaza1993regularity}. Our method is based on layer potentials using the so-called Neumann functions rather than the fundamental solution in $\R^d$, and on layer potential representation of the Dirichlet to Neumann maps, and the extra assumption above is not needed.
	
For isotropic elastic materials, the physical significance of the Lam\'{e} pair $(\widetilde{\lambda},\widetilde{\mu})$ is as follows: $\widetilde{\mu}$ is the shear modulus which measures shear stiffness; the Young's modulus $\widetilde E$ and the bulk modulus $\widetilde K$ are related to $(\widetilde{\lambda},\widetilde{\mu})$ by 
	\begin{equation}\label{lame_constant}
		\widetilde E= \frac{2\widetilde{\mu}(d\widetilde{\lambda} + 2\widetilde{\mu})}{(d-1)\widetilde{\lambda}+2\widetilde{\mu}}  \quad \mathrm{and}\quad
		\widetilde K= \frac{d\widetilde{\lambda}+2\widetilde{\mu}}{d}.
	\end{equation}
	They measure, respectively, the stiffness with respect to uniaxial stress and the compressibility under ambient pressure (see \cite{steinbach2007numerical}). We are interested in the following asymptotic settings when one or both of $\widetilde\lambda,\widetilde\mu$ go to the extreme values, namely, $\infty$ and $0$:

\noindent\emph{Case 1: The incompressible inclusions limit, as $\widetilde{\lambda} \rightarrow \infty$ and $\widetilde{\mu}$ fixed}. The bulk mudulus $\widetilde K$ of the inclusions hence tends to infinity, so the inclusions become incompressible. A typical situation is when the inclusions behave like \emph{rubber}, able to change their shapes but not volumes.

\noindent\emph{Case 2: The soft inclusions limit, as $\widetilde{\lambda}$ and $\widetilde{\mu}$ both tend to zero}. Then $\widetilde E$, $\widetilde K$ and $\widetilde \mu$ of the inclusions all vanish in the limit, so the inclusions are soft in the sense that they can barely hold any stress. 

\noindent\emph{Case 3: The hard inclusions limit, as $\widetilde{\mu}\rightarrow \infty$ and $\widetilde{\lambda}$ fixed}. Then $\widetilde E$, $\widetilde K$ and $\widetilde \mu$ of the inclusions all tend to infinity. The inclusions behave like rigid bodies, only able to move or rotate as a whole while keeping their shapes.

	The limiting behaviors of those elastic inclusions in the above asymptotic regimes are easy to figure out \emph{intuitively}. It turns out that, for \emph{Case 1}, the limit of the transmission problem is a coupled Lam\'{e}-Stokes system:
	\begin{equation}\label{eq:limStokes}
		\left\{
		\begin{aligned}
			&		\mathcal{L}_{\lambda,\mu}\mathbf{u}=0  && \mathrm{in}\ \Omega\setminus \ol D,\\
			&	\mathcal{L}_{\infty,\widetilde{\mu}}(\mathbf{u},p)=0 \quad \text{and} \quad \mathrm{div}\,\mathbf{u}=0 \qquad && \mathrm{in}\  D,\\
			& \mathbf{u}|_-=\mathbf{u}|_+ \quad \text{and} \quad \left. \frac{\partial (\mathbf{u},p)}{\partial \nu_{(\infty,\widetilde{\mu})}}\right|_-=\left. \frac{\partial \mathbf{u}}{\partial \nu_{(\lambda,\mu)}}\right|_+  \quad && \mathrm{on} \ \partial D,\\
			& \left. \frac{\partial \mathbf{u}}{\partial \nu_{(\lambda,\mu)}}\right|_{\partial \Omega}=\mathbf{g} \in H^{-\frac{1}{2}}_{\mathbf{R}}(\partial \Omega) \quad \text{and} \quad \mathbf{u}|_{\partial \Omega}\in H^{\frac{1}{2}}_{\mathbf{R}}(\partial \Omega), \qquad &&
		\end{aligned}\right.
	\end{equation}
where $\mathcal{L}_{\infty,\widetilde{\mu}} (\mathbf{u},p)=\widetilde{\mu}\Delta \mathbf{u}+\nabla p$ denotes the Stokes operator with viscosity constant $\widetilde{\mu}$ and $p$ is the pressure field. 

In \emph{Case 2}, the limit problem (one then only cares about the part in $\Omega\setminus \ol D$) is the Neumann boundary value problem with soft (e.g.\,vacuum) inclusions, i.e.,
\begin{equation}\label{eq:limZero}
	\left\{
	\begin{aligned}
		&		\mathcal{L}_{\lambda,\mu}\mathbf{u}=0 && \mathrm{in}\ \Omega\setminus \ol D,\\
		&\left. \frac{\partial \mathbf{u}}{\partial \nu_{(\lambda,\mu)}}\right|_+=0 && \mathrm{on} \ \partial D, \\
		& \left. \frac{\partial \mathbf{u}}{\partial \nu_{(\lambda,\mu)}}\right|_{\partial \Omega}=\mathbf{g} \in H^{-\frac{1}{2}}_{\mathbf{R}}(\partial \Omega)\quad \text{and} \quad \mathbf{u}|_{\partial \Omega}\in H^{\frac{1}{2}}_{\mathbf{R}}(\partial \Omega). &&
	\end{aligned}\right.
\end{equation}
The inclusions now play the role of perforated \emph{holes} where no stress is imposed.

In \emph{Case 3}, the limit problem is the Neumann boundary value problem with rigid inclusions, i.e.,
\begin{equation}\label{eq:limRigid}
	\left\{
	\begin{aligned}
		&\mathcal{L}_{\lambda,\mu} \mathbf{u}=0 \  \text{in } \, \Omega\setminus \ol D \quad \text{and} \quad \mathbf{u}  \in \mathbf{R} \quad \text{in each component of $D$},\\
				& \mathbf{u}|_-=\mathbf{u}|_+ \quad \text{and} \quad \left.\frac{\partial \mathbf{u}}{\partial \nu_{(\lambda,\mu)}} \right|_{+} \in H^{-\frac{1}{2}}_{\mathbf{R}}(\partial D) \quad   \text{on}\ \partial D, \\
		& \left. \frac{\partial \mathbf{u}}{\partial \nu_{(\lambda,\mu)}}\right|_{\partial \Omega}=\mathbf{g}\in H^{-\frac{1}{2}}_{\mathbf{R}}(\partial \Omega) \quad \text{and} \quad \mathbf{u}|_{\partial \Omega} \in H^{\frac{1}{2}}_{\mathbf{R}}(\partial \Omega).
	\end{aligned}\right.
\end{equation}
The asymptotic models in the three cases above are more or less standard, and there are quite a few mathematical studies for Cases 2 and 3; see for instance \cite{MR1195131,MR3296149,Jing-Neumann,MR548785}. Case 1 was less studied, but see related work in \cite{greenleaf2003selected,MR2399553}. Although those models are natural to establish purely using physics, it is very natural to ask the following mathematical question:

\begin{question}\emph{Are \eqref{eq:limStokes}, \eqref{eq:limZero} and \eqref{eq:limRigid} the mathematical limits, in proper senses, of the transmission problem \eqref{eq:transmissionproblem} in the corresponding asymptotic settings?}
\end{question}

The answer is of course affirmative. In fact, Ammari \emph{et.\,al.\;}confirmed this for all three cases in \cite{ammari2013strong,ammari2008method} (for Case 1, they considered the setting where $\lambda,\widetilde\lambda$ are sent to $\infty$ at the same time) using layer potential methods. For Case 3, this was proved also by Bao \emph{et.\,al.\;}in \cite{bao2009gradient} using variational method. The approach by Ammari \emph{et.\,al.\;}also yields convergence rates, of order $O(\widetilde{\lambda}^{-\frac{1}{2}})$, $O((\widetilde{\lambda}+\widetilde{\mu})^{\frac{1}{4}})$ and $O(\widetilde{\mu}^{-\frac{1}{2}})$ for Cases 1, 2 and 3. Among many other things, a recent work \cite{MR4242872} by Craster \emph{et.\,al.\,}studied those convergence rates numerically and showed evidence that the convergence rates obtained in \cite{ammari2013strong,ammari2008method} could be improved. Note also, all of those previous studies are essentially for a fixed number of inclusions, as in our assumption (A1).  

In this paper, we aim to reproduce those convergence results and obtain convergence rates for the three asymptotic settings above. Moreover, we can treat the case when $D = D_\eps$ is a periodic array of small inclusions with periodicity $\eps > 0$ which is another small parameter, as specified in the geometric setup (A2). We establish sharper convergence rates and prove the important fact that the convergence rates are uniform in $\eps$.

Before stating those results, let us confirm that, just as in Theorem \ref{thm:reptrans}, layer potential representation of the solutions to the limit problems \eqref{eq:limStokes}, \eqref{eq:limZero} and \eqref{eq:limRigid} are all available. We state the result for the coupled Lam\'e-Stokes system only; the other two cases are more standard. The layer potential operators $\cS^{\infty,\widetilde\mu}_D$ and $\bK^{\infty,\widetilde\mu,*}_D$ are defined in the next section.

\begin{theorem}\label{thm:limStokes}
  Assume {\upshape(A1)} holds. The problem \eqref{eq:limStokes} has a unique pair of solutions $(\mathbf{u}^{\infty,\widetilde{\mu}},p)$ ($p$ modulus a constant), and $\mathbf{u}^{\infty,\widetilde{\mu}}$ is given by the single-layer potential formula:
	\begin{equation}
		\mathbf{u}^{\infty,\widetilde{\mu}}=\left\{
		\begin{aligned}
			&\mathbf{G}+\mathcal{S}^{\lambda,\mu}_{D} \text{\boldmath $\phi$} &\mathrm{in}\   \Omega\setminus \ol D  \\
			&\mathcal{S}^{\infty,\widetilde{\mu}}_D  \text{\boldmath $\psi$}&\mathrm{in}\  D.
		\end{aligned}	\right.
	\end{equation}
	Here, $(\text{\boldmath $\psi$}, \text{\boldmath $\phi$})\in H^{-1/2}(\partial D)\times H_{\mathbf{R}}^{-1/2}(\partial D)$ solve the boundary integral equations 
	\begin{equation}\label{boundary integral equations for stokes system}
		\left\{
		\begin{aligned}
			& \mathbf{G}|_{\partial D}+\mathbb{S}^{\lambda,\mu}_{D}\bm{\phi}=\mathbb{S}^{\infty,\widetilde{\mu}}_{D}\bm{\psi}, \\
			&\left(-\frac{\mathbb{I}}{2}+\mathbb{K}^{\infty,\widetilde{\mu},*}_{D}\right)\bm{\psi} - \left(\frac{\mathbb{I}}{2}+\mathbb{K}^{\lambda,\mu,*}_{D}\right)\bm{\phi} =\left. \frac{\partial \mathbf{G}}{\partial \nu_{(\lambda,\mu)}}\right|_{\partial D}.
		\end{aligned}\right.
	\end{equation}
	Moreover, the system above has a unique solution.
\end{theorem}

The proofs of Theorems \ref{thm:reptrans} and \ref{thm:limStokes} are presented in the next section as standard applications of layer potential theory. We emphasize that that those results hold as long as the inclusions set $D$ satisfies the condition (A1), which contains the situation of (A2), for each fixed $\eps \in (0,1)$.   

For the next main theorem, we consider the geometric setup (A2) and hence treat a family of inclusions set $D = D_\eps$, $\eps \in (0,1)$. We prove not only the transmission problem \eqref{eq:transmissionproblem} converges to the corresponding limits in the three asymptotic settings, but also the convergence rates can be made uniform in $\eps$. To simplify notations, we fix a positive number $\delta > 0$, and say a pair of Lam\'e coefficients $(\lambda',\mu')$ is \emph{uniformly admissible} if
\begin{equation}
  \label{eq:unifadm}
  \delta \le \min\{d\lambda'+2\mu',\mu'\}, \qquad \max\{d\lambda'+2\mu',\mu'\} \le \delta^{-1}.
\end{equation}
We say a bounding constant $C >0$ in an estimate is \emph{universal} if it depends on the data $d,\Omega,\omega,\delta$ and the background Lam\'e pair $(\lambda,\mu)$, but is independent of $\eps$ or other asymptotic parameters, namely, one or both of $\widetilde\lambda,\widetilde\mu$ in each of the asymptotic settings. 

\begin{theorem}\label{main result}
  Assume that {\upshape(A2)} holds; assume further for Case 3 that, as $\widetilde\mu$ goes to infinity, the rescaled pair $(\widetilde\lambda/\widetilde\mu,1)$ is uniformly admissible (this is the case for sufficiently large $\widetilde\mu$ if $\widetilde\lambda$ is fixed). In each asymptotic setting, let $\mathbf{u}$ be the solution to \eqref{eq:transmissionproblem}, and $\mathbf{u}_{\lim}$ be the solution to the limit problems \eqref{eq:limStokes}, \eqref{eq:limZero} or \eqref{eq:limRigid}; then there exists a universal constant $C > 0$ such that 
  \begin{equation}\|\mathbf{u}-\mathbf{u}_{\lim}\|_{H^1(\Omega_{\varepsilon})} \leq \left\{
		\begin{aligned}
		& \frac{C}{\widetilde{\lambda}} \|\mathbf{g}\|_{H^{-\frac{1}{2}}(\partial \Omega)}  & \text{in Case 1},\\
		& C(|\widetilde{\lambda}|+\widetilde{\mu})\|\mathbf{g}\|_{H^{-\frac{1}{2}}(\partial \Omega)}  & \text{in Case 2}, \\
		&\frac{C}{\widetilde{\mu}} \|\mathbf{g}\|_{H^{-\frac{1}{2}}(\partial \Omega)}  & \text{in Case 3}.
		\end{aligned}\right.
	\end{equation}
\end{theorem}
As mentioned earlier, the convergence results in Theorem \ref{main result} improve those obtained by Ammari \emph{et.\,al.}\,in \cite{ammari2008method,ammari2013strong}. In fact, our convergence rates, in Case 2 and 3, match very well with the numerical computations carried out by Craster \emph{et.\,al.\,}in \cite[Section 6]{MR4242872}). This theorem is proved in Section \ref{Strong convergence to the extreme}, and the key ingredient is the establishment of the following uniform spectral gaps for the Neumann-Poincar\'{e} (NP) operator:
\begin{theorem}\label{thm:specgap}
  Assume that {\upshape(A2)} holds. Then there exists a universal constant $\delta_1\in (0,1)$ so that, for any $\varepsilon>0$, the spectrum of the Neumann-Poincar\'{e} operator 
		$\mathbb{K}_{D_{\varepsilon}}^{\lambda,\mu,*}:H^{-\frac{1}{2}}_{\mathbf{R}}(\partial D_{\varepsilon}) \rightarrow H^{-\frac{1}{2}}_{\mathbf{R}}(\partial D_{\varepsilon})$, defined in \eqref{eq:KDstar}, 
	 is contained in $(-\frac{1}{2}+\delta_1 , \frac{1}{2}-\delta_1)$.
\end{theorem}

Theorem \ref{thm:specgap} is proved in Section \ref{sec:unifesti}, and our proof is partially inspired by the work of \cite{MR2308861,bonnetier2019homogenization,bunoiu2020homogenization}, where some uniform estimates for the spectra of the NP operator associated to electrostatic and electromagnetic systems were proved. A salient feature of the elastostatic setting considered in this paper is that the operator $\bK^{\lambda,\mu,*}_D$ is non-compact even for smooth $\partial D$. Previous works in \cite{dahlberg1988boundary,MR3769919,MR3973113} have revealed nice properties of the elastostatic NP operator; in particular, for a fixed domain $D$ satisfying (A1), results in \cite{MR3769919,MR3973113} showed the spectra of $\bK^{\lambda,\mu,*}$ consist of (countably many) eigenvalues and their accumulate points $\pm \mu/2(\lambda+2\mu)$; in particular, there are gaps between the spectra and the points $\{-\frac12,\frac12\}$. While those results are for a fixed domain, our result in Theorem \ref{thm:specgap} shows that those gaps are uniform in $\eps$, when $D_\eps$ is a family of $\eps$-periodic array of small sets.

Before concluding this introduction, let us comment first that imposing Neumann boundary condition at $\partial \Omega$ in \eqref{eq:transmissionproblem} is just to fix notations. The method of this paper works equally well if Dirichlet data are imposed at $\partial \Omega$. In fact, Theorem \ref{thm:specgap} (and Theorem \ref{unif m M}) remain the key steps to prove this fact. Finally, we put the present work in the following framework concerning homogenization with high contrast inclusions, which is illustrated in the diagram:
\begin{equation}
\begin{CD}
u^{\eps,\kappa} @>{\kappa \to \kappa_\infty}>> u^{\eps}\\
@V{\eps\to 0}VV  @VV{\eps\to 0}V\\
u^{\kappa} @>{\kappa \to \kappa_\infty}>> u
\end{CD}
\end{equation}
Here, $\eps > 0$ stands for the periodicity of the geometric setup of the high contrast inclusions $D_\eps$ and is sent to zero, $\kappa$ stands for the high contrast parameter of the inclusions and is sent to some extreme value $\kappa_\infty$. More generally, both $\eps$ and $\kappa$ (like considered in this paper) could be a set of parameters. We expect that the convergence results implied by the four arrows all hold, and the limit function in each arrow can be characterized. In this paper, we studied the upper arrow and showed that the limit $u^\eps$ is modeled by equations in perforated domains or that coupled with equations posed inside the inclusions, and we obtained uniform convergence rates. The downward arrow on the right, hence, corresponds to homogenization in perforated domains, and there are a lot of studies in the literature; see e.g.\,\cite{MR4075336,Jing-Neumann,MR4240768}. The downward arrow on the left corresponds to homogenization of high contrast inclusions before taking the limit $\kappa \to \kappa_\infty$, and it seems a more difficult task to establish uniform in $\kappa$ quantitative homogenization results. We aim to address this problem in future works; see \cite{MR4267502}, however, for related considerations. The lower arrow in the diagram is somehow easier and can be treated by the method of this paper; and this convergence can be viewed as a continuity property of the homogenized coefficients with respect to the high contrast parameter $\kappa$.

The rest of the paper is organized as follows. In section \ref{Layer potentials for Lame system}, we review the layer potential theory for Lam\'{e} system and for Stokes system, and study some properties of the corresponding Dirichlet to Neumann maps. In section \ref{Single layer potential representations for the solutions}, we use those theories to derive  single-layer potential representations for solutions to the transmission problem, and to the limit models. In section \ref{sec:unifesti}, we study the spectral gaps of elastostatic Neumann-Poincar\'{e} operators associated to the periodic array $D_\eps$ and establish some uniform boundedness and invertibility of related operators. Those results are used in section \ref{Strong convergence to the extreme} to prove the uniform convergence rates of Theorem \ref{main result}.

\subsection*{Notations} For two $m\times m$ matrices $A,B$, the Frobenius product of them is denoted by $A:B$, and $|A|$ denotes the Frobenius norm $(A:A)^{\frac12}$. When $A$ is symmetric, the basic trace identity $(\mathrm{tr}\,A)^2 \le m|A|^2$ will be used. Suppose $H$ is a Hilbert space and $H^*$ is the dual space. The inner product on $H$ is denoted by $(\cdot,\cdot)_{H}$, and the $H$-$H^*$ pairing, e.g.\;for $f\in H$ and $\varphi \in H^*$, is denoted by $\langle f,\varphi\rangle_{H,H^*} = \varphi(f)$; the subscript is omitted when doing so causes no confusion. The notation $\mathscr{L}(H)$ stands for the space of bounded linear transformations of $H$.

\section{Preliminaries on layer potential theory}\label{Layer potentials for Lame system}
	
In this section, we review the layer potential theories for the Lam\'e system and for the Stokes system in $\Omega$ associated to (the surface of) a subset $D\subset \Omega$. To fix ideas, assume $\Omega$ and $D$ satisfy the geometric setup (A1). 


Let $(\lambda,\mu)$ be a fixed constant Lam\'e pair. The bilinear energy form for the Lam\'{e} system \eqref{Lame system} on a domain $E$ (usually taken as $\Omega$, $\Omega\setminus \ol D$ or components of $D$) is defined, for $\mathbf{u},\mathbf{v}\in H^1(E)$, by
	\begin{equation}
	  \label{eq:JDlammu}
		J^E_{\lambda,\mu}(\mathbf{u},\mathbf{v}):=\lambda \int_E (\mathrm{div}\,\mathbf{u})(\mathrm{div}\,\mathbf{v}) + 
		2 \mu \int_E \bD(\mathbf{u}) : \bD(\mathbf{v}).
	\end{equation}
The Green's identity with conormal derivative \eqref{conormal derivative} then reads
	\begin{equation}\label{eq:GreenLame}
		J^E_{\lambda,\mu}(\mathbf{u},\mathbf{v}) = - \int_E \mathbf{v}\cdot (\mathcal{L}_{\lambda,\mu}\mathbf{u}) 
		+ \int_{\partial E} 
		\frac{\partial \mathbf{u}}{\partial \nu_{(\lambda,\mu)}} \cdot   \mathbf{v}.
	\end{equation}
Clearly, $J^E_{\lambda,\mu}(\cdot,\cdot)$ is symmetric in $\mathbf{u}$ and $\mathbf{v}$ so their roles can be exchanged in the above identity. Moreover, if one of $\{\mathbf{u},\mathbf{v}\}$ solves the homogeneous Lam\'e system, the formula above simplifies to boundary integrals only. For convenience,  we denote $J^E_{\lambda,\mu}(\mathbf{u},\mathbf{u})$ by $J^E_{\lambda,\mu}(\mathbf{u})$. 

\begin{remark}
\label{rem:equinorm}
Due to the trace inequality of symmetric matrices, we have
\begin{equation}
\label{eq:Jadm}
\min\{d\lambda + 2\mu, 2\mu\}\int_E \bD(\mathbf{u}):\bD(\mathbf{u}) \le J^E_{\lambda,\mu}(\mathbf{u}) \le \max\{d\lambda + 2\mu, 2\mu\}\int_E \bD(\mathbf{u}):\bD(\mathbf{u}).
\end{equation}
It follows that for admissible $(\lambda,\mu)$, $\|\mathbf{u}\|_{J} := (J^E_{\lambda,\mu}(\mathbf{u}))^{\frac12}$ is a seminorm in $H^1(\Omega)$ and only functions in $\mathbf{R}$ are eliminated by it. By the Korn's inequality, if $H \subset H^1(E)$ is a subspace satisfying $H\cap \mathbf{R} = \{0\}$, then $\|\cdot\|_J$ is a norm equivalent to the standard $H^1$ norm; furthermore, if $(\lambda,\mu)$ is uniformly admissible, then the two norms bound each other by universal bounding constants.
\end{remark}

We also have the useful \emph{Dirichlet principle} stated below. 
\begin{lemma}\label{lem:diriprin}
Suppose that $(\lambda',\mu')$ is an admissible Lam\'e pair and $D$ is a domain with Lipschitz boundary. Suppose that $\mathbf{u} \in H^1(D)$ satisfies $\mathcal{L}_{\lambda',\mu'}(\mathbf{u}) = 0$. Then, for any $\mathbf{v} \in H^1(D)$ such that $\mathbf{u} = \mathbf{v}$ on $\partial D$, it holds
\begin{equation}
\label{eq:diriprin}
J^D_{\lambda',\mu'}(\mathbf{u}) \le J^D_{\lambda',\mu'}(\mathbf{v}).
\end{equation}
\end{lemma}
\begin{proof} Apply the Green's identity \eqref{eq:GreenLame}, we get
\begin{equation*}
J^{D}_{\lambda',\mu'}(\mathbf{u},\mathbf{u}-\mathbf{v}) = 0, \qquad \text{i.e.,}\quad J^{D}_{\lambda',\mu'}(\mathbf{u}) = J^{D}_{\lambda',\mu'}(\mathbf{u},\mathbf{v}). 
\end{equation*}
By the Cauchy-Schwarz inequality, we get \eqref{eq:diriprin}.
\end{proof}

	\subsection{Layer potential for Lam\'{e} system} \label{Prelim lame}
	We review some well-known layer potential theory of Lam\'{e} system, with more details referred to \cite{kupradze2012three,ammari2007polarization,MR3769919,MR4172687}. The Kelvin matrix of fundamental solution for Lam\'{e} system $\mathcal{L}_{\lambda,\mu}$ is the $d \times d$ matrix $\bm{\Gamma}(x,z)$, for $x\neq z$, with entries
	\begin{equation*}\Gamma_{ij}(x,z)=\left\{
		\begin{aligned}
			& -\frac{c_1}{4\pi }\frac{\delta_{ij}}{|x-z|}-\frac{c_2}{4\pi} \frac{(x_i-z_i)(x_j-z_j)}{|x-z|^3},\quad  & d=3, \\
			& \frac{c_1}{2\pi} \delta_{ij}\log |x-z|-\frac{c_2}{2\pi} \frac{(x_i-z_i)(x_j-z_j)}{|x-z|^2},\quad & d=2.
		\end{aligned}\right.
	\end{equation*}
	which satisfies $\mathcal{L}_{\lambda,\mu}\bm{\Gamma}(\cdot,z)=\delta_z(\cdot) \mathbb{I}_d$. Here, $\delta_z$ is the Dirac mass centered at $z\in \R^d$, and $c_1,c_2$ are two constants defined by
	\begin{equation*}
		c_1 =\frac{1}{2}\left( \frac{1}{\mu}+\frac{1}{2\mu+\lambda} \right)\quad \mathrm{and}\quad c_2=\frac{1}{2}\left( \frac{1}{\mu}-\frac{1}{2\mu+\lambda} \right).
	\end{equation*}
We will use the so-called Neumann function for 
the Lam\'{e} system in $\Omega$. For each $z\in \Omega$, it is the unique solution $\text{\boldmath $\Gamma$}^{\mathrm{N}}(\cdot,z)$ to the following problem: 
\begin{equation}
\label{eq:GammaN}
\left\{
	\begin{aligned}
		& (\mathcal{L}_{\lambda,\mu})_x\,\text{\boldmath $\Gamma$}^{\mathrm{N}}(x,z)=\delta_z(x)\mathbb{I}_d  & & \mathrm{in}\  \Omega, \\
		&\left. \frac{\partial \text{\boldmath $\Gamma$}^{\mathrm{N}}}{\partial \nu_x}\right|_{\partial \Omega}=\frac{1}{|\partial \Omega|}, \ \int_{\partial \Omega} \text{\boldmath $\Gamma$}^{\mathrm{N}}(x,z)\mathbf{r}(x)\,d\sigma(x)=0 & & \mathrm{for} \ z\in \Omega\ \mathrm{and\ all}\ \mathbf{r}\in \mathbf{R}.
	\end{aligned}\right.
\end{equation}
We emphasize that the Neumann function $\text{\boldmath $\Gamma$}^{\mathrm{N}}(x,z)$ is defined as a function of $x\in \overline{\Omega}$ for each fixed $z\in \Omega$. The superscript `$\mathrm{N}$' is to emphasize ``Neumann" condition is imposed at $\partial \Omega$, and is to highlight the contrast with the so-called Dirichlet function ({\it i.e.}, the Poisson kernel) $\text{\boldmath $\Gamma$}^{\mathrm{D}}(x,z)$. The latter, for each $z\in \Omega$, is the unique solution to
\begin{equation}
\label{eq:GammaD}
\left\{
	\begin{aligned}
		& (\mathcal{L}_{\lambda,\mu})_x\,\text{\boldmath $\Gamma$}^{\mathrm{D}}(x,z)=\delta_z(x)\mathbb{I}_d  \quad \mathrm{in}\  \Omega, \\
		&\bm{\Gamma}^{\mathrm{D}}|_{\partial \Omega}=0.
	\end{aligned}\right.
\end{equation}
Comparing the definitions of the Neumann and Dirichlet functions with the Kelvin matrix, we check that $\text{\boldmath $\Gamma$}^{\mathrm{N}}$ and $\text{\boldmath $\Gamma$}^{\mathrm{D}}$ are \emph{perturbations} of $\text{\boldmath $\Gamma$}$ in the sense that
\begin{equation}
\label{eq:GammaNpert}
\text{\boldmath $\Gamma$}^{\mathrm{N}}(\cdot,z) = \text{\boldmath $\Gamma$}(\cdot-z) + R^{\mathrm{N}}(\cdot,z), \qquad \text{\boldmath $\Gamma$}^{\mathrm{D}}(\cdot,z) = \text{\boldmath $\Gamma$}(\cdot-z) + R^{\mathrm{D}}(\cdot,z),
\end{equation}
where $R^{\mathrm{N}}$ and $R^{\mathrm{D}}$ solve the Lam\'e systems in $\Omega$ with Neumann, respectively, Dirichlet boundary data at $\partial \Omega$, and they are regular functions. 

For most part of the paper, we use single-layer potentials associated to the Neumann function $\text{\boldmath $\Gamma$}^{\mathrm{N}}$: given a \emph{density} (or, \emph{moment}) function $\bm{\phi}$ on $\partial D$, the single-layer potential $\mathcal{S}_D\bm{\phi}$  is 
\begin{equation}
\label{eq:singleS}
	\mathcal{S}_D\bm{\phi}(x):=\int_{\partial D} \bm{\Gamma}^{\mathrm{N}}(x,z)\bm{\phi}(z)\,d\sigma(z)\quad \mathrm{for}\ x \in \Omega \setminus \partial D.
\end{equation}
The boundary trace $\mathbb{S}_D\text{\boldmath $\phi$}$ on $\partial D $ is defined as the non-tangential limit
\begin{equation}\label{eq:bSD}
	\mathbb{S}_D\text{\boldmath $\phi$}(x):=\lim\limits_{y\rightarrow x\atop y\in C(x)}\mathcal{S}_D\text{\boldmath $\phi$}(y)\quad \mathrm{for}\ x\in \partial D,
\end{equation}
where $C(x)$ is the non-tangential cone at $x$, see \cite{10.1007/BF02545747}. The Neumann-Poincar\'{e} operator $\mathbb{K}_D^*$  is defined by 
\begin{equation}\label{eq:KDstar}
	\mathbb{K}_D^*\text{\boldmath $\phi$}(x):=\mathrm{p.v.}\int_{\partial D}  \frac{\partial \bm{\Gamma}^{\mathrm{N}}}{\partial_x \nu}(x,z)\text{\boldmath $\phi$}(z)\,d\sigma(z)\quad \mathrm{for} \ x \in  \partial D,
\end{equation}
where $\mathrm{p.v.}$ stands for the Cauchy principal value. Note that, we have omitted the reference to the parameters $(\lambda,\mu)$ in related operators as they are fixed and can be read from the context; we continue using this simplification below.

\smallskip

The following well known results make single-layer potentials extremely useful for solving Lam\'e systems in the setting of this paper. 
\begin{proposition}\label{prop:SDclassical}
Under the assumption {\upshape(A1)} for the domains $\Omega$ and $D$, the following results hold. 
\begin{enumerate}
	\item [{\upshape(i)}] For any $\bm{\phi} \in H^{-\frac{1}{2}}(\partial D)$, $
	\mathcal{S}_D\bm{\phi} \in H^1(\Omega)$ and
	\begin{equation}\label{eq:normalpO}
		\left\{
		\begin{aligned}
			& \mathcal{L}_{\lambda,\mu}\mathcal{S}_D\bm{\phi}=0 \quad \mathrm{in}\ \Omega\setminus \partial D, \\
			& \mathcal{S}_D\bm{\phi}|_{\partial \Omega} \in H^{\frac{1}{2}}_{\mathbf{R}}(\partial \Omega), \\
			& \left.\frac{\partial \mathcal{S}_D\bm{\phi}}{\partial \nu}\right|_{\partial \Omega} = \frac{1}{|\partial \Omega|}\int_{\partial D}\bm{\phi}.
		\end{aligned}\right.
	\end{equation}
	In particular, if $\bm{\phi} \in H_{\mathbf{R}}^{-\frac{1}{2}}(\partial D)$, then $ \left.\frac{\partial \mathcal{S}_D\bm{\phi}}{\partial \nu}\right|_{\partial \Omega} =0$.
	\item [{\upshape(ii)}]  $\mathbb{K}^*_D:H^{-\frac{1}{2}}(\partial D) \rightarrow H^{-\frac{1}{2}}(\partial D)$ is a bounded operator, furthermore, $\pm \frac{\mathbb{I}}{2}+\mathbb{K}^*_D:H_{\mathbf{R}}^{-\frac{1}{2}}(\partial D) \rightarrow H_{\mathbf{R}}^{-\frac{1}{2}}(\partial D)$ are isomorphisms. Here and through the paper $\mathbb{I}$ denotes the identity map.
	\item [{\upshape(iii)}] $\mathrm{(Jump\ relation)}$ For any $\bm{\phi}\in H^{-\frac{1}{2}}(\partial D)$,
	\begin{equation}\label{jump relation}
		\left. \frac{\partial \mathcal{S}_D\bm{\phi}}{\partial \nu}\right|_{D\pm }=\left( \pm \frac{\mathbb{I}}{2} +\mathbb{K}_D^* \right)\bm{\phi} ,
	\end{equation}
	here and through the paper the subscript $D_+$ and $D_-$ indicate the limit taken outside $D$ and inside $D$, respectively. 
	\item [{\upshape(iv)}]  $\mathbb{S}_D:H^{-\frac{1}{2}}(\partial D) \rightarrow H^{\frac{1}{2}}(\partial D)$ is an isomorphism.
	\item[(v)] The \emph{Calder\'on's identity} $\bK^{\mathrm{N}}\bS^{\mathrm{N}} = \bS^{\mathrm{N}} \bK^{\mathrm{N},*}$ holds.
\end{enumerate}
\end{proposition}

We refer to \cite{MR3769919,dahlberg1988boundary} for the proofs of those results. In plain words, the above says, the single-layer potential $\mathcal{S}_D \bm{\phi}$ solves the homogeneous Lam\'e system in $D$ and in $\Omega\setminus \ol D$, with certain normalization condition at the exterior boundary $\partial \Omega$ as shown in \eqref{eq:normalpO}; $\mathcal{S}_D \bm{\phi}$ is continuous across $\partial D$, but $\partial \mathcal{S}_D\bm{\phi}/\partial \nu$
has a jump across $\partial D$, and the jump is represented through the Neumann-Poincar\'e operator acting on $\bm{\phi}$. Moreover, in view of item (iii) (respectively, item (iv)), the Neumann (respectively, Dirichlet) boundary value problems in $D$ and in $\Omega\setminus \ol D$ can be solved by single-layer potentials. The Calder\'on identity is equivalent to
\begin{equation*}
\mathcal{D}_D [(\cS_D\bm\phi)\rvert_-](x) = \cS_D[(\textstyle\frac{\partial \cS_D\bm\phi}{\partial \nu}\rvert_-)](x), \qquad x\in \Omega\setminus \ol D,
\end{equation*}
where $\mathcal{D}_D$ is the double-layer potential with Schwartz kernel $\bm\Gamma^{\rm N}$. The identity above basically follows from the Green's identity and we omit the details. 

\begin{remark}
	In this paper we use the Neumann function $\bm{\Gamma}^{\mathrm{N}}$ rather than the more frequently used Kelvin matrix $\bm{\Gamma}$ as the Schwartz kernel to define the single-layer potential. The resulted $\mathcal{S}_D$ is hence different from the classical layer potential operators, for which results in Proposition \ref{prop:SDclassical} were standard. However, due to the relation \eqref{eq:GammaNpert}, $\mathcal{S}_D$ defined here is a compact perturbation to the classical one and the above results still hold. 
	
	Our choice makes it easier to deal with the Neumann conditions set at the exterior boundary $\partial \Omega$. Needless to say, if Dirichlet type data are imposed there, it is better to use layer potentials defined by the Dirichlet function $\bm{\Gamma}^{\mathrm{D}}$; if $\Omega = \R^d$ and far field limits of the solution is imposed, it is better to use the classical layer potentials. 
\end{remark}

\subsection{Other related layer potentials}

We review several other layer potential related theories that will be useful in later parts of the paper.

\subsubsection{Layer potentials for Lam\'e system with Dirichlet boundary} 

In section \ref{sec:unifesti} we need to study the single-layer potential operator with Schwartz kernel $\bm{\Gamma}^{\rm D}$. Let $\mathcal{S}^{\mathrm{D}}$ denote this operator, together with its boundary trace $\mathbb{S}^{\mathrm{D}}$ and the Neumann-Poincar\'{e} operator $\mathbb{K}^{\mathrm{D},*}$. They are all defined in the same way as their Neumann counterpart in \eqref{eq:singleS}, \eqref{eq:bSD} and \eqref{eq:KDstar}, simply by replacing the Neumann function $\bm{\Gamma}^{\mathrm{N}}$ there by $\bm{\Gamma}^{\mathrm{D}}$. As before, the reference to the Lam\'e pair $(\lambda,\mu)$ and to (the surface of) the domain $D$ are omitted. 
The corresponding NP operator $\mathbb{K}^{\mathrm{D},*}$ was studied in details in \cite{ammari2007polarization}, and the following results hold as analogues of those in Proposition \ref{prop:SDclassical}.

\begin{proposition}\label{prop:SDDirichlet}
	Under the assumption {\upshape(A1)}, we have the following properties:
	\begin{enumerate}
	  \item [{\upshape(i)}] For any $\bm{\phi} \in H^{-\frac{1}{2}}(\partial D)$, $
		\mathcal{S}^{\mathrm{D}}\bm{\phi} \in H^1(\Omega)$ and
		\begin{equation}
			\left\{
			\begin{aligned}
				& \mathcal{L}_{\lambda,\mu}\mathcal{S}^{\mathrm{D}}\bm{\phi}=0 \quad \mathrm{in}\ \Omega\setminus \partial D, \\
				& \mathcal{S}^{\mathrm{D}}\bm{\phi}|_{\partial \Omega} =0.
			\end{aligned}\right.
		\end{equation}
	  \item [{\upshape (ii)}] $\mathbb{K}^{\mathrm{D},*}:H^{-\frac{1}{2}}(\partial D) \rightarrow H^{-\frac{1}{2}}(\partial D)$ is a bounded operator, furthermore, $\pm \frac{\mathbb{I}}{2}+\mathbb{K}^{\mathrm{D},*}:H_{\mathbf{R}}^{-\frac{1}{2}}(\partial D) \rightarrow H_{\mathbf{R}}^{-\frac{1}{2}}(\partial D)$ are isomorphisms.
	  \item [{\upshape (iii)}] $\mathrm{(Jump\ relation)}$ For any $\bm{\phi}\in H^{-\frac{1}{2}}(\partial D_{\varepsilon})$,
	\begin{equation}
		\left. \frac{\partial \mathcal{S}^{\mathrm{D}}\bm{\phi}}{\partial \nu}\right|_{D_\pm }=\left( \pm \frac{\mathbb{I}}{2} +\mathbb{K}^{\mathrm{D},*} \right)\bm{\phi} .
	\end{equation}
  \item [{\upshape (iv)}] $\mathbb{S}^{\mathrm{D}}:H^{-\frac{1}{2}}(\partial D) \rightarrow H^{\frac{1}{2}}(\partial D)$ is an isomorphism.
  \item[(v)] The \emph{Calder\'on's identity} $\bK^{{\rm D}} \bS^{\rm D} = \bS^{\rm D} \bK^{{\rm D},*}$ holds.
	\end{enumerate}
\end{proposition}

\subsubsection{Layer potential for Stokes system}
For the incompressible inclusions limit modeled by \eqref{eq:limStokes}, we need to investigate Stokes system in $D$. We recall here some layer potential theory for Stokes system; for more details (and applications in hydrostatics), see \cite{dautray1999mathematical,Vooren1965OAL,JLP-stokes}.
	
	Let $\mathbf{{u}}$ be the displacement field and $p$ be the pressure field, the Stokes system in linear hydrostatics with viscosity parameter $\mu$ reads
	\begin{equation}
	  \mathcal{L}_{\infty,\mu}(\mathbf{u},p):=\mu \Delta \mathbf{u}+\nabla p \quad \text{and} \quad \mathrm{div}\,\mathbf{u}=0, \qquad \mathrm{in}\ \Omega.
	\end{equation}
	The corresponding conormal derivative at the surface $\partial E$, where $E\subset \Omega$ is either $\Omega$ or $D$, is 
	\begin{equation}
	  \label{eq:conormalStokes}
		\left. \frac{\partial (\mathbf{u},p)}{\partial \nu_{(\infty,\mu)}}\right|_{\partial E}:=p\mathbf{N}+2\mu \bD(\mathbf{u}) \mathbf{N} \quad \mathrm{on} \ \partial E.
	\end{equation}
	Consider the energy bilinear form $J_{\infty,\mu}^E(\mathbf{u},\mathbf{v})$ defined by
	\begin{equation}
	  \label{eq:JStokes}
		J_{\infty,\mu}^E(\mathbf{u},\mathbf{v}):=2 \mu \int_E \bD(\mathbf{u}) :\bD(\mathbf{v}) -\mu\int_E (\mathrm{div}\,\mathbf{u})(\mathrm{div}\,\mathbf{v}) .
	\end{equation}
It is easy to check the following Green's identity: for any $\mathbf{u},\mathbf{v},p\in H^1(E)$, 
	\begin{equation}\label{eq:GreenStokes}
J_{\infty,\mu}^E(\mathbf{u},\mathbf{v}) = 
			\int_{\partial E} \frac{\partial (\mathbf{u},p)}{\partial \nu_{(\infty,\mu)}}\cdot \mathbf{v} - \int_E \mathcal{L}_{\infty,\mu}(\mathbf{u},p)\cdot \mathbf{v} -
			\mu \int_{\partial E} (\mathrm{div}\,\mathbf{u})(\mathbf{v}\cdot \mathbf{N}) - \int_E p(\mathrm{div}\,\mathbf{v}).
	\end{equation}
	Again, the above is symmetric in $\mathbf{u}$ and $\mathbf{v}$, and the identity simplifies if one of the functions satisfies the homogeneous Stokes system.

	Let $(\bm{\Gamma}^{\mu},Q^{\mu})$ be the Neumann functions of the Stokes system in $\Omega$, i.e., the unique solution of
	\begin{equation}\left\{
		\begin{aligned}
			& \mathcal{L}_{\infty,\mu}(\bm{\Gamma}^{\mu},Q^{\mu})(\cdot,z)=\delta_z\mathbb{I}_d, \quad \mathrm{div}\,\bm{\Gamma}^{\mu}(\cdot,z)=0, \quad \int_{\Omega} Q^{\mu}(\cdot,z)=0  \quad \mathrm{in}\  \Omega, \\
			&\left. \frac{\partial (\bm{\Gamma}^{\mu},Q^{\mu})}{\partial \nu_{(\infty,\mu)}}\right|_{\partial \Omega}=\frac{1}{|\partial \Omega|}, \ \int_{\partial \Omega} \bm{\Gamma}^{\mu}(x,z)\mathbf{r}(x)\,d\sigma(x)=0 \quad \mathrm{for} \ z\in \Omega\ \mathrm{and\ all}\ \mathbf{r}\in \mathbf{R}.
		\end{aligned}\right.
	\end{equation}
Given a density $\bm{\phi}$ on $\partial D$, the single-layer potentials pair $(\mathcal{S}_D^{\infty,\mu}\bm{\phi},\mathcal{P}^{\mu}_D\bm{\phi})$  by
	\begin{equation}
	  \label{eq:cSStokes}
	  \left\{
		\begin{aligned}
			& \mathcal{S}^{\infty,\mu}_D\bm{\phi}(x):=\int_{\partial D} \bm{\Gamma}^{\mu}(x,z)\bm{\phi}(z)\,d\sigma(z) \\
			& \mathcal{P}^{\mu}_D\text{\boldmath $\phi$}(x):=\int_{\partial D}Q^{\mu}(x,z)\text{\boldmath $\phi$}(z)\,d\sigma(z) 
		\end{aligned}\right. \quad\quad \mathrm{for}\ x \in \Omega \setminus \partial D.
	\end{equation}
The boundary trace pair $(\mathbb{S}^{\infty,\mu}_D\text{\boldmath $\phi$},\mathbb{P}_D^{\mu}\bm{\phi})$ on $\partial D $ is defined by
\begin{equation}
	(\mathbb{S}^{\infty,\mu}_D\text{\boldmath $\phi$}(x),\mathbb{P}^{\mu}_D\bm{\phi}(x)):=\lim\limits_{y\rightarrow x\atop y\in C(x)}(\mathcal{S}_D^{\infty,\mu}\bm{\phi}(y),\mathcal{P}^{\mu}_D\bm{\phi}(y))\quad \mathrm{for}\ x\in \partial D.
\end{equation}
The Neumann-Poincar\'{e} operator $\mathbb{K}_D^{\infty,\mu,*}$  is defined by 
\begin{equation}
	\mathbb{K}_D^{\infty,\mu,*}\text{\boldmath $\phi$}(x):=\mathrm{p.v.}\int_{\partial D}  \frac{\partial ( \bm{\Gamma}^{\mu},Q^{\mu})}{\partial_x \nu_{(\infty,\mu)}}(x,z)\text{\boldmath $\phi$}(z)\,d\sigma(z)\quad \mathrm{for} \ x \in  \partial D.
\end{equation}

The following fundamental properties of the single-layer potential for the Stokes system, as analogues of Proposition \ref{prop:SDclassical}, hold:
	
	\begin{proposition}\label{prop:cSStokes} 
		Assume {\upshape(A1)} and $\mu > 0$. We have the following well-known results:
		\begin{enumerate}
		  \item [{\upshape(i)}] For any $\bm{\phi} \in H^{-\frac{1}{2}}(\partial D)$, $
			\mathcal{S}^{\infty,\mu}_D\bm{\phi} \in H^1(\Omega)$ and $\mathcal{P}^{\mu}_D\text{\boldmath $\phi$}\in L^2(\Omega)$ and
			\begin{equation}\label{eq:cSStokes_system}
				\left\{
				\begin{aligned}
				  & \mathcal{L}_{\infty,\mu}(\mathcal{S}^{\infty,\mu}_D\bm{\phi}, \mathcal{P}^{\mu}_D\text{\boldmath $\phi$} )=0 \quad \text{and} \quad \mathrm{div}\,\mathcal{S}^{\infty,\mu}_D \bm{\phi} =0, \quad \mathrm{in}\ \Omega\setminus \partial D, \\
					& \mathcal{S}^{\infty,\mu}_D\bm{\phi}|_{\partial \Omega} \in H^{\frac{1}{2}}_{\mathbf{R}}(\partial \Omega), \\
					& \left.\frac{\partial (\mathcal{S}^{\infty,\mu}_D\bm{\phi}, \mathcal{P}^{\mu}_D\text{\boldmath $\phi$} )}{\partial \nu_{(\infty,\mu)}}\right|_{\partial \Omega} = \frac{1}{|\partial \Omega|}\int_{\partial D}\bm{\phi}.
				\end{aligned}\right.
			\end{equation}
		  \item [{\upshape(ii)}] $\mathbb{K}_D^{\infty,\mu,*}:H^{-\frac{1}{2}}(\partial D) \rightarrow H^{-\frac{1}{2}}(\partial D)$ is a bounded operator.
		  \item [{\upshape(iii)}] $\mathrm{(Jump\ relation)}$ For any $\bm{\phi}\in H^{-\frac{1}{2}}(\partial D)$,
			\begin{equation}\label{jump relation for Stokes system}
				\left.\frac{\partial (\mathcal{S}^{\infty,\mu}_D\bm{\phi}, \mathcal{P}^{\mu}_D\text{\boldmath $\phi$} )}{\partial \nu_{(\infty,\mu)}}\right|_{D_\pm} =\left( \pm \frac{\mathbb{I}}{2} +\mathbb{K}_D^{\infty,\mu,*} \right)\bm{\phi} .
			\end{equation}
		  \item [{\upshape(iv)}] $\mathbb{S}^{\infty,\mu}_D:H^{-\frac{1}{2}}(\partial D) \rightarrow H^{\frac{1}{2}}(\partial D)$ is an isomorphism.
		\end{enumerate}
	\end{proposition}

	As in the case of Proposition \ref{prop:SDclassical}, items (i) and (iii) are more or less standard and item (ii) follows from (iv). The last item was essentially proved in \cite{fabes1988dirichlet}, where $\mathbb{S}^{\infty,\mu}$ was proved to be an isomorphism from $L^2(\partial D)$ to $H^1(\partial D)$. The results above follow by extending the domain of $\mathcal{S}^{\infty,\mu}_D$, and those of the related operators, to $H^{-\frac{1}{2}}(\partial D)$ and duality arguments.

\subsection{Dirichlet to Neumann maps} A key ingredient in our method is the layer potential representation of the Dirichlet to Neumann (DtN) operators associated to the Lam\'e and Stokes systems. 

We focus first on Lam\'e systems. Given a pair of Lam\'e coefficients $(\lambda,\mu)$ on $\Omega\setminus \ol D$ 
and another (could be the same) pair $(\widetilde\lambda,\widetilde\mu)$ on $D$, we consider two DtN maps, an exterior one associated to the outer domain $\Omega\setminus \ol D$ and an inner one associated to $D$. They are denoted, respectively, by $\Lambda^{\lambda,\mu,\mathrm{e}}_D$ and $\Lambda^{\widetilde\lambda,\widetilde\mu,\mathrm{i}}_D$, and they are defined by
	\begin{equation}\label{eq:DtNdef}
		\Lambda^{\lambda,\mu,\mathrm{e}}_D \bm{\phi} :=\left. \frac{\partial \mathbf{u}}{\partial \nu_{(\lambda,\mu)}}\right|_{+},\quad \Lambda^{\widetilde\lambda,\widetilde\mu,\mathrm{i}}_D \bm{\phi} :=\left. \frac{\partial \mathbf{u}}{\partial \nu_{(\widetilde\lambda,\widetilde\mu)}}\right|_{-}, \qquad \bm{\phi} \in H^{\frac12}(\partial D),
	\end{equation} 
and are functions on $\partial D$. Here $\mathbf{u}$ is the unique solution of the Lam\'e system in $\Omega\setminus \ol D$ and in $D$ with the corresponding coefficients, and with Dirichlet data $\mathbf{u} = \bm{\phi}$ on $\partial D$. More precisely, $\mathbf{u}$ solves the following problem:
	\begin{equation}\label{eq:DtNproblem}
		\left\{
		\begin{aligned}
			& \mathcal{L}_{\lambda,\mu}\mathbf{u}=0 \quad  \mathrm{in}\ \Omega\setminus \overline{D}, \\
			& \mathcal{L}_{\widetilde{\lambda},\widetilde{\mu}}\mathbf{u}=0 \quad  \mathrm{in}\ D, \\
			& \mathbf{u}|_{\partial D}=\bm{\phi}, \quad \mathbf{u}|_{\partial \Omega} \in H^{\frac{1}{2}}_{\mathbf{R}}(\partial \Omega), \\
			& \left.\frac{\partial \mathbf{u}}{\partial \nu}\right|_{\partial \Omega}\mathrm{\ is \ a \ constant \ vector}.
		\end{aligned}\right.
	\end{equation}
The above system combines an exterior and an interior problems related by the same Dirichlet data along $\partial D$. The particular boundary data at $\partial \Omega$ for the exterior (to $D$) problem is consistent with the Neumann condition at $\partial \Omega$ in \eqref{eq:transmissionproblem}. As an immediate consequence of the following result, the DtN maps are bounded operators from $H^{\frac{1}{2}}(\partial D)$ to $ H^{-\frac{1}{2}}(\partial D)$.

	\begin{proposition}\label{prop:DtNLameSol}
		Assume that {\upshape(A1)} holds. Then, for any $\bm{\phi}\in H^{\frac{1}{2}}(\partial D)$, the problem \eqref{eq:DtNproblem} has a unique solution $\mathbf{u}\in H^1(\Omega)$. 
	\end{proposition}
	\begin{proof}
	\emph{Uniqueness}. Suppose that $\mathbf{u}\in H^1(\Omega)$ solves \eqref{eq:DtNproblem} with $\bm{\phi}=0$. In view of the Green's identity \eqref{eq:GreenLame}, we get
	\begin{equation*}
		\int_{\partial D} \left(\left.\frac{\partial \mathbf{u}}{\partial \nu_{(\widetilde{\lambda},\widetilde{\mu})}}\right|_- -\left.\frac{\partial \mathbf{u}}{\partial \nu_{(\lambda,\mu)}}\right|_+ \right)\cdot \mathbf{u} +\int_{\partial \Omega} \frac{\partial \mathbf{u}}{\partial \nu_{(\lambda,\mu)}}\cdot \mathbf{u} = J_{\lambda,\mu}^{\Omega \setminus \ol D}(\mathbf{u}) + J_{\widetilde{\lambda},\widetilde{\mu}}^D(\mathbf{u}).
	\end{equation*}
	Using the boundary conditions for $\mathbf{u}$, we obtain $J_{\lambda,\mu}^{\Omega \setminus \ol D}(\mathbf{u})+J_{\widetilde{\lambda},\widetilde{\mu}}^D(\mathbf{u})=0$. This implies that $\mathbf{u}\in \mathbf{R}$ in $\Omega\setminus \ol D$ and in each component of $D$. By $\mathbf{u}|_{\partial \Omega} \in H^{\frac{1}{2}}_{\mathbf{R}}(\partial \Omega)$, we deduce that $\mathbf{u}=0$ in $\Omega\setminus \ol D$. By continuity across $\partial D$, we get $\mathbf{u} = 0$ in $\Omega$.
	
	\emph{Existence}. In view of Proposition \ref{prop:SDclassical} (iv), $\mathbb{S}^{\lambda,\mu}_D$ and $\mathbb{S}^{\widetilde{\lambda},\widetilde{\mu}}_D$ are isomorphisms from $H^{-\frac{1}{2}}(\partial D)$ to $H^{\frac{1}{2}}(\partial D)$. We then check that
	\begin{equation}\label{eq:DtNexplicit}
		\mathbf{u}=\left\{
		\begin{aligned}
			& \mathcal{S}^{\lambda,\mu}_D(\mathbb{S}^{\lambda,\mu}_D)^{-1}\bm{\phi} \quad \mathrm{in}\ \Omega\setminus \ol D \\
			& \mathcal{S}^{\widetilde{\lambda},\widetilde{\mu}}_D(\mathbb{S}^{\widetilde{\lambda},\widetilde{\mu}}_D)^{-1}\bm{\phi} \quad \mathrm{in}\ D 
		\end{aligned}\right.
	\end{equation}
	 is well-defined and satisfies $\mathbf{u} = \bm{\phi}$ on $\partial D$. By Proposition \ref{prop:SDclassical} (i), we see that $\mathbf{u}$ solves \eqref{eq:DtNproblem}.
\end{proof}

Using the explicit formula \eqref{eq:DtNexplicit}, the \emph{jump relation} in Proposition \ref{prop:SDclassical} (iii) and the definition  \eqref{eq:DtNdef}, we get the following representation of DtN maps.

\begin{proposition}\label{DtN representation}
	The DtN maps are represented by layer potentials as follows.
		\begin{equation}
		\label{eq:elastDtN}
			\Lambda^{\lambda,\mu,\mathrm{e}}_D=\left(  \frac{\mathbb{I}}{2}  + \mathbb{K}^{\lambda,\mu,*}_D \right)(\mathbb{S}^{\lambda,\mu}_D)^{-1}, \quad \Lambda^{\widetilde{\lambda},\widetilde{\mu},\mathrm{i}}_D=\left(  -\frac{\mathbb{I}}{2}  + \mathbb{K}^{\widetilde{\lambda},\widetilde{\mu},*}_D \right)(\mathbb{S}^{\widetilde{\lambda},\widetilde{\mu}}_D)^{-1}.
		\end{equation}
\end{proposition}

It is clear that $\mathrm{ker}\,\Lambda^{\widetilde{\lambda},\widetilde{\mu},\mathrm{i}}_D$ is non-empty; in fact, it consists of functions that belong to $\mathbf{R}$ in each components of $D$. The following result shows that the difference operator $\Lambda_D^{\widetilde{\lambda},\widetilde{\mu},\mathrm{i}}-\Lambda_D^{\lambda,\mu,\mathrm{e}}$ is an isomorphism from $H^{\frac{1}{2}}(\partial D)$ to $H^{-\frac{1}{2}}(\partial D)$. This plays a key role in our method.

		\begin{proposition}\label{boundary operator isomorphism}
		Assume that {\upshape(A1)} holds. Then the boundary operator $\Lambda_D^{\widetilde{\lambda},\widetilde{\mu},\mathrm{i}}-\Lambda_D^{\lambda,\mu,\mathrm{e}}$ is an isomorphism from $H^{\frac{1}{2}}(\partial D)$ to $H^{-\frac{1}{2}}(\partial D)$.
	\end{proposition}
	
	\begin{proof}
		Let $\bm{\phi}\in H^{\frac{1}{2}}(\partial D)$. Using Green's identity we have
		\begin{equation*}
			\langle \bm{\phi}, (\Lambda_D^{\widetilde{\lambda},\widetilde{\mu},\mathrm{i}}-\Lambda_D^{\lambda,\mu,\mathrm{e}})\bm{\phi}\rangle_{H^{\frac{1}{2}},H^{-\frac{1}{2}}}=J^D_{\widetilde{\lambda},\widetilde{\mu}}(\mathbf{u})+J^{\Omega\setminus \ol D}_{\lambda,\mu}(\mathbf{u}),
		\end{equation*}
		where $\mathbf{u}$ is the solution of problem \eqref{eq:DtNproblem}. By the Cauchy-Schwarz inequality, the elliptic condition \eqref{elliptic condition} and the trace inequality of symmetric matrices, we have
		\begin{equation*}
			\begin{aligned}
				J^D_{\widetilde{\lambda},\widetilde{\mu}}(\mathbf{u})+J^{\Omega\setminus \ol D}_{\lambda,\mu}(\mathbf{u}) &\geq \min \{ 2\widetilde{\mu}, d\widetilde{\lambda}+2\widetilde{\mu}\} \int_D |\bD(\mathbf{u})|^2 +\min \{ 2\mu, d\lambda+2\mu \}\int_{\Omega\setminus \ol D} |\bD(\mathbf{u})|^2 \\
				&\geq C(\lambda,\mu,\widetilde{\lambda},\widetilde{\mu})\int_{\Omega } |\bD(\mathbf{u})|^2.
			\end{aligned}
		\end{equation*}
		Since $\mathbf{u}|_{\partial \Omega}\in H^{\frac{1}{2}}_{\mathbf{R}}(\partial \Omega)$, by the argument in Remark \ref{rem:equinorm} and thanks to the Korn's inequality (Lemma \ref{second korn}), we obtain
		\begin{equation*}
			\langle \bm{\phi},(\Lambda^{\widetilde{\lambda},\widetilde{\mu},\mathrm{i}}-\Lambda^{\lambda,\mu,\mathrm{e}})\bm{\phi}\rangle_{H^{\frac{1}{2}},H^{-\frac{1}{2}}} \geq C \|\mathbf{u}\|_{H^1(\Omega)}^2\geq C \|\mathbf{u}|_{\partial D} \|_{H^{\frac{1}{2}}}^2=C \|\bm{\phi} \|_{H^{\frac{1}{2}}}^2.
		\end{equation*}
		Note that $C$ depends on $D$ and is not universal. In view of Lemma \ref{coercive theorem II}, we get the conclusion.
	\end{proof}

	\begin{remark}
	If the boundary $\partial D$ is $C^{\infty}$ (this can be further relaxed to $C^{1,\alpha}$), one can use pseudo-differential calculus to prove that $\Lambda_D^{\widetilde{\lambda},\widetilde{\mu},\mathrm{i}}-\Lambda_D^{\lambda,\mu,\mathrm{e}}$ is a Fredholm operator. For instance, assume $d=3$ and choose a suitable coordinate, it is not difficult to calculate that $\Lambda_D^{\widetilde{\lambda},\widetilde{\mu},\mathrm{i}}-\Lambda_D^{\lambda,\mu,\mathrm{e}}$ is a pseudo-differential operator with order $-1$ and the eigenvalues of its principal symbol are
	\begin{equation*}
		p(0)\ \ \ \mathrm{and} \ \ \ p(\pm i|\xi'|),
	\end{equation*}
	where $\xi' \in \mathbb{R}^2$ is the new coordinate and $p(t)$ is the polynomial
	\begin{equation*}
		p(t)= (\mu+\widetilde{\mu})|\xi'|+2i\left( \frac{ \mu^2 }{\lambda+3\mu} -\frac{ \widetilde{\mu}^2 }{\widetilde{\lambda}+3\widetilde{\mu}} \right)  t -\left(  \frac{(\lambda+\mu)\mu }{\lambda+3\mu}+\frac{(\widetilde{\lambda}+\widetilde{\mu})\widetilde{\mu} }{\widetilde{\lambda}+3\widetilde{\mu}}  \right)|\xi'|^{-1}t^2 ,
	\end{equation*}
	the elliptic condition \eqref{elliptic condition} implies that $p(0)$ and $p(\pm i|\xi'|)$ are nonzero, as $\Lambda^{\widetilde{\lambda},\widetilde{\mu},\mathrm{i}}-\Lambda^{\lambda,\mu,\mathrm{e}}$ is a Fredholm operator with index zero, see \cite{Agranovich1999SpectralLS,taylor2013partial}.
\end{remark}

\subsubsection{Dirichlet to Neumann map for the coupled Lam\'e-Stokes system}

Given an admissible Lam\'e pair $(\lambda,\mu)$ in $\Omega\setminus \ol D$ and a viscosity constant $\widetilde\mu > 0$ in $D$, we also consider the DtN maps associated to the coupled Lam\'e-Stokes system related to the limit problem \eqref{eq:limStokes}. More precisely, let $\mathbf{u} \in H^1(\Omega)$ and $p\in L^2(D)$ so that $(\mathbf{u},p)$ solves 
the following coupled Lam\'e-Stokes problem: 
\begin{equation}\label{DtN problem stokes}
	\left\{
	\begin{aligned}
		& \mathcal{L}_{\lambda,\mu}\mathbf{u}=0, \qquad \mathrm{in}\ \Omega\setminus \overline{D}, \\
		& \mathcal{L}_{\infty,\widetilde{\mu}}(\mathbf{u},p)=0 \quad \text{and} \quad  \mathrm{div}\,\mathbf{u}=0, \qquad \mathrm{in}\ D, \\
		& \mathbf{u}|_{\partial D}=\bm{\phi}, \quad \mathbf{u}|_{\partial \Omega} \in H^{\frac{1}{2}}_{\mathbf{R}}(\partial \Omega), \\
		& \left.\frac{\partial \mathbf{u}}{\partial \nu}\right|_{\partial \Omega}\mathrm{\ is \ a \ constant \ vector}.
	\end{aligned}\right.
\end{equation}
Then the associated exterior DtN map $\Lambda^{\lambda,\mu,\mathrm{e}}_D$ is defined as before in \eqref{eq:DtNdef}, and the interior DtN map $\Lambda^{\infty,\widetilde{\mu},\mathrm{i}}_D$ is defined by
\begin{equation}\label{DtN definition stokes} \Lambda^{\infty,\widetilde{\mu},\mathrm{i}}_D \bm{\phi} :=\left. \frac{\partial (\mathbf{u},p)}{\partial \nu_{(\infty,\widetilde{\mu})}}\right|_{\partial D_-}.
\end{equation} 

As in the previous subsection, we prove that the problem above is well-posed, derive a layer potential representation for the solution and identify the DtN map.

\begin{proposition}\label{solva for DtN Stokes}
	For any $\bm{\phi}\in H^{\frac{1}{2}}(\partial D)$, problem \eqref{DtN problem stokes} has a unique solution $\mathbf{u}\in H^1(\Omega)$.
\end{proposition}
\begin{proof}
  \emph{Uniqueness}. Suppose that $(\mathbf{u},p)$ is the solution of problem \eqref{DtN problem stokes} with $\bm{\phi}=0$. Then by the Green's identities \eqref{eq:GreenLame} and \eqref{eq:GreenStokes}, we have
	\begin{equation*}
		\int_{\partial D} \left(\left.\frac{\partial (\mathbf{u},p)}{\partial \nu_{(\infty,\widetilde{\mu})}}\right|_- -\left.\frac{\partial \mathbf{u}}{\partial \nu_{(\lambda,\mu)}}\right|_+ \right)\cdot \mathbf{u} + \int_{\partial \Omega} \frac{\partial \mathbf{u}}{\partial \nu_{(\lambda,\mu)}}\cdot \mathbf{u} = J_{\infty,\widetilde{\mu}}^D(\mathbf{u}) + J^{\Omega \setminus \ol D}_{\lambda,\mu}(\mathbf{u}).
	\end{equation*}
In view of the boundary conditions of $\mathbf{u}$ at $\partial D$ and at $\partial \Omega$, we obtain 
$J_{\infty,\widetilde{\mu}}^D(\mathbf{u})+J^{\Omega \setminus \ol D}_{\lambda,\mu}(\mathbf{u})=0$. Following the same argument in the proof of Proposition \ref{prop:DtNLameSol}, we see $\mathbf{u}=0$. 
	
\emph{Existence}. Since $\mathbb{S}^{\lambda,\mu}_D$ and $\mathbb{S}_D^{\infty,\widetilde{\mu}}$ are isomorphisms from $H^{-\frac{1}{2}}(\partial D) $ to $ H^{\frac{1}{2}}(\partial D)$, the functions
	\begin{equation}\label{explicit DtN stokes}
		\mathbf{u}=\left\{
		\begin{aligned}
			& \mathcal{S}^{\lambda,\mu}_D(\mathbb{S}^{\lambda,\mu}_D)^{-1}\bm{\phi} \quad \mathrm{in}\ \Omega\setminus D \\
			& \mathcal{S}^{\infty,\widetilde{\mu}}_D(\mathbb{S}^{\infty,\widetilde{\mu}}_D)^{-1}\bm{\phi} \quad \mathrm{in}\ D 
		\end{aligned}\right. \quad \mathrm{and} \quad p=\mathcal{P}^{\widetilde{\mu}}_D(\mathbb{S}^{\infty,\widetilde{\mu}}_D)^{-1}\bm{\phi} \quad \mathrm{in}\ D
	\end{equation}
	are well-defined. By Proposition \ref{prop:cSStokes} and Proposition \ref{prop:SDclassical}, we see that $(\mathbf{u},p)$ solve \eqref{DtN problem stokes}.
\end{proof}

Using the explicit formula \eqref{explicit DtN stokes}, the jump relation in Proposition \ref{prop:cSStokes}, and by repeating the proof of Proposition \ref{boundary operator isomorphism}, we get the following results.
\begin{proposition}\label{boundary operator isomorphism stokes}
  The DtN map $\Lambda^{\infty,\widetilde{\mu},\mathrm{i}}_D$ is represented by layer potentials:
	\begin{equation}
	 \Lambda^{\infty,\widetilde{\mu},\mathrm{i}}_D=\left(  -\frac{\mathbb{I}}{2}  + \mathbb{K}^{\infty,\widetilde{\mu},*}_D \right)(\mathbb{S}^{\infty,\widetilde{\mu}}_D)^{-1}.
	\end{equation}
	Moreover, the boundary operator $\Lambda_D^{\infty,\widetilde{\mu},\mathrm{i}}-\Lambda_D^{\lambda,\mu,\mathrm{e}}:H^{\frac{1}{2}}(\partial D)\rightarrow H^{-\frac{1}{2}}(\partial D)$ is an isomorphism.
\end{proposition}


\subsection{The space of single-layer potentials and its decompositions}

Let $D$ be a sub-domain of $\Omega$, and let $(\lambda,\mu)$ be an admissible Lam\'e pair. Partially inspired by the work of \cite{MR2308861} and \cite{bonnetier2019homogenization}, we define the space
\begin{equation}
\label{eq:Espace}
	\mathcal{E}:=\{\mathbf{u}\in H^1(\Omega): \mathcal{L}_{\lambda,\mu} \mathbf{u}=0 \ \mathrm{in}\ D \cup (\Omega\setminus \ol D)\},
\end{equation}
which is a closed subspaces of $H^1(\Omega)$ and consists of functions that solve the Lam\'e system with the given coefficients both in $D$ and in $\Omega\setminus \ol D$ (but not necessarily in $\Omega$). We further define the following subspaces of $\mathcal{E}$ that are formed by single-layer potentials:
\begin{equation}
\label{eq:frakH}
\begin{aligned}
	&\frakH^{\mathrm{N}}:=\{\mathcal{S}^{\mathrm{N}}\bm{\phi}:\bm{\phi}\in H^{-\frac{1}{2}}(\partial D) \}, \qquad 
	&\frakH_{\mathbf{R}}^{\mathrm{N}}:=\{\mathcal{S}^{\mathrm{N}}\bm{\phi}:\bm{\phi}\in H_{\mathbf{R}}^{-\frac{1}{2}}(\partial D) \}, \\
&\frakH^{\mathrm{D}}:=\{\mathcal{S}^{\mathrm{D}}\bm{\phi}:\bm{\phi}\in H^{-\frac{1}{2}}(\partial D) \}, \qquad 
	&\frakH_{\mathbf{R}}^{\mathrm{D}}:=\{\mathcal{S}^{\mathrm{D}}\bm{\phi}:\bm{\phi}\in H_{\mathbf{R}}^{-\frac{1}{2}}(\partial D) \}.
	\end{aligned}
\end{equation}
Because of the boundary conditions at $\partial \Omega$, we verify that $\frakH^{\rm N} \cap \mathbf{R} = \{0\}$ and $\frakH^{\rm N} \cap \mathbf{R} = \{0\}$.  On $\frakH^{\mathrm{N}}$ and $\frakH^{\mathrm{D}}$, we define the bilinear and symmetric product
\begin{equation}
	(\mathbf{u},\mathbf{v})_{\frakH}:=J^{\Omega}_{\lambda,\mu}(\mathbf{u},\mathbf{v}).
\end{equation}
In view of Remark \ref{rem:equinorm}, in the space $\frakH$, $\mathbf{u}\mapsto (\mathbf{u},\mathbf{u})_{\frakH}^{\frac12}$ is a norm equivalent to $\|\cdot\|_{H^1(\Omega)}$, and $(\cdot,\cdot)_{\frakH}$ is an inner product.

We work mainly with the Hilbert space $H^{-\frac{1}{2}}(\partial D)$, since the density functions in the layer potential operators, e.g.\,$\cS, \mathbb{S},\mathbb{K}^*, \Lambda^{\rm i}, \Lambda^{\rm e}$, belong to this space. The usual norm on $H^{-\frac{1}{2}}(\partial D)$ is the dual norm, i.e., $\|\bm{\phi}\|_{H^{-\frac12}} = \sup\{\langle\mathbf{h},\bm{\phi}\rangle \,:\, \|\mathbf{h}\|_{H^{\frac12}} = 1\}$. Inspired by \cite{MR2308861}, we introduce the following inner products on $H^{-\frac12}$:
\begin{equation}\label{new inner product}
	(\bm{\phi},\bm{\psi})_{\mathbb{S}^{\mathrm{N}}}:=-\int_{\partial D_{\varepsilon}} \bm{\phi}\cdot \mathbb{S}^{\mathrm{N}}\bm{\psi}=J^{\Omega}_{\lambda,\mu}(\mathcal{S}^{\mathrm{N}}\bm{\phi},\mathcal{S}^{\mathrm{N}}\bm{\psi}) = (\mathcal{S}^{\mathrm{N}}\bm{\phi},\mathcal{S}^{\mathrm{N}}\bm{\psi})_{\frakH}.
\end{equation}
\begin{equation}\label{new inner product 1}
	(\bm{\phi},\bm{\psi})_{\mathbb{S}^{\mathrm{D}}}:=-\int_{\partial D_{\varepsilon}} \bm{\phi}\cdot \mathbb{S}^{\mathrm{D}}\bm{\psi}=J^{\Omega}_{\lambda,\mu}(\mathcal{S}^{\mathrm{D}}\bm{\phi},\mathcal{S}^{\mathrm{D}}\bm{\psi}) = (\mathcal{S}^{\mathrm{D}}\bm{\phi},\mathcal{S}^{\mathrm{D}}\bm{\psi})_{\frakH}.
\end{equation}
Here, the integrals are understood as pairings. They are indeed inner products because we have seen that $J^\Omega_{\lambda,\mu}(\cdot,\cdot)$ defines an inner product on $\frakH$. 
\begin{remark}\label{rem:NPselfadj}
The salient feature of those new inner products is: the Neumann-Poincar\'e operator $\mathbb{K}^{{\rm N},*}$ becomes self-adjoint as a bounded linear transformation on $(H^{-\frac12},(\cdot,\cdot)_{\bS^{\rm N}})$. To check this, take any $\bm{\phi},\bm{\psi}\in H^{-\frac12}(\partial D)$, we compute
\begin{equation*}
\begin{aligned}
  ((\frac{\bI}{2} + \bK^{{\rm N},*})\bm{\phi},\bm{\psi})_{\bS^{\rm N}} &= \langle \bS^{\rm N}(\frac{\bI}{2} + \bK^{{\rm N},*})\bm{\phi},\bm{\psi} \rangle_{H^{\frac12},H^{-\frac12}} = \langle (\frac{\bI}{2} + \bK^{{\rm N}})\bS^{\rm N}\bm{\phi},\bm{\psi}\rangle_{H^{\frac12},H^{-\frac12}}\\
  &= \langle \bS^{\rm N}\bm{\phi},(\frac{\bI}{2} + \bK^{{\rm N},*})\bm{\psi}\rangle_{H^{\frac12},H^{-\frac12}} = (\bm{\phi},(\frac{\bI}{2} + \bK^{{\rm N},*})\bm{\psi})_{\bS^{\rm N}}.
  \end{aligned}  
\end{equation*}  
The second equality sign holds due to the Calder\'on identity (see Proposition \ref{prop:SDclassical}). 
\end{remark}

The following characterizations and decompositions of $\frakH^{\rm N}$ and $\frakH^{\rm D}$ are important.

\begin{proposition}
	Assume that {\upshape(A1)} holds. We have the basic properties:
	\begin{enumerate}
		\item [{\upshape (i)}] The space $\frakH^{\rm N}$ and its subspace $\frakH^{\rm N}_{\mathbf{R}}$ are characterized by
		\begin{equation}\label{eq:HNspace}
		\begin{aligned}
			&\frakH^{\mathrm{N}}=\left\{\mathbf{u}\in \mathcal{E}: \mathbf{u}|_{\partial \Omega}\in H^{\frac{1}{2}}_{\mathbf{R}}(\partial \Omega),\left.\frac{\partial \mathbf{u}}{\partial \nu} \right|_{\partial \Omega}\mathrm{\ is \ a\ constant \ vector}\right\},\\
&\frakH^{\mathrm{N}}_{\mathbf{R}}= \left\{ \mathbf{u}\in \mathcal{E}:    \mathbf{u}|_{\partial \Omega}\in H^{\frac{1}{2}}_{\mathbf{R}}(\partial \Omega),\left.\frac{\partial \mathbf{u}}{\partial \nu} \right|_{\partial \Omega}=0 , \left. \frac{\partial \mathbf{u}}{\partial \nu} \right|_{\partial D, +} \in H^{-\frac{1}{2}}_{\mathbf{R}}(\partial D)  \right\}.
\end{aligned}
		\end{equation}
		
		\item [{\upshape (ii)}] The space $\frakH^{\rm D}$ and its subspace $\frakH^{\rm D}_{\mathbf{R}}$ are characterized by 
		\begin{equation}
		\label{eq:HDspace}
		\begin{aligned}
			&\frakH^{\mathrm{D}}=\left\{\mathbf{u}\in \mathcal{E}: \mathbf{u}|_{\partial \Omega} =0\right\},\\
			&\frakH^{\mathrm{D}}_{\mathbf{R}}= \left\{ \mathbf{u}\in \mathcal{E}:   \mathbf{u}|_{\partial \Omega}=0 , \left. \frac{\partial \mathbf{u}}{\partial \nu} \right|_{\partial D, +} \in H^{-\frac{1}{2}}_{\mathbf{R}}(\partial D)  \right\}.
			\end{aligned}
		\end{equation}
			
	\item [{\upshape(iii)}] The space $\frakH^{\rm D}$ satisfies the following decomposition:
	\begin{equation}
	\label{eq:HDdecom}
		\frakH^{\mathrm{D}} = \frakH^{\mathrm{D}}_{\mathbf{R}} 
		\oplus \mathcal{S}^{\mathrm{D}}\,\mathrm{ker}\,\left(-\frac{\mathbb{I}}{2} +\mathbb{K}^{\mathrm{D},*} \right),
	\end{equation}
	and this is an orthogonal decomposition in $\frakH^{\rm D}$ with respect to the inner product $(\cdot,\cdot)_\frakH$. Moreover, the second space in the decomposition \ref{eq:HDdecom} is characterized by
	\begin{equation}
	\label{eq:kerD}
    \mathcal{S}^{\mathrm{D}}\,\mathrm{ker}\,\left(-\frac{\mathbb{I}}{2} +\mathbb{K}^{\mathrm{D},*} \right) = \{\mathbf{u} \in \frakH^{\rm D} \,:\, \mathbf{u} \in \mathbf{R} \ \text{in each component of $D$}\}.
	\end{equation}
	\item [{\upshape (iv)}] The space $\frakH^{\rm N}$ satisfies the following orthogonal decomposition with respect to $(\cdot,\cdot)_{\frakH}$:
	\begin{equation}
	\label{eq:HNdecom}
		\frakH^{\mathrm{N}} = \frakH^{\mathrm{N}}_{\mathbf{R}} 
		\oplus \mathcal{S}^{\mathrm{N}}\,\mathrm{ker}\,\left( -\frac{\mathbb{I}}{2} +\mathbb{K}^{\mathrm{N},*} \right).
	\end{equation}
	Moreover, the second space in the decomposition \eqref{eq:HNdecom} is
    \begin{equation}
	\label{eq:kerN}
    \mathcal{S}^{\mathrm{N}}\,\mathrm{ker}\,\left(-\frac{\mathbb{I}}{2} +\mathbb{K}^{\mathrm{N},*} \right) = \{\mathbf{u} \in \frakH^{\rm N} \,:\, \mathbf{u} \in \mathbf{R} \ \text{in each component of $D$}\}
	\end{equation}
	\end{enumerate}
\end{proposition}
\begin{remark}
\label{rem:H-1/2decomp}
By definition, $\frakH^{\rm N} = \cS^{\rm N}H^{-\frac12}$ and $\frakH^{\rm N}_\mathbf{R} = \cS^{\rm N}H^{-\frac12}_{\mathbf{R}}$; similar identifications hold for $\frakH^{\rm D}$ and $\frakH^{\rm D}_{\mathbf{R}}$. In view of \eqref{new inner product}, the orthogonal decomposition \eqref{eq:HNdecom} is equivalent to
	\begin{equation}
	\label{eq:H-1/2decomp}
		H^{-\frac{1}{2}}(\partial D)=H_{\mathbf{R}}^{-\frac{1}{2}}(\partial D) \oplus \mathrm{ker}\,\left( -\frac{\mathbb{I}}{2} +\mathbb{K}^{\mathrm{N},*}\right),
	\end{equation}
	and the two subspaces on the right hand side are orthogonal with respect to the inner product $(\cdot,\cdot)_{\bS^{\rm N}}$ on $H^{-\frac12}$. Similarly, the orthogonal decomposition \eqref{eq:HDdecom} is equivalent to
\begin{equation}
	\label{eq:H-1/2decompD}
		H^{-\frac{1}{2}}(\partial D)=H_{\mathbf{R}}^{-\frac{1}{2}}(\partial D) \oplus \mathrm{ker}\,\left( -\frac{\mathbb{I}}{2} +\mathbb{K}^{\mathrm{D},*}\right),
	\end{equation}
	and this is an orthogonal decomposition in $(H^{-\frac12},(\cdot,\cdot)_{\bS^{\rm D}})$. 
	
	We further claim  $(-\frac{\bI}{2}+\bK^{{\rm N},*})H^{-\frac12} = H^{-\frac12}_{\mathbf{R}}$. Indeed, in view of the Green's identity (used in each component of $D$), we see that the left side is contained in the right. The other direction is also true as stated in Proposition \ref{prop:SDclassical} (ii). The decompositions \eqref{eq:H-1/2decomp} then just says
	\begin{equation}
	  \label{eq:H-1/2decomp2}
H^{-\frac12}(\partial D) = \mathrm{ran}(-\frac{\bI}{2} + \bK^{{\rm N},*}) \oplus \mathrm{ker}(-\frac{\bI}{2} + \bK^{{\rm N},*}).
 	\end{equation}  
\end{remark}
\begin{proof}
	We only outline the proofs for the propositions concerning $\frakH^{\rm N}$; the case of $\frakH^{\rm D}$ is similar. Item (i) follows from Propositions \ref{prop:SDclassical} directly. Indeed, by \eqref{eq:normalpO} we see $\cS^{\rm N}\bm{\phi} \in \mathcal{E}$ for any $\bm{\phi} \in H^{-\frac12}(\partial D)$ and satisfies the boundary conditions at $\partial \Omega$. If we have further $\bm{\phi} \in H^{-\frac12}_{\mathbf{R}}(\partial D)$, then $\cS^{\rm N}\bm{\phi}$ also satisfies the condition at $\partial D$. Conversely, by Proposition \ref{prop:SDclassical}, the functions in the right hand sides of \eqref{eq:HNspace} can be realized by single-layer potentials of the form $\cS^{\rm N}\bm{\phi}$. 

	For \eqref{eq:kerN}, by the jump relation \eqref{jump relation} it is clear that the right hand side is a subset of the left. For the other direction, we note if $\bm{\phi} \in \mathrm{ker}(-\frac{\mathbb{I}}{2} + \mathbb{K}^{{\rm N},*})$, then $\mathbf{u} := \cS^{\rm N}\bm{\phi}$ satisfies $\mathcal{L}_{\lambda,\mu}\mathbf{u} = 0$ in $D$ and $\partial \mathbf{u}/\partial \nu = 0$ in each component of $D$. It follows that $\mathbf{u} \in \mathbf{R}$ in each component of $D$.

	Finally we prove \eqref{eq:HNdecom}. Suppose $\mathbf{u} = \cS^{\rm N} \bm{\phi}$ is orthogonal to the space in \eqref{eq:kerN} with respect to the inner product $(\cdot,\cdot)_{\frakH}$. For each component $E$ of $D$ and for each basis vector $\mathbf{r}_j$ of $\mathbf{R}$, there exists $\mathbf{w} \in \frakH^{\rm N}$ so that $\mathbf{w} = \mathbf{r}_j$ in $E$ and $\mathbf{w} = 0$ in all other components of $D$. In view of the Green's identity on $\Omega\setminus D$ and the facts that $\partial \mathbf{u}/\partial \nu\rvert_{\partial \Omega}$ is a constant and $\mathbf{w}\rvert_{\partial \Omega} \in H^{\frac12}_{\mathbf{R}}$, we see that 
	\begin{equation*}
	0 = J^{\Omega\setminus \ol D}(\mathbf{u},\mathbf{w}) = - \langle\mathbf{r}_j, \partial \mathbf{u}/\partial \nu\rvert_+\rangle_{H^{\frac12}(\partial E),H^{-\frac12}(\partial E)}.
	\end{equation*}
	This shows $\partial \mathbf{u}/\partial \nu\rvert_+ \in H^{-\frac12}_{\mathbf{R}}(\partial D)$. By definition we already know that $\partial \mathbf{u}/\partial \nu\rvert_{\partial \Omega} = \mathbf{c}$ is a constant. Using Green's identity again, we see that
	\begin{equation*}
	0 = J^{\Omega\setminus \ol D}(\mathbf{u},\mathbf{c}) = \int_{\partial \Omega} |\mathbf{c}|^2.
	\end{equation*}
	It follows that $\mathbf{c} = 0$ and $\mathbf{u} \in \frakH^{\rm N}_{\mathbf{R}}$.
\end{proof}

\section{Layer potential representations for solutions of the transmission problems}\label{Single layer potential representations for the solutions}
	
In this section, we establish representation formulas for the transmission problem \eqref{eq:transmissionproblem}, \eqref{eq:limStokes}, \eqref{eq:limZero} and \eqref{eq:limRigid} and, in particular, prove Theorems \ref{thm:reptrans} and \ref{thm:limStokes}. We always assume that $\Omega$ and $D$ satisfies (A1).

	\subsection{The solution to problem \eqref{eq:transmissionproblem}} 
	
	The aim here is to prove Theorem \ref{thm:reptrans}. Recall the definition of the background solution $\mathbf{G}$ in \eqref{background solution}. Using layer potentials, we seek for solutions of the form	
	\begin{equation*}
		\mathbf{u}=\left\{
		\begin{aligned}
			&\mathbf{G}+\mathcal{S}^{\lambda,\mu}_{D} \text{\boldmath $\phi$} &\mathrm{in}\   \Omega\setminus \ol D\\
			&\mathcal{S}^{\widetilde{\lambda},\widetilde{\mu}}_{D}  \text{\boldmath $\psi$}&\mathrm{in}\  D.
		\end{aligned}	\right. 
	\end{equation*}
	where $\bm{\phi}$ and $\bm{\psi}$ are densities defined on $\partial D$ to be found. In view of Proposition \ref{prop:SDclassical}, for $\mathbf{u}$ to be a solution of \eqref{eq:transmissionproblem} and, in particular, to satisfy the boundary conditions at $\partial D$ and $\partial \Omega$, it is necessary and sufficient to impose that $(\text{\boldmath $\psi$}, \text{\boldmath $\phi$})\in H^{-1/2}(\partial D)\times H_{\mathbf{R}}^{-1/2}(\partial D)$ and that they solve the boundary integral equations \eqref{boundary integral equations}. 
Using the layer potential representations for DtN maps Proposition \ref{DtN representation}, we can rewrite \eqref{boundary integral equations} to the following form:
\begin{equation}\label{boundary integral equations III}
	\left\{
		\begin{aligned}
		& \mathbf{G}|_{\partial D}+\mathbb{S}^{\lambda,\mu}_{D}\bm{\phi}=\mathbb{S}^{\widetilde{\lambda},\widetilde{\mu}}_{D}\bm{\psi}, \\
		&(\Lambda^{\widetilde{\lambda},\widetilde{\mu},\mathrm{i}}_{D}-\Lambda^{\lambda,\mu,\mathrm{e}}_{D})\mathbb{S}_{D}^{\lambda,\mu}\bm{\phi} =\left. \frac{\partial \mathbf{G}}{\partial \nu_{(\lambda,\mu)}}\right|_{\partial D}-\Lambda^{\widetilde{\lambda},\widetilde{\mu},\mathrm{i}}_{D}(\mathbf{G}|_{\partial D}).
	\end{aligned}\right.
\end{equation}
Because  $\mathbb{S}^{\lambda,\mu}_{D}$, $\mathbb{S}^{\widetilde{\lambda},\widetilde{\mu}}_{D}$ are isomorphisms from $H^{-\frac{1}{2}}(\partial D) $ to $H^{\frac{1}{2}}(\partial D)$ (by Proposition \ref{prop:SDclassical}.(iv)), and because  $\Lambda^{\widetilde{\lambda},\widetilde{\mu},\mathrm{i}}_{D}-\Lambda^{\lambda,\mu,\mathrm{e}}_{D}$ is an isomorphism from $H^{\frac{1}{2}}(\partial D)$ to $H^{-\frac{1}{2}}(\partial D) $ (by Proposition \ref{boundary operator isomorphism}), there exists a unique pair $(\text{\boldmath $\psi$}, \text{\boldmath $\phi$})\in H^{-1/2}(\partial D)\times H^{-1/2}(\partial D)$ solving equations \eqref{boundary integral equations III}. In view of the jump relation of $\mathcal{S}^{\lambda,\mu}_D \bm{\phi}$ and the second equation in \eqref{boundary integral equations}, we also get
\begin{equation*}
	\bm{\phi}=\left. \frac{\partial \mathcal{S}^{\widetilde{\lambda},\widetilde{\mu}}_{D}\bm{\psi}}{\partial \nu_{(\widetilde{\lambda},\widetilde{\mu})}}
	\right|_{-}
	-\left.   \frac{\partial \mathcal{S}^{\lambda,\mu}_{D}\bm{\phi}}{\partial \nu_{(\lambda,\mu)}}   \right|_{-}
	-\left.\frac{\partial \mathbf{G}}{\partial \nu_{(\lambda,\mu)}}\right|_{\partial D}.
\end{equation*}
For any $\mathbf{r} \in \mathbf{R}$ and for any component $E = D_i$ of $D$, we compute and get
\begin{equation*}
  \int_{\partial E} \bm{\phi}\cdot \mathbf{r} = \int_{\partial E} \left(\frac{\partial \mathcal{S}^{\widetilde \lambda,\widetilde \mu}_{D}\bm{\psi}}{\partial \nu_{(\widetilde \lambda,\widetilde\mu)}} -  \frac{\partial \mathcal{S}^{\lambda, \mu}_{D}\bm{\phi}}{\partial \nu_{(\lambda,\mu)}} - \mathbf{G}\right)\cdot \mathbf{r} = J^{E}_{\widetilde \lambda,\widetilde\mu}(\mathcal{S}^{\widetilde\lambda,\widetilde\mu}_D \bm{\psi},\mathbf{r}) -  J^{E}_{\lambda,\mu}(\mathcal{S}^{\lambda,\mu}_D\bm{\phi},\mathbf{r}) - J^{E}_{\lambda,\mu}(\mathbf{G},\mathbf{r}).
\end{equation*}
In the last step we applied the Green's identity \eqref{eq:GreenLame}. Note that for $\mathbf{r} \in \mathbf{R}$, $\mathrm{div}\,\mathbf{r} = \mathrm{tr} \bD(\mathbf{r}) = 0$; we deduce that the right hand side of the above equality vanishes. Hence, automatically, we get $\bm{\phi} \in H^{-\frac{1}{2}}_{\mathbf{R}}(\partial D)$. The proof of  Theorem \ref{thm:reptrans} is hence complete.

\begin{remark}
  \label{rem:orthR}
  The argument in the last paragraph actually shows, if $\mathbf{v}$ solves the Lam\'e system in $D$, the conormal derivative of $\mathbf{v}$ on $\partial D$ automatically belongs to $H^{-\frac12}_{\mathbf{R}}(\partial D)$. Similarly, this is also the case for any solution $(\mathbf{v},p)$ of the Stokes system in $D$ when the conormal derivative is understood as \eqref{eq:conormalStokes}. 
\end{remark}
\subsection{The solution to problem \eqref{eq:limStokes}} In this subsection we solve the problem \eqref{eq:limStokes} by single-layer potentials and prove Theorem \ref{thm:limStokes}. We use the ansatz
\begin{equation*}
		\mathbf{u}=\left\{
		\begin{aligned}
			&\mathbf{G}+\mathcal{S}^{\lambda,\mu}_{D} \text{\boldmath $\phi$} &\mathrm{in}\   \Omega\setminus \ol D \\
			&\mathcal{S}^{\infty,\widetilde{\mu}}_{D}  \text{\boldmath $\psi$}&\mathrm{in}\  D
		\end{aligned}	\right.  \quad \mathrm{and } \quad p=\mathcal{P}^{\widetilde{\mu}}_{D} \bm{\psi} \quad \mathrm{in} \ D
	\end{equation*}
	where the background solution $\mathbf{G}$ is defined in \eqref{background solution}. By the layer potential theory, solving \eqref{eq:limStokes} is equivalent to finding densities $\bm{\phi}$ and $\bm{\psi}$ on $\partial D$ so that $(\text{\boldmath $\psi$}, \text{\boldmath $\phi$})\in H^{-1/2}(\partial D)\times H_{\mathbf{R}}^{-1/2}(\partial D)$ and they satisfy the integral equations \eqref{boundary integral equations for stokes system}. 
		Again, using the DtN maps and thanks to Propositions \ref{DtN representation} and \ref{boundary operator isomorphism stokes}, we can rewrite this system as 
	\begin{equation}\label{boundary integral equations III stokes}
		\left\{
		\begin{aligned}
			& \mathbf{G}|_{\partial D}+\mathbb{S}^{\lambda,\mu}_{D}\bm{\phi}=\mathbb{S}^{\infty,\widetilde{\mu}}_{D}\bm{\psi}, \\
			&(\Lambda^{\infty,\widetilde{\mu},\mathrm{i}}_{D}-\Lambda^{\lambda,\mu,\mathrm{e}}_{D})\mathbb{S}_{D}^{\lambda,\mu}\bm{\phi} =\left. \frac{\partial \mathbf{G}}{\partial \nu_{(\lambda,\mu)}}\right|_{\partial D}-\Lambda^{\infty,\widetilde{\mu},\mathrm{i}}_{D}(\mathbf{G}|_{\partial D}).
		\end{aligned}\right.
	\end{equation}
	By item (iv) of Proposition \ref{prop:SDclassical}, $\mathbb{S}^{\lambda,\mu}_{D}$, $\mathbb{S}^{\widetilde{\lambda},\widetilde{\mu}}_{D}$ are isomorphisms from $H^{-\frac{1}{2}}(\partial D) $ to $H^{\frac{1}{2}}(\partial D)$, and by Proposition \ref{boundary operator isomorphism}, $\Lambda^{\widetilde{\lambda},\widetilde{\mu},\mathrm{i}}_{D}-\Lambda^{\lambda,\mu,\mathrm{e}}_{D}$ is an isomorphism from $H^{\frac{1}{2}}(\partial D)$ to $H^{-\frac{1}{2}}(\partial D)$. It follows that there exists a unique pair $(\text{\boldmath $\psi$}, \text{\boldmath $\phi$})\in H^{-1/2}(\partial D)\times H^{-1/2}(\partial D)$ solving equations \eqref{boundary integral equations III}. The jump relation of $\mathcal{S}^{\lambda,\mu}_D \bm{\phi}$ and the second equation in \eqref{boundary operator isomorphism stokes} then yield
	\begin{equation*}
	\bm{\phi}=\left.\frac{\partial (\mathcal{S}^{\infty,\widetilde{\mu}}_{D}\bm{\psi}, \mathcal{P}^{\widetilde{\mu}}_{D}\text{\boldmath $\psi$} )}{\partial \nu_{(\infty,\widetilde{\mu})}}\right|_{D-}
	-\left.   \frac{\partial\mathcal{S}^{\lambda,\mu}_{D}\bm{\phi}}{\partial \nu_{(\lambda,\mu)}}   \right|_{\partial D-}
	-\left.\frac{\partial \mathbf{G}}{\partial \nu_{(\lambda,\mu)}}\right|_{\partial D}.
\end{equation*}
In view of Remark \ref{rem:orthR}, we deduce that $\bm{\phi} \in H^{-\frac{1}{2}}_{\mathbf{R}}(\partial D)$. The proof of Theorem \ref{thm:limStokes} is then complete.

\subsection{The solution to problem \eqref{eq:limZero}} The layer potential representation for the solution of \eqref{eq:limZero} is most standard. The following result holds.
\begin{theorem}\label{representation formula for zero zero}
	Assume that {\upshape (A1)} holds. Then the problem \eqref{eq:limZero} has a unique solution $\mathbf{u}^{0,0}\in H^1(\Omega\setminus \ol D)$ and it is represented by the single-layer potential:
	\begin{equation}
		\mathbf{u}^{0,0}=\mathbf{G}+\mathcal{S}^{\lambda,\mu}_{D}\bm{\phi}\quad \mathrm{in}\  \Omega\setminus \ol D,
	\end{equation}
	and $\bm{\phi}\in H^{-\frac{1}{2}}_{\mathbf{R}}(\partial D)$ is the unique solution to
	\begin{equation}\label{equation 0,0}
		\left( \frac{\mathbb{I}}{2} +\mathbb{K}_{D}^{\lambda,\mu,*} \right)\bm{\phi} = \Lambda^{\lambda,\mu,{\rm e}}_D \mathbb{S}^{\lambda,\mu}_D \bm{\phi} = -\left.\frac{\partial \mathbf{G}}{\partial \nu_{(\lambda,\mu)}}\right|_{\partial D}.
	\end{equation}
\end{theorem}
\begin{proof}
	From Proposition \ref{prop:SDclassical}, $\frac{\mathbb{I}}{2} +\mathbb{K}_{D}^{\lambda,\mu,*}$ is an isomorphism on $H^{-\frac{1}{2}}_{\mathbf{R}}(\partial D)$, therefore \eqref{equation 0,0} is uniquely solvable, and the solution gives the desired representation.
\end{proof}

\subsection{The solution to problem \eqref{eq:limRigid}} 
We consider the ansatz 
\begin{equation}
\label{eq:uRigid}
	\mathbf{u}=\begin{cases} \mathbf{G}+\mathcal{S}^{\lambda,\mu}_{D} \bm{\phi}, \qquad &\text{in } \Omega \setminus D,\\
	\mathbf{w} \in \mathbf{R}, \qquad &\text{in } D.
	\end{cases}
\end{equation}
Then $\mathbf{u}$ solves \eqref{eq:limRigid} if and only if $\bm{\phi}\in H_{\mathbf{R}}^{-\frac{1}{2}}(\partial D)$, and $\mathbf{w} = \mathbf{G} + \mathbb{S}^{\lambda,\mu}_D\bm{\phi}$ on $\partial D$. Note that $\mathbf{w} \in \mathbf{R}$ in $D$ if and only if, for any admissible Lam\'e pair $(\lambda',\mu')$, $\mathbf{w}$ solves 
\begin{equation*}
\mathcal{L}_{\lambda',\mu'}\mathbf{w} = 0 \quad \text{in $D$}, \qquad \text{and}\qquad \frac{\partial \mathbf{w}}{\partial \nu_{(\lambda',\mu')}} = 0 \quad \text{on $\partial D$}.
\end{equation*}
Moreover, the above identities hold true for all admissible Lam\'e pairs as long as they hold for one such pair. Hence, a necessary and sufficient condition for $\bm{\phi}$ (so that \eqref{eq:uRigid} solves \eqref{eq:limRigid}) is 
\begin{equation}\label{last equation deterministic}
\Lambda^{\lambda',\mu',i}_D (\mathbb{S}^{\lambda,\mu}_D \bm{\phi} + \mathbf{G}\rvert_{\partial D}) = 0, \qquad \text{for some (and hence, for all) admissible}\; (\lambda',\mu').
\end{equation}
If we choose $(\lambda',\mu') = (\lambda, \mu)$, the above reduces to
\begin{equation}
\label{eq:urigidphi}
(-\frac{\mathbb{I}}{2} + \mathbb{K}^{\lambda,\mu,*}_D)\bm{\phi} = -\frac{\partial \mathbf{G}}{\partial \nu}.
\end{equation}
Note that $\mathcal{L}_{\lambda,\mu}\mathbf{G} = 0$ in $D$ implies that $\partial \mathbf{G}/\partial \nu \in H^{-\frac12}_{\mathbf{R}}(\partial D)$. Due to Proposition \ref{prop:SDclassical}, there exists a unique $\bm{\phi} \in H^{-\frac12}_{\mathbf{R}}(\partial D)$ that solves the equation above, and, consequently, \eqref{eq:uRigid} solves \eqref{eq:limRigid} and is the unique solution.

In view of the definition of the DtN operators, we easily check that
\begin{equation}
\label{eq:Lambdaiscale}
\Lambda^{\widetilde\lambda,\widetilde\mu,{\rm i}}_D = \widetilde\mu \Lambda^{\frac{\widetilde \lambda}{\widetilde \mu},1,{\rm i}}_D.
\end{equation}
To quantify the convergence of \eqref{eq:transmissionproblem} to \eqref{eq:limRigid} when $\widetilde\mu \to \infty$ while the other parameters are fixed, it is useful to view the right hand side above as $\widetilde\mu \Lambda^{\lambda',\mu',{\rm i}}_D$ with $\lambda' = \widetilde\lambda/\widetilde\mu$ and $\mu' = 1$. 

To summarize, we have proved the following result:
\begin{theorem}\label{repre lambda infty}
	The transmission problem \eqref{eq:limRigid} has a unique solution $\mathbf{u}\in H^1(\Omega)$ and is represented by \eqref{eq:uRigid}, where $\bm{\phi} \in H^{-\frac12}_{\mathbf{R}}(\partial D)$ is uniquely determined by \eqref{eq:urigidphi}. Moreover, 
	$\bm{\phi}$ satisfies
	\begin{equation}\label{boundary integral equation lambda infty 2}
	\Lambda^{\frac{\widetilde \lambda}{\widetilde \mu}, 1, \mathrm{i}}_{D} \mathbb{S}_D \bm{\phi} = - \Lambda^{\frac{\widetilde\lambda}{\widetilde\mu},1,{\rm i}}_D(\mathbf{G}|_{\partial D}).
	\end{equation}
\end{theorem}

The formulas we obtained in this section, for the problems \eqref{eq:transmissionproblem}, \eqref{eq:limStokes}, \eqref{eq:limZero} and \eqref{eq:limRigid}, are the starting points to establish quantitative convergence estimates; see \eqref{eq:disform} for explicit formulas. They work as long as $D$ satisfies (A1). When $D= D_\eps$ is a family satisfying (A2), to obtain convergence rates that are independent of $\eps$, we need to bound various operators uniformly, and this is the main task of the next section.
\section{Uniform bounds on layer-potential related operators} 
\label{sec:unifesti}

In this section, the inclusions set $D = D_\eps$ is assumed to satisfy the periodic structure specified in (A2). We establish uniform spectra gaps for the Neumann-Poincar\'e operator $\bK^{\lambda,\mu,*}_{D_\eps}$ with respect to the periodicity $\eps$, and uniform bounds for various operators and their inverse.

For notational simplifications, we omit the references to the Lam\'e pair $(\lambda,\mu)$ and to the domain $D_\eps$ in the notations of layer potential related operators, namely, $\cS,\mathbb{S},\mathbb{K}^{*}, \Lambda^{\rm e}$ (and Neumann function $\bm\Gamma^{\rm N}$ is used in the definition), as those parameters are always fixed in this section. On the other hand, when the Dirichlet function $\bm\Gamma^{\rm D}$ is used for the kernel of the single-layer potential, we denote the corresponding operators by $\cS^{\rm D}, \mathbb{S}^{\rm D}, \mathbb{K}^{{\rm D},*}$, etc. If other Lam\'e pair $(\lambda',\mu')$ are used instead, we will specify them in the notations of those operators.

We denote the Hilbert space $H^{-\frac12}(\partial D_\eps)$ equipped with the inner product $(\cdot,\cdot)_{\bS^{\rm N}}$ by $\cH$; note that the inner product is defined by \eqref{new inner product} using the background Lam\'e pair $(\lambda,\mu)$. Similarly, the subspace $H^{-\frac12}_\mathbf{R}(\partial D_\eps)$ is denoted by $\cHR$. 
\subsection{Uniform spectral gap of Neumann-Poincar\'{e} operator}

In this subsection we prove Theorem \ref{thm:specgap}, i.e., establishing uniform spectra gaps for the Neumann-Poincar\'e operators associated to the single-layer potential $\mathcal{S}^{\lambda,\mu}_D$. This is the key part of our paper, and our analysis is partially inspired by the works of \cite{MR2308861,bonnetier2019homogenization,bunoiu2020homogenization}. 

We study simultaneously the spectral properties of $\mathbb{K}^{\mathrm{N},*}$ and $\mathbb{K}^{\mathrm{D},*}$. As pointed out in Remark \ref{rem:NPselfadj}, the operator $\bK^{{\rm N},*}$ is self-adjoint as a bounded linear operator in $\cH$. It is well known that $\frac12$ is an eigen-value of $\bK^{{\rm N},*}$; in fact, $\mathrm{ker}(-\frac{\bI}{2}+\bK^{{\rm N},*})$ is characterized in \eqref{eq:HNdecom} and \eqref{eq:H-1/2decomp} and its dimension is $|\Pi_\eps|\times d(d+1)/2$ where $|\Pi_\eps|$ is the number of connected components in $D_\eps$. Note that, $|\Pi_\eps|$ increase to infinity as $\eps \to 0$.  

We argue first that to prove Theorem \ref{thm:specgap} it suffices to restrict $\bK^{{\rm N},*}$ to the subspace $\cHR$. Indeed, in view of the orthogonal decompositions \eqref{eq:HNdecom}  and \eqref {eq:H-1/2decomp}, $\ker(-\frac{\bI}{2} + \bK^{{\rm N},*})$ and $\cHR$ are orthogonal and invariant subspaces of $\cH$. Then if a number $\rho\in \mathbb{C}$ is in the resolvent of $\bK^{{\rm N},*}$ restricted to each of the two invariant subspaces, it is also in the resolvent of the whole operator. 


Now the restriction of $\bK^{{\rm N},*}$ in $\cHR$ remains self-adjoint and bounded. By standard theory, its spectra are real and contained in an interval $[m_\eps,M_\eps]$, the end points of which belong to the spectra and are determined by the extrema of the Rayleigh quotient
\begin{equation*}
\frac{(\bK^* \bm{\phi},\bm{\phi})_{\bS^{\rm N}}}{(\bm{\phi},\bm{\phi})_{\bS^{\rm N}}} = \frac{-\langle \bS\bm{\phi}, \bK^* \bm{\phi} \rangle_{H^{\frac12},H^{-\frac12}}}{-\langle \bS\bm{\phi},\bm{\phi}\rangle_{H^{\frac12},H^{-\frac12}}} = \frac{-\langle \bS\bm{\phi}, (-\frac{\bI}{2} + \bK^*) \bm{\phi} \rangle_{H^{\frac12},H^{-\frac12}}}{-\langle \bS\bm{\phi},\bm{\phi}\rangle_{H^{\frac12},H^{-\frac12}}} + \frac12 = \frac{J^{\Omega_{\varepsilon}}_{\lambda,\mu}(\cS \bm{\phi})}{J^\Omega_{\lambda,\mu}(\cS \bm{\phi})} - \frac12.
\end{equation*}
Recall that $\Omega_{\varepsilon}=\Omega \setminus D_{\varepsilon}$. More precisely, since $\cS^{\mathrm{N}} H^{-\frac12} = \frakH^{\rm N}$, we also have
\begin{equation}\label{def m M n}
	m_{\varepsilon}^{\mathrm{N}}:=\inf_{\mathbf{u}\in  \frakH^{\mathrm{N}}_{\mathbf{R}} \setminus \{0\}} \frac{J^{\Omega_{\varepsilon}}_{\lambda,\mu}(\mathbf{u})}{J^{\Omega}_{\lambda,\mu}(\mathbf{u})}-\frac{1}{2} , \qquad
	M_{\varepsilon}^{\mathrm{N}}:=\sup_{\mathbf{u}\in  \frakH^{\mathrm{N}}_{\mathbf{R}} \setminus \{0\}} \frac{J^{\Omega_{\varepsilon}}_{\lambda,\mu}(\mathbf{u})}{J^{\Omega}_{\lambda,\mu}(\mathbf{u})} -\frac{1}{2}.
\end{equation}
Similarly, the spectra of $\bK^{{\rm D},*}$ is real and contained in $[m^{\rm D}_\eps,M^{\rm D}_\eps]$ with
\begin{equation}\label{def m M d}
	m_{\varepsilon}^{\mathrm{D}}:=\inf_{\mathbf{u}\in  \frakH^{\mathrm{D}}_{\mathbf{R}} \setminus \{0\}} \frac{J^{\Omega_{\varepsilon}}_{\lambda,\mu}(\mathbf{u})}{J^{\Omega}_{\lambda,\mu}(\mathbf{u})}-\frac{1}{2} ,\qquad
	M_{\varepsilon}^{\mathrm{D}}:=\sup_{\mathbf{u}\in  \frakH^{\mathrm{D}}_{\mathbf{R}} \setminus \{0\}} \frac{J^{\Omega_{\varepsilon}}_{\lambda,\mu}(\mathbf{u})}{J^{\Omega}_{\lambda,\mu}(\mathbf{u})} -\frac{1}{2}.
\end{equation}
In view of Propositions \ref{prop:SDclassical} and \ref{prop:SDDirichlet}, $\pm \frac{\mathbb{I}}{2}+\mathbb{K}^{\mathrm{N},*}$ and $\pm \frac{\mathbb{I}}{2}+\mathbb{K}^{\mathrm{D},*}$ are isomorphisms on $H_{\mathbf{R}}^{-\frac{1}{2}}(\partial D)$, which implies that 
\begin{equation}
	-\frac{1}{2}<m_{\varepsilon}^{\mathrm{N}}\leq M_{\varepsilon}^{\mathrm{N}}<\frac{1}{2},\quad -\frac{1}{2}<m_{\varepsilon}^{\mathrm{D}}\leq M_{\varepsilon}^{\mathrm{D}}<\frac{1}{2}.
\end{equation}
So for each fixed $\eps > 0$, there is a gap between $-\frac12$ and $m^{\rm N}_\eps$, and a gap between $M^{\rm N}_\eps$ and $\frac12$. To show that those gaps are uniform in $\eps$, we need to explore deeper relations among the points $m_{\varepsilon}^{\mathrm{N}},M_{\varepsilon}^{\mathrm{N}},m_{\varepsilon}^{\mathrm{D}},M_{\varepsilon}^{\mathrm{D}}$.

Recall that $\mathbf{R}$ is the space of rigid motions in $\R^d$ and $\{\mathbf{r}_1,\cdots,\mathbf{r}_{d(d+1)/2}\}$ form a basis for $\mathbf{R}$; see \eqref{basis of R}. The following two lemmas (see sections \ref{appendix C} and \ref{appendix D} for proofs) will be useful. 

\begin{lemma}\label{appendix lemma 1}
	Assume that {\upshape(A2)} holds. Given $c_{\mathbf{n}}^j \in \mathbb{R}$ for $\mathbf{n}\in \Pi_{\varepsilon}$ and $j\in \{1,\cdots,\frac{d(d+1)}{2}\}$,
	and $\mathbf{g}\in H^{-\frac{1}{2}}(\partial \Omega)$ satisfying 
	\begin{equation}\label{concij}
		\sum_{\mathbf{n}\in \Pi_{\varepsilon}} c_{\mathbf{n}}^j =\int_{\partial \Omega} \mathbf{g}\cdot \mathbf{r}_j\quad \text{for all $j$}.
	\end{equation}
	Then there exists a unique solution $\mathbf{u} \in H^1(\Omega)$ for the problem
	\begin{equation}\label{ann}
		\left\{
		\begin{aligned}
			& \mathcal{L}_{\lambda,\mu} \mathbf{u} =0 \quad \text{in $\Omega_{\varepsilon}$}  \quad \text{and } \quad \mathbf{u} \in \mathbf{R} \quad \text{in each component of $D_{\varepsilon}$},\\
			& \mathbf{u}|_- =\mathbf{u}|_+ \quad \text{on $\partial D_{\varepsilon}$}, \\
			& \int_{\partial \omega^{\mathbf{n}}_{\varepsilon}} \left.\frac{\partial \mathbf{u}}{\partial \nu_{(\lambda,\mu)}} \right|_+\cdot \mathbf{r}_j=c_{\mathbf{n}}^j \quad 
\text{for all $\mathbf{n}$ and $j$},\\
			& \left. \frac{\partial \mathbf{u}}{\partial \nu_{(\lambda,\mu)}}\right|_{\partial \Omega}=\mathbf{g}, \quad \text{and} \quad \mathbf{u}|_{\partial \Omega}\in H^{\frac{1}{2}}_{\mathbf{R}}(\partial \Omega)
		\end{aligned}
		\right.
	\end{equation}
\end{lemma}

\begin{lemma}\label{ann dirichlet}
	Assume that {\upshape(A2)} holds. Given $c_{\mathbf{n}}^j \in \mathbb{R}$ for $\mathbf{n}\in \Pi_{\varepsilon}$ and $j\in \{1,\cdots,\frac{d(d+1)}{2}\}$, and $\mathbf{f}\in H^{\frac{1}{2}}(\partial \Omega)$, there exists a unique solution $\mathbf{u} \in H^1(\Omega)$ for the problem
	\begin{equation}\label{ann dirichlet equation}
		\left\{
		\begin{aligned}
			& \mathcal{L}_{\lambda,\mu} \mathbf{u} =0 \quad \text{in $\Omega_{\varepsilon}$}  \quad \text{and } \quad \mathbf{u} \in \mathbf{R} \quad \text{in each component of $D$},\\
			& \mathbf{u}|_- =\mathbf{u}|_+ \quad \text{on $\partial D_{\varepsilon}$}, \\
			& \int_{\partial\omega^{\mathbf{n}}_{\varepsilon}} \left.\frac{\partial \mathbf{u}}{\partial \nu_{(\lambda,\mu)}} \right|_+\cdot \mathbf{r}_j=c_{\mathbf{n}}^j, \quad
			\text{for all $\mathbf{n}$ and $j$},\\
			& \mathbf{u}|_{\partial \Omega}=\mathbf{f}
		\end{aligned}
		\right.
	\end{equation}
\end{lemma}

Next, we relax the definition of $m^{\rm N}_\eps$, establish comparisons of $m^{\rm N}_\eps$ with $m^{\rm D}_\eps$, and of $M^{\rm N}_\eps$ with $M^{\rm D}_\eps$.

\begin{proposition}\label{ao cri MN}
	The definition of $m_{\varepsilon}^{\mathrm{N}}$ can be relaxed to
	\begin{equation}
		m_{\varepsilon}^{\mathrm{N}}=\inf_{\mathbf{u}\in \mathcal{E}\setminus \{0\}} \frac{J^{\Omega_{\varepsilon}}_{\lambda,\mu}(\mathbf{u})}{J^{\Omega}_{\lambda,\mu}(\mathbf{u})} -\frac{1}{2}.
	\end{equation}
\end{proposition}

\begin{proof}
	Since $\frakH^{\rm N}\subset \mathcal{E}$, the right hand side is not larger than $m^{\rm N}_\eps$. To prove the other direction, take any $\mathbf{u} \in \mathcal{E}\setminus \{0\}$, let $\mathbf{v}$ solve \eqref{ann} with data $c_{\mathbf{n}}^j=\int_{\partial \omega^{\mathbf{n}}_{\varepsilon}} \left.\frac{\partial \mathbf{u}}{\partial \nu_{(\lambda,\mu)}} \right|_+\cdot \mathbf{r}_j$ and $\mathbf{g}=\left. \frac{\partial \mathbf{u}}{\partial \nu}\right|_{\partial \Omega}$, and define
	\begin{equation*}
		\mathbf{w} := \mathbf{v}+\sum_{j=1}^{\frac{d(d+1)}{2}}  \frac{\int_{\partial \Omega} \mathbf{u}\cdot \mathbf{r}_j }{\int_{\partial \Omega}|\mathbf{r}_j|^2 } \mathbf{r}_j.
	\end{equation*}
	Then it is easy to verify that $\mathbf{u}-\mathbf{w}\in \frakH^{\mathrm{N}}_{\mathbf{R}}$ and $\mathbf{w} \in \mathbf{R}$ in each component of $D_\eps$. It follows that $J^{\Omega}_{\lambda,\mu}(\mathbf{u}-\mathbf{w} ,\mathbf{w})=0$ and $J_{\lambda,\mu}^{D_\eps}(\mathbf{h},\mathbf{w}) = 0$ for all $\mathbf{h}\in H^1(D_\eps)$. As a result,
	\begin{equation*}
		\frac{J^{\Omega_{\varepsilon}}_{\lambda,\mu}(\mathbf{u})}{J^{\Omega}_{\lambda,\mu}(\mathbf{u})} = 1- \frac{J^{D_{\varepsilon}}_{\lambda,\mu}(\mathbf{u})}{J^{\Omega}_{\lambda,\mu}(\mathbf{u})} =1- \frac{J^{D_{\varepsilon}}_{\lambda,\mu}(\mathbf{u}-\mathbf{w})}{J^{\Omega}_{\lambda,\mu}(\mathbf{u}-\mathbf{w})+J^{\Omega}_{\lambda,\mu}(\mathbf{w})} \geq \frac{J^{\Omega_{\varepsilon}}_{\lambda,\mu}(\mathbf{u}-\mathbf{w})}{J^{\Omega}_{\lambda,\mu}(\mathbf{u}-\mathbf{w})}\geq m_{\varepsilon}^{\mathrm{N}}+\frac{1}{2},
	\end{equation*}
	Take the infimum over $\mathbf{u}$, we get the desired result.
\end{proof}

\begin{proposition} \label{quan m M} The following comparison results hold:
	\begin{equation}
	\label{eq:endpoints}
		m_{\varepsilon}^{\mathrm{N}}\leq m_{\varepsilon}^{\mathrm{D}} ,\quad M_{\varepsilon}^{\mathrm{N}}\leq M_{\varepsilon}^{\mathrm{D}}.
	\end{equation}
\end{proposition}

\begin{proof}
For the second inequality, take an arbitrary nonzero $\mathbf{u} \in \frakH^{\mathrm{N}}_{\mathbf{R}}$, let $\mathbf{v}\in H^1(\Omega)$ solve \eqref{ann dirichlet equation} with data $c_{\mathbf{n}}^j=0$ and $\mathbf{f}=-\mathbf{u}|_{\partial \Omega}$. Then we verify easily that $\mathbf{u}+\mathbf{v}\in \frakH^{\mathrm{D}}_{\mathbf{R}} $, and $\mathbf{v} \in \mathbf{R}$ in each component of $D_\eps$. The last point and $\partial \mathbf{u}/\partial \nu\rvert_{\partial \Omega} = 0$ also imply that $J^{\Omega}_{\lambda,\mu}(\mathbf{u}, \mathbf{v})=0$ (they are orthogonal with respect to this inner product). Hence,
		\begin{equation}\label{m d inff}
			\frac{J^{\Omega_{\varepsilon}}_{\lambda,\mu}(\mathbf{u})}{J^{\Omega}_{\lambda,\mu}(\mathbf{u})}= 1- \frac{J^{D_{\varepsilon}}_{\lambda,\mu}(\mathbf{u}+\mathbf{v})}{J^{\Omega}_{\lambda,\mu}(\mathbf{u})}\leq 1- \frac{J^{D_{\varepsilon}}_{\lambda,\mu}(\mathbf{u}+\mathbf{v})}{J^{\Omega}_{\lambda,\mu}(\mathbf{u})+J^{\Omega}_{\lambda,\mu}(\mathbf{v})}=\frac{J^{\Omega_{\varepsilon}}_{\lambda,\mu}(\mathbf{u}+\mathbf{v})}{J^{\Omega}_{\lambda,\mu}(\mathbf{u}+\mathbf{v})}\leq M_{\varepsilon}^{\mathrm{D}} + \frac12.
		\end{equation}
		Take the supremum over $\mathbf{u}$, we obtain $M^{\rm N}_\eps \le M^{\rm D}_\eps$.

	The proof for the first inequality in \eqref{eq:endpoints} is similar. Taken any nonzero $\mathbf{u} \in \frakH^{\mathrm{D}}_{\mathbf{R}}$, let $\mathbf{v}\in H^1(\Omega)$ be the solution to \eqref{ann} with $c_{\mathbf{n}}^j=\int_{\partial \omega^{\mathbf{n}}_{\varepsilon}} \left.\frac{\partial \mathbf{u}}{\partial \nu_{(\lambda,\mu)}} \right|_+\cdot \mathbf{r}_j$ and $\mathbf{g}=\left. \frac{\partial \mathbf{u}}{\partial \nu}\right|_{\partial \Omega}$. Then we verify easily that $\mathbf{u}-\mathbf{v}\in \frakH^{\mathrm{N}}_{\mathbf{R}}$ and $\mathbf{v}\in \mathbf{R}$ in each component of $D_\eps$. It follows that $J^{\Omega}_{\lambda,\mu}(\mathbf{u}-\mathbf{v},\mathbf{v})=0$, and
	\begin{equation}\label{M N sss}
		\frac{J^{\Omega_{\varepsilon}}_{\lambda,\mu}(\mathbf{u})}{J^{\Omega}_{\lambda,\mu}(\mathbf{u})} = 1- \frac{J^{D_{\varepsilon}}_{\lambda,\mu}(\mathbf{u})}{J^{\Omega}_{\lambda,\mu}(\mathbf{u})}= 1-\frac{J^{D_{\varepsilon}}_{\lambda,\mu}(\mathbf{u}-\mathbf{v})}{J^{\Omega}_{\lambda,\mu}(\mathbf{u}-\mathbf{v}) + J^\Omega(\mathbf{v})} \geq \frac{J^{\Omega_{\varepsilon}}_{\lambda,\mu}(\mathbf{u}-\mathbf{v})}{J^{\Omega}_{\lambda,\mu}(\mathbf{u}-\mathbf{v})}\geq m_{\varepsilon}^{\mathrm{N}} + \frac12.
	\end{equation}
	Take the infimum over $\mathbf{u}$ and we obtain the desired inequality.
\end{proof}

Now we are ready to establish the gaps between $-\frac12$ and $m^{\rm N}_\eps$ and between $M^{\rm N}_\eps$ and $\frac12$. Combined with the arguments at the beginning of this section, this also completes the proof of Theorem \ref{thm:specgap}.

\begin{theorem}\label{unif m M}
Assume that {\upshape(A2)} holds. Then:
	\begin{equation}
	\label{eq:gapMm}
			-\frac12 < \inf_{\varepsilon >0} m_{\varepsilon}^{\mathrm{N}} \le  \inf_{\varepsilon >0} m_{\varepsilon}^{\mathrm{D}}, \qquad \text{and} \qquad \sup_{\varepsilon >0} M_{\varepsilon}^{\mathrm{N}} \le \sup_{\varepsilon >0} M_{\varepsilon}^{\mathrm{D}} <\frac{1}{2}.
	\end{equation}
\end{theorem}

Note that in the proof below, we actually provides rather explicit bounds for $\inf\{m^{\rm N}_\eps\}$ and for $\sup\{M^{\rm D}_\eps\}$. They are given by the NP operators associated to the model inclusion $\omega$ inside the cell $Y$ at the unit scale.
\begin{proof}
	Thanks to Proposition \ref{quan m M}, we only need to establish a uniform lower bound for $\{m_{\varepsilon}^{\mathrm{N}}\}$ and an upper bound for $\{M_{\varepsilon}^{\mathrm{D}}\}$. Note also, for each fixed $\eps > 0$, by the setting of $D = D_\eps$, each component $\omega^{\mathbf{n}}_\eps$, $\mathbf{n}\in \Pi_\eps$, is contained in an $\eps$-cube $Y^{\mathbf{n}}_\eps$ that is completely contained in $\Omega$. The cushion area $\Omega\setminus (\cup_{\mathbf{n}\in \Pi_\eps} \ol Y^\mathbf{n}_\eps)$ is denoted by $K_\eps$ below, and $Y_\eps$ denotes the remaining part $\Omega\setminus \ol K_\eps$.

	To establish an upper bound of $\{M^{\rm D}_\eps\}$, take an arbitrary nonzero $\mathbf{u} \in \frakH^{\rm D}_\mathbf{R}$, and construct a function $\mathbf{v}$ piecewise, as follows: let $\mathbf{v} = -\mathbf{u}$ in $K_\eps$, and on each $\eps$-cube $Y^\mathbf{n}_\eps$, let $\mathbf{v}$ solve
	\begin{equation}
		\mathcal{L}_{\lambda,\mu}\mathbf{v} = 0 \; \text{in $Y^\mathbf{n}_\eps \setminus \bar{\omega}_{\varepsilon}^{\mathbf{n}}$ and }\; \mathbf{v} \in \mathbf{R} \;\text{in $\ol\omega^\mathbf{n}_\eps$}, \quad \mathbf{v}\rvert_{\partial Y^\mathbf{n}_\eps} = -\mathbf{u}\rvert_{\partial Y^\mathbf{n}_\eps}, \quad \frac{\partial \mathbf{v}}{\partial \nu}\Big\rvert_+ = 0 \; \text{on $\partial \omega^\mathbf{n}_\eps$}.
	\end{equation}
	The existence and uniqueness of $\mathbf{v}$ in each $\eps$-cube is guaranteed by an application of Lemma \ref{ann dirichlet}. By construction, $\mathbf{v}\in H^1_0(\Omega)$ (since $\frakH^{\rm D}_\mathbf{R}\subset H^1_0(\Omega)$), $\mathbf{v} + \mathbf{u}$ vanishes in $K_\eps$, and $\mathbf{v} \in \mathbf{R}$ in each components of $D_\eps$. It follows that $J_{\lambda,\mu}^{\Omega}(\mathbf{u}, \mathbf{v}) = J_{\lambda,\mu}^{\Omega_\eps}(\mathbf{u},\mathbf{v})=0$, so
	\begin{equation}
	\label{eq:comp1}
		\frac{J^{D_{\varepsilon}}_{\lambda,\mu}(\mathbf{u})}{J^{\Omega}_{\lambda,\mu}(\mathbf{u})}=\frac{J^{D_{\varepsilon}}_{\lambda,\mu}(\mathbf{u}+\mathbf{v})}{J^{\Omega}_{\lambda,\mu}(\mathbf{u})}\geq \frac{J^{D_{\varepsilon}}_{\lambda,\mu}(\mathbf{u}+\mathbf{v})}{J^{\Omega}_{\lambda,\mu}(\mathbf{u})+J^{\Omega}_{\lambda,\mu}(\mathbf{v})}=\frac{J^{D_{\varepsilon}}_{\lambda,\mu}(\mathbf{u}+\mathbf{v})}{J^{\Omega}_{\lambda,\mu}(\mathbf{u}+\mathbf{v} )}=
		\frac{\sum_{\mathbf{n}\in \Pi_{\varepsilon} } J_{\lambda,\mu}^{ \omega_{\varepsilon}^{\mathbf{n}} }(\mathbf{u}+\mathbf{v})}{\sum_{\mathbf{n}\in \Pi_{\varepsilon} } J_{\lambda,\mu}^{Y_{\varepsilon}^{\mathbf{n}}}(\mathbf{u}+\mathbf{v})}.
	\end{equation}
	Observe the following elementary inequality: for any positive integer $N$, positive numbers $\{b_i\}_{i=1}^N$ and non-negative numbers $\{a_i\}_{i=1}^N$, it holds 
	\begin{equation}
	\label{eq:elemineq}
	\min_{1\le i\le N} \frac{a_i}{b_i} \le \frac{\sum_{i=1}^N a_i}{\sum_{i=1}^N b_i}\leq \max_{1\le i \le N}\frac{a_i}{b_i}.
	\end{equation}
	By inspecting we see the very right hand side of \eqref{eq:comp1} is a quotient of the form $(\sum_i a_i)/(\sum_i b_i)$. Moreover, we may assume $J^{Y^{\mathbf{n}}_\eps}_{\lambda,\mu}(\mathbf{u}) > 0$ for all $\mathbf{n}$ because, if this fails for some $\mathbf{n}$, the corresponding term on the nominator, i.e., $J^{\omega^\mathbf{n}_\eps}_{\lambda,\mu}(\mathbf{u})$ also vanishes. Apply the above inequality to \eqref{eq:comp1}, we get
	\begin{equation*}
		\frac{J^{D_{\varepsilon}}_{\lambda,\mu}(\mathbf{u})}{J^{\Omega}_{\lambda,\mu}(\mathbf{u})} \geq  \min_{\mathbf{n}\in \Pi_{\varepsilon}} \frac{ J_{\lambda,\mu}^{ \omega_{\varepsilon}^{\mathbf{n}} }(\mathbf{u}+\mathbf{v})}{ J_{\lambda,\mu}^{ Y_{\varepsilon}^{\mathbf{n}} }(\mathbf{u}+\mathbf{v})}.
	\end{equation*}
	For each $\mathbf{n} \in \Pi_\eps$, we observe that $\mathbf{u} + \mathbf{v}$ belongs to $\frakH^{\rm D}_\mathbf{R}(Y^\mathbf{n}_\eps,\omega^\mathbf{n}_\eps)$, the space defined by \eqref{eq:HDspace} with $\Omega$ replaced by $Y^\mathbf{n}_\eps$ and $D$ replaced by $\omega^\mathbf{n}_\eps$. We hence get
	\begin{equation}
	\label{eq:comp2}
		\frac{ J_{\lambda,\mu}^{ \omega_{\varepsilon}^{\mathbf{n}} }(\mathbf{u}+\mathbf{v})}{ J_{\lambda,\mu}^{ Y_{\varepsilon}^{\mathbf{n}} }(\mathbf{u}+\mathbf{v})} \ge \inf_{\mathbf{w}\in \frakH^{\mathrm{D}}_{\mathbf{R}}(Y^\mathbf{n}_\eps,\omega^\mathbf{n}_\eps)\setminus\{0\}} \frac{J^{\omega^\mathbf{n}_\eps}_{\lambda,\mu}(\mathbf{w})}{J^{Y^\mathbf{n}_\eps}_{\lambda,\mu}(\mathbf{w})} = \inf_{\mathbf{w}\in \frakH^{\mathrm{D}}_{\mathbf{R}}(Y,\omega)\setminus\{0\}} \frac{J^{\omega}_{\lambda,\mu}(\mathbf{w})}{J^Y_{\lambda,\mu}(\mathbf{w})}.
	\end{equation}
	The equality above holds because the translation and rescaling transform $\mathbf{w}(\cdot) \to \widetilde{\mathbf{w}}((\cdot-\mathbf{n})/\eps)$ is an isomorphism between $\frakH^{\rm D}(Y,\omega)$ and $\frakH^{\rm D}(Y^\mathbf{n}_\eps,\omega^\mathbf{n}_\eps)$, and this transform leaves the quotient unchanged. Similar to the definitions in \eqref{def m M d}, the very right hand side of \eqref{eq:comp2} is exactly $1/2 - M$, where $M = M^{\rm D}(\omega)$ is the maximum point in the spectra of $\bK^{{\rm D},*}_{\omega}$ acting on $(H^{-\frac12}_\mathbf{R}(\partial \omega),(\cdot,\cdot)_{\bS^{\rm D}})$. Since $-\frac{\bI}{2} + \bK^{{\rm D},*}_{\omega}$ is an isometry on $H^{-\frac12}_\mathbf{R}(\partial \omega)$, we get $M < \frac12$. It is important to note that $M$ depends on $(\lambda,\mu)$, the dimension $d$ and the model set $\omega$, but is independent of $\eps$. In other words, $M < \frac12$ is a universal constant. We then verify that
\begin{equation}
\label{eq:comp3}
		\frac{J^{\Omega_{\varepsilon}}_{\lambda,\mu}(\mathbf{u})}{J^{\Omega}_{\lambda,\mu}(\mathbf{u})} = 1-\frac{J^{D_{\varepsilon}}_{\lambda,\mu}(\mathbf{u})}{J^{\Omega}_{\lambda,\mu}(\mathbf{u})} \le M^{\rm D}(\omega) + \frac12, \qquad \forall \mathbf{u} \in \frakH^{\rm D}_\mathbf{R} \setminus \{0\}, \; \forall \eps > 0.
	\end{equation}
	It follows immediately that for all $\eps > 0$, $M^{\rm D}_\eps \le M^{\rm D}(\omega) < \frac12$.
	
	The proof of the lower bound of $\{m^{\rm N}_\eps\}$ is easier. For any fixed $\varepsilon>0$ and nonzero $\mathbf{u} \in\frakH^{\mathrm{N}}_{\mathbf{R}}$, we easily get 
	\begin{equation*}
		\frac{J^{D_{\varepsilon}}_{\lambda,\mu}(\mathbf{u})}{J^{\Omega}_{\lambda,\mu}(\mathbf{u})}\leq \frac{J^{D_{\varepsilon}}_{\lambda,\mu}(\mathbf{u})}{J^{Y_{\varepsilon}}_{\lambda,\mu}(\mathbf{u})}=\frac{\sum_{\mathbf{n}\in \Pi_{\varepsilon} } J_{\lambda,\mu}^{ \omega_{\varepsilon}^{\mathbf{n}} }(\mathbf{u})}{\sum_{\mathbf{n}\in \Pi_{\varepsilon} } J_{\lambda,\mu}^{Y_{\varepsilon}^{\mathbf{n}}}(\mathbf{u})},
	\end{equation*}
	Applying the inequality \eqref{eq:elemineq} (the second half) to the last term above, and repeating the rescaling arguments used earlier, we obtain
	\begin{equation*}
		\frac{J^{D_{\varepsilon}}_{\lambda,\mu}(\mathbf{u})}{J^{\Omega}_{\lambda,\mu}(\mathbf{u})} \leq   \max_{\mathbf{n}\in \Pi_{\varepsilon}} \frac{ J_{\lambda,\mu}^{ \omega_{\varepsilon}^{\mathbf{n}} }(\mathbf{u})}{ J_{\lambda,\mu}^{ Y_{\varepsilon}^{\mathbf{n}} }(\mathbf{u})} \le \sup_{\mathbf{w}\in \mathcal{E}\setminus \{0\}} \frac{J^{\omega}_{\lambda,\mu}(\mathbf{w})}{J^Y_{\lambda,\mu}(\mathbf{w})}.
	\end{equation*}
	Here $\mathcal{E}$ is defined as in \eqref{eq:Espace} with $\Omega$ replaced by $Y$ and $D$ replaced by $\omega$. It worths mentioning that $\mathcal{E}$ is indeed the right place for $\mathbf{u}$ restricted to each $\eps$-cube, since no further information at the boundaries of the cube and the inclusion is available. Nevertheless, due to Proposition \ref{ao cri MN}, the term on the very right hand side above is precisely $\frac12 - m$, where $m = m^{\rm N}(\omega)$ is the minimum element in the spectra of $-\frac{\bI}{2}+\bK^{{\rm N},*}_\omega$ acting on $(H^{-\frac12}(\partial \omega),(\cdot,\cdot)_{\bS^{\rm N}})$. In particular, $m>-\frac12$ and is a universal constant. We hence have proved that
	\begin{equation}
	\label{eq:comp4}
		\frac{J^{\Omega_{\varepsilon}}_{\lambda,\mu}(\mathbf{u})}{J^{\Omega}_{\lambda,\mu}(\mathbf{u})} = 1-\frac{J^{D_{\varepsilon}}_{\lambda,\mu}(\mathbf{u})}{J^{\Omega}_{\lambda,\mu}(\mathbf{u})} \ge m^{\rm N}(\omega) + \frac12, \qquad \forall \mathbf{u} \in \frakH^{\rm N}_\mathbf{R} \setminus \{0\}, \; \forall \eps > 0.
	\end{equation}
	It follows immediately that for all $\eps > 0$, $m^{\rm N}_\eps \ge m^{\rm N}(\omega) > \frac12$. This completes the proofs.
\end{proof}

The following results, which establish comparisons of the energies distributed in $D_\eps$ and $\Omega_\eps$ for functions in $\mathcal{E}$, will be useful later. 
\begin{corollary}
	Assume that {\upshape(A2)} holds. Then there exists a universal constant $C > 0$ such that
	\begin{equation}\label{uniform estimate 1}
		J^{\Omega_{\varepsilon}}_{\lambda,\mu}(\mathbf{u}) \leq CJ_{\lambda,\mu}^{D_{\varepsilon}}(\mathbf{u}) \quad \mathrm{for\ any}\ \mathbf{u}\in \frakH^{\mathrm{D}}_{\mathbf{R}}\cup \frakH^{\mathrm{N}}_{\mathbf{R}}.
	\end{equation} 
and
\begin{equation}\label{uniform estimate 2}
	J^{ D_{\varepsilon}}_{\lambda,\mu}(\mathbf{v}) \leq CJ_{\lambda,\mu}^{\Omega_{\varepsilon}}(\mathbf{v}) \quad \mathrm{for\ any}\ \mathbf{v}\in \mathcal{E}.
\end{equation}
Clearly, replacing $\Omega_\eps$ by $\Omega$ in those inequalities, the results still hold.
\end{corollary}
\begin{proof} Those inequalities are direct consequences of Theorem \ref{unif m M}. More precisely, \eqref{uniform estimate 1} follows from the uniform upper bound of $M^{\rm D}_\eps$ and $M^{\rm N}_\eps$, and \eqref{uniform estimate 2} follows from the uniform lower bound of $m^{\rm D}_\eps$ and $m^{\rm N}_\eps$ and the characterization in Proposition \ref{ao cri MN}.
\end{proof}

\subsection{Uniform estimates for related operators} In this subsection, we establish uniform (in $\eps$) bounds for some operators related to the single-layer potentials defined using the Neumann function $\bm{\Gamma}^{\rm N}$, and reference to the Lam\'e coefficients is specified when they are different from the background one, i.e., $(\lambda,\mu)$. 

\begin{proposition}\label{prop:Lambdae}
	Assume that {\upshape(A2)} holds. 
	Then $\frac{\mathbb{I}}{2}+\mathbb{K}_{D_{\varepsilon}}^{\lambda,\mu,*} : \cH \to \cH$ is a strictly positive bounded self-adjoint operator. Furthermore, there is a universal constant $C>0$, such that
\begin{equation}
\label{eq:DtNebdd}
	\left\| (\frac{\mathbb{I}}{2}+\mathbb{K}_{D_{\varepsilon}}^{\lambda,\mu,*})^{-1}  \right\|_{\mathscr{L}(\mathcal{H})}
	\leq C.
\end{equation}
\end{proposition}
\begin{proof}
	Denote $\bK^{\lambda,\mu,*}_{D_\eps}$ by $\bK^*$. We have seen that $\frac{\bI}{2}+\bK^{*}$ is self-adjoint and bounded. To show positivity and invertibility, fix any $\bm{\phi}\in H^{-\frac{1}{2}}(\partial D_{\varepsilon})$ and define $\mathbf{w} = \cS\bm{\phi}$. Using the definition of $(\cdot,\cdot)_{\mathbb{S}^{\mathrm{N}}}$ and the uniform bound \eqref{uniform estimate 2}, we compute and check that, for some $C>0$ depending only on $\Omega,\omega,d,\lambda,\mu$ but independent of $\eps$, 
	\begin{equation*}
	\begin{aligned}
		((\frac{\mathbb{I}}{2}+\mathbb{K}^{*})\bm{\phi},\bm{\phi} )_{\mathbb{S}^{\mathrm{N}}} &= -\int_{\partial D_\eps} (\frac{\bI}{2}+\bK^*)\bm{\phi} \cdot \cS \bm{\phi} = -\int_{\partial D_\eps} \frac{\partial \mathbf{w}}{\partial \nu} \Big\rvert_{+} \cdot \mathbf{w} \\
		&=J^{\Omega_{\varepsilon}}_{\lambda,\mu}(\mathcal{S} \bm{\phi}) \geq CJ^{\Omega}_{\lambda,\mu}(\mathcal{S} \bm{\phi})=C\|\bm{\phi}\|_{\mathbb{S}^{\mathrm{N}}}^2.
		\end{aligned}
	\end{equation*}
By Lemma \ref{lem:LaxMilgram}, the desired conclusion follows.
\end{proof}

\begin{proposition}\label{inverse boundedness}
	Assume that {\upshape(A2)} holds. We have the following:
	\begin{enumerate}
		\item [{\upshape (i)}] Suppose $(\lambda',\mu')$ is admissible, then the operator 
			$\Lambda_{D_{\varepsilon}}^{\lambda',\mu',\mathrm{i}}\mathbb{S}_{D_{\varepsilon}}^{\lambda,\mu} : \cH \to \cH$
		is a non-positive self-adjoint operator. Moreover, if $(\lambda',\mu')$ is uniformly admissible, i.e., satisfying \eqref{eq:unifadm}, then  there exists a universal constant $C>0$ such that
		\begin{equation}
			\left\|  \Lambda_{D_{\varepsilon}}^{\lambda',\mu',\mathrm{i}}\mathbb{S}_{D_{\varepsilon}}^{\lambda,\mu}  \right\|_{\mathscr{L}(\cH)} \leq C.
		\end{equation}
	\item [{\upshape (ii)}] Suppose $(\lambda',\mu')$ is admissible, then the restriction of   
$\Lambda_{D_{\varepsilon}}^{\lambda',\mu',\mathrm{i}}\mathbb{S}_{D_{\varepsilon}}^{\lambda,\mu}$ to $\cHR$ is a strictly negative self-adjoint linear transformation on $\cHR$. Moreover, if $(\lambda',\mu')$ is uniformly admissible, then there exists a universal constant $C>0$ such that
	\begin{equation}
		\left\| ( \Lambda_{D_{\varepsilon}}^{\lambda',\mu',\mathrm{i}}\mathbb{S}_{D_{\varepsilon}}^{\lambda,\mu} )^{-1} \right\|_{\mathscr{L}(\cHR)} \leq C.
	\end{equation}
	\end{enumerate}
\end{proposition}
\begin{proof}
	For any fixed $\bm{\phi},\bm{\psi}\in H^{-\frac{1}{2}}(\partial D_{\varepsilon})$, define
	\begin{equation*}
	\mathbf{v}: =\mathcal{S}_{D_{\varepsilon}}^{\lambda,\mu}\bm{\phi}, \quad \mathbf{u} := \mathcal{S}_{D_{\varepsilon}}^{\lambda',\mu'}(\mathbb{S}_{D_{\varepsilon}}^{\lambda',\mu'})^{-1}\mathbb{S}_{D_{\varepsilon}}^{\lambda,\mu}\bm{\phi}, \quad \text{and} \quad \mathbf{w} := \mathcal{S}_{D_{\varepsilon}}^{\lambda',\mu'}(\mathbb{S}_{D_{\varepsilon}}^{\lambda',\mu'})^{-1}\mathbb{S}_{D_{\varepsilon}}^{\lambda,\mu}\bm{\psi}.
	\end{equation*}
	They all belong to $\frakH^{\rm N}$ and satisfy $\mathbf{v} = \mathbf{u}$ on $\partial D_\eps$. By the Dirichlet principle \eqref{eq:diriprin} (first apply it to $D_\eps$ and to $\Omega_\eps$, respectively, and then combine the results), we have 
	\begin{equation*}
	J^{\Omega}_{\lambda',\mu'}(\mathbf{u})\leq  J^{\Omega}_{\lambda',\mu'}(\mathbf{v}), \qquad \text{and}\quad  J^{\Omega}_{\lambda,\mu}(\mathbf{v})\leq J^{\Omega}_{\lambda,\mu}(\mathbf{u}).
	\end{equation*}
	Combine those with the second Korn's inequality, we get 
\begin{equation*}
	\|\mathbf{u}\|_{H^1(\Omega)}^2 \sim J^{\Omega}_{\lambda',\mu'}(\mathbf{u}) \sim J^{\Omega}_{\lambda,\mu}(\mathbf{u}) \sim J^{\Omega}_{\lambda,\mu}(\mathbf{v}) \sim J^{\Omega}_{\lambda',\mu'}(\mathbf{v})\sim \|\mathbf{v}\|_{H^1(\Omega)}^2.
\end{equation*}
Here, the relation $\sim$ means each side can be bounded from above by the other side multiplied by a universal constant $C > 0$. The above also implies
\begin{equation}
\label{eq:uvcomp}
\|\bm{\phi}\|_{\bS^{\rm N}}^2 \sim J^\Omega_{\lambda',\mu'}(\mathbf{u}), \qquad \|\bm{\psi}\|_{\bS^{\rm N}}^2 \sim J^\Omega_{\lambda',\mu'}(\mathbf{w}).
\end{equation}
\noindent\emph{Proof of {\upshape(i)}.} By the formulas of DtN maps in \eqref{eq:elastDtN}, we see that the range of $\Lambda^{\lambda',\mu',{\rm i}}$ is contained in $H^{-\frac12}_\mathbf{R}(\partial D_\eps)$. Clearly, $\Lambda^{\lambda',\mu',{\rm i}}\bS^{\lambda,\mu}\bm{\phi} = \frac{\partial \mathbf{u}}{\partial \nu}\rvert_-$ on $\partial D_\eps$, $\mathbf{w}\rvert_{\partial D_\eps} = \bS^{\lambda,\mu}\bm{\psi}$. By the Green's identity, 
	\begin{equation*}
		-(\Lambda_{D_{\varepsilon}}^{\lambda',\mu',\mathrm{i}}\mathbb{S}_{D_{\varepsilon}}^{\lambda,\mu}\bm{\phi},\bm{\psi})_{\mathbb{S}^{\mathrm{N}}} = \int_{\partial D_\eps} \frac{\partial \mathbf{u}}{\partial \nu}\Big\rvert_{-} \cdot \mathbf{w} 
		= J^{D_{\varepsilon}}_{\lambda',\mu'}(\mathbf{u},\mathbf{w}).
	\end{equation*}
 	Hence, $\Lambda_{D_{\varepsilon}}^{\lambda',\mu',\mathrm{i}}\mathbb{S}_{D_{\varepsilon}}^{\lambda,\mu}$ is non-positive and self-adjoint. Using the Cauchy-Schwarz inequality and \eqref{eq:uvcomp}, we get
	\begin{equation*}
		\left| J^{D_{\varepsilon}}_{\lambda',\mu'}(\mathbf{u},\mathbf{w})\right| \le \left(J^{D_\eps}_{\lambda',\mu'}(\mathbf{u})\right)^{\frac12} \left(J^{D_\eps}_{\lambda',\mu'}(\mathbf{w})\right)^{\frac12} \le  \left(J^{\Omega}_{\lambda',\mu'}(\mathbf{u})\right)^{\frac12} \left(J^{\Omega}_{\lambda',\mu'}(\mathbf{w})\right)^{\frac12} \le C\|\bm{\phi}\|_{\bS^{\rm N}} \|\bm{\psi}\|_{\bS^{\rm N}}.
	\end{equation*}
	It follows that $\Lambda_{D_{\varepsilon}}^{\lambda',\mu',\mathrm{i}}\mathbb{S}_{D_{\varepsilon}}^{\lambda,\mu}$ is uniformly bounded.
	
\noindent\emph{Proof of {\upshape (ii)}.}
	We have shown that the range of $\Lambda_{D_{\varepsilon}}^{\lambda',\mu',\mathrm{i}}\mathbb{S}_{D_{\varepsilon}}^{\lambda,\mu}$ is contained in $\cHR$, so this operator is a bounded self-adjoint linear transformation on $\cHR$. For any $\bm\phi \in \cHR$, using \eqref{uniform estimate 1} and \eqref{eq:uvcomp} we compute and verify
	\begin{equation*}
		-(\Lambda_{D_{\varepsilon}}^{\lambda',\mu',\mathrm{i}}\mathbb{S}_{D_{\varepsilon}}^{\lambda,\mu}\bm{\phi},\bm{\phi})_{\mathbb{S}^{\mathrm{N}}}=J^{D_{\varepsilon}}_{\lambda',\mu'}(\mathbf{u}) \geq CJ^{\Omega}_{\lambda',\mu'}(\mathbf{u}) \ge C\|\bm{\phi}\|_{\mathbb{S}^{\mathrm{N}}}^2.
	\end{equation*}
	Lemma \ref{lem:LaxMilgram} then yields the uniform bound of the inverse.
\end{proof}

\begin{corollary}\label{inverse boundedness II}
	Assume that {\upshape(A2)} holds, that $(\lambda',\mu')$ is admissible and that $\mu' > 0$. Then there exists a universal constant $C > 0$ so that:
	\begin{enumerate}
		\item [{\upshape (i)}] The operator $(\Lambda_{D_{\varepsilon}}^{\lambda',\mu',\mathrm{i}}-\Lambda_{D_{\varepsilon}}^{\lambda,\mu,\mathrm{e}})\mathbb{S}_{D_{\varepsilon}}^{\lambda,\mu} : \cH \to \cH$ is an isomorphism, and 
		\begin{equation}\label{coer 1}
			\left\| \left( (\Lambda_{D_{\varepsilon}}^{\lambda',\mu',\mathrm{i}}-\Lambda_{D_{\varepsilon}}^{\lambda,\mu,\mathrm{e}})\mathbb{S}_{D_{\varepsilon}}^{\lambda,\mu} \right)^{-1} \right\|_{\mathscr{L}(\cH)} \leq C.
		\end{equation}
		\item [{\upshape (ii)}] The operator $(\Lambda_{D_{\varepsilon}}^{\infty,\mu',\mathrm{i}}-\Lambda_{D_{\varepsilon}}^{\lambda,\mu,\mathrm{e}}) \mathbb{S}_{D_{\varepsilon}}^{\lambda,\mu} : \cH \to \cH$ is an isomorphism, and
		\begin{equation}
		\label{eq:uinvStokes}
			\left\| \left(  (\Lambda_{D_{\varepsilon}}^{\infty,\mu',\mathrm{i}}-\Lambda_{D_{\varepsilon}}^{\lambda,\mu,\mathrm{e}}) \mathbb{S}_{D_{\varepsilon}}^{\lambda,\mu}\right)^{-1} \right\|_{\mathscr{L}(\cH)} \leq C.
		\end{equation}
	\end{enumerate}
\end{corollary}

\begin{proof}
	\noindent\emph{For item {\upshape(i)}},
		by Theorem \ref{inverse boundedness}, $\Lambda_{D_{\varepsilon}}^{\lambda',\mu',\mathrm{i}}\mathbb{S}_{D_{\varepsilon}}^{\lambda,\mu}$ is non-positive self-adjoint, and $\Lambda_{D_{\varepsilon}}^{\lambda,\mu,\mathrm{e}}\mathbb{S}_{D_{\varepsilon}}^{\lambda,\mu}=\frac{\mathbb{I}}{2}+\mathbb{K}_{D_{\varepsilon}}^{\lambda,\mu,*}$ is strictly positive and self-adjoint. It follows that $(\Lambda_{D_{\varepsilon}}^{\lambda,\mu,\mathrm{e}}-\Lambda_{D_{\varepsilon}}^{\lambda',\mu',\mathrm{i}})\mathbb{S}_{D_{\varepsilon}}^{\lambda,\mu} \ge \Lambda_{D_{\varepsilon}}^{\lambda,\mu,\mathrm{e}}\bS^{\lambda,\mu}_{D_\eps}$. The desired inequality \eqref{coer 1} then follows from \eqref{eq:DtNebdd} and Lemma \ref{lem:LaxMilgram}.
		
		\noindent\emph{For item {\upshape(ii)}}, by the same reasoning above, it suffices to show $\Lambda^{\infty,\mu',{\rm i}}_{D_\eps}\bS^{\lambda,\mu}_{D_\eps}$ is non-positive and self-adjoint. By repeating the arguments in the proof for item (i) of Theorem \ref{inverse boundedness}, we get
		\begin{equation*}
			-(\Lambda_{D_{\varepsilon}}^{\infty,\mu',\mathrm{i}}\mathbb{S}_{D_{\varepsilon}}^{\lambda,\mu}\bm{\phi},\bm{\psi})_{\mathbb{S}^{\mathrm{N}}}=J^{D_{\varepsilon}}_{\infty,\mu'}(\mathcal{S}_{D_{\varepsilon}}^{\infty,\mu'}(\mathbb{S}_{D_{\varepsilon}}^{\infty,\mu'})^{-1}\mathbb{S}_{D_{\varepsilon}}^{\lambda,\mu}\bm{\phi},\mathcal{S}_{D_{\varepsilon}}^{\infty,\mu'}(\mathbb{S}_{D_{\varepsilon}}^{\infty,\mu'})^{-1}\mathbb{S}_{D_{\varepsilon}}^{\lambda,\mu}\bm{\psi}), \qquad \forall \bm{\phi}, \bm{\psi} \in H^{-\frac12}(\partial D_\eps).
		\end{equation*}
		The desired property of $\Lambda^{\infty,\mu',{\rm i}}_{D_\eps}\bS^{\lambda,\mu}_{D_\eps}$ follows and the proof is complete.
\end{proof}

	\section{Convergence rates of the high contrast limits}\label{Strong convergence to the extreme}
	
	In this section, we prove Theorem \ref{main result} which provides quantifications for the convergence of the transmission problem \eqref{eq:transmissionproblem} to the limit models, namely, \eqref{eq:limStokes} for the incompressible inclusions limit, \eqref{eq:limZero} for the soft inclusions limit and \eqref{eq:limRigid} for the hard inclusions setting. 

\subsection{Formula for the convergence analysis}
To treat the three asymptotic settings in a unified manner, let $\mathbf{u}$ denote the solution of the original transmission problem, and let $\mathbf{u}_{\rm lim}$ denote the limit model. In the exterior domain $\Omega_\eps = \Omega\setminus \ol D_\eps$, they are of the form
\begin{equation}
\label{eq:usolutions}
	\mathbf{u}=\mathbf{G}+\mathcal{S}_{D_{\varepsilon}}^{\lambda,\mu}\bm{\phi} \quad \mathrm{and} \quad \mathbf{u}_{\lim}=\mathbf{G}+\mathcal{S}_{D_\varepsilon}^{\lambda,\mu}\bm{\phi}_{\lim},
\end{equation}
where $\bm{\phi}$ is given by \eqref{boundary integral equations III}, and $\bm{\phi}_{\lim}$ is given by \eqref{boundary integral equations III stokes} in Case 1, by \eqref{equation 0,0} in Case 2, and by \eqref{boundary integral equation lambda infty 2} in Case 3. In each of those cases, subtracting the equation satisfied by $\bm{\phi}_{\rm lim}$ from the one of $\bm{\phi}$, we obtain the equations satisfied by $\bm{\phi} - \bm{\phi}_{\rm lim}$ on $\partial D_\eps$. They are summarized below:
\begin{equation}
\left\{
\label{eq:disform}
\begin{aligned}
	&(\Lambda^{\infty,\widetilde\mu,{\rm i}}_{D_\eps} - \Lambda^{\lambda,\mu,{\rm e}}_{D_\eps})\mathbb{S}_{D_\eps} (\bm{\phi} - \bm{\phi}_{\rm lim}) = (\Lambda^{\infty,\widetilde\mu,{\rm i}}_{D_\eps} - \Lambda^{\widetilde \lambda,\widetilde\mu, {\rm i}}_{D_\eps})\mathbb{S}_{D_\eps} \left(\bm{\phi} + \mathbb{S}_{D_\eps}^{-1} \mathbf{G}\right), \qquad &\text{in Case 1},\\
	&\Lambda^{\lambda,\mu,{\rm e}}_{D_\eps} \mathbb{S}_{D_\eps}(\bm{\phi} - \bm{\phi}_{\rm lim}) = \Lambda^{\widetilde\lambda,\widetilde\mu, {\rm i}}_{D_\eps}\mathbb{S}_{D_\eps} \left(\bm{\phi} + \mathbb{S}_{D_\eps}^{-1} \mathbf{G}\right), \qquad &\text{in Case 2},\\
	&\Lambda^{\frac{\widetilde\lambda}{\widetilde\mu},1,{\rm i}}_{D_\eps} \mathbb{S}_{D_\eps}(\bm{\phi} - \bm{\phi}_{\rm lim}) = \frac{1}{\widetilde\mu} \left(\Lambda^{\lambda,\mu,{\rm e}}_{D_\eps} \mathbb{S}_{D_\eps} \bm{\phi} + \frac{\partial \mathbf{G}}{\partial \nu} \right), \qquad &\text{in Case 3}.
	\end{aligned}
	\right.
\end{equation}

Our goal is to give quantitative estimates of $\|\mathbf{u} - \mathbf{u}_{\rm lim}\|_{H^1(\Omega_\eps)}$. By Remark \ref{rem:equinorm}, there is a universal constant $C > 0$ such that
\begin{equation*}
\begin{aligned}
 \|\mathbf{u} - \mathbf{u}_{\rm lim}\|_{H^1(\Omega_\eps)} &= \|\cS^{\lambda,\mu} (\bm{\phi} -\bm{\phi}_{\rm lim})\|_{H^1(\Omega_\eps)} \le \|\cS^{\lambda,\mu} (\bm{\phi} - \bm{\phi}_{\rm lim})\|_{H^1(\Omega)} \\
 &\le C\|\bD(\cS^{\lambda,\mu} (\bm{\phi} - \bm{\phi}_{\rm lim}))\|_{L^2(\Omega)} \le C\|\bm{\phi} - \bm{\phi}_{\rm lim}\|_{\bS^{\rm N}}.
 \end{aligned}
\end{equation*} 
It suffices to estimate the right hand side above in the three settings. Here and below, we simplify the notation $\mathbf{G}\rvert_{\partial D_\eps}$ to $\mathbf{G}$ when it is convenient, and also note
\begin{equation}
\label{eq:conormalG}
\frac{\partial \mathbf{G}}{\partial \nu}\Big\rvert_{\partial D_\eps} = \Lambda^{\lambda,\mu,{\rm i}}_{D_\eps} \mathbf{G}. 
\end{equation}
In view of \eqref{eq:disform} and the uniform bounds established in section \ref{sec:unifesti}, we need to control $\bm{\phi}$, $ \bS^{-1}\mathbf{G}$, and to estimate certain operators on the right hand sides of \eqref{eq:disform}.
	
	\subsection{Some basic energy estimates} We first control the background solution $\mathbf{G}$, and obtain uniform estimates of the density function $\bm{\phi}$.

\begin{lemma}\label{es G}
	Assume that {\upshape(A2)} holds. Let $\mathbf{G}$ be the background solution defined in \eqref{background solution}, then there exists a universal constant $C>0$ such that
	\begin{equation}
	\label{eq:S-1G}
		\|(\mathbb{S}_{D_{\varepsilon}}^{\lambda,\mu})^{-1} \mathbf{G}\|_{\mathbb{S}^{\mathrm{N}}} \leq C\|\mathbf{g}\|_{H^{-\frac{1}{2}}(\partial \Omega)}.
	\end{equation}
\end{lemma}
	\begin{proof}
		Let $\mathbf{v} := \mathcal{S}^{\lambda,\mu}_{D_{\varepsilon}}(\mathbb{S}_{D_{\varepsilon}}^{\lambda,\mu})^{-1} \mathbf{G}$. Then by definition, $\|(\mathbb{S}_{D_{\varepsilon}}^{\lambda,\mu})^{-1} \mathbf{G}\|_{\mathbb{S}^{\mathrm{N}}}^2 = J^{\Omega}_{\lambda,\mu}  (\mathbf{v})$. Note that $\mathbf{G}-\mathbf{v}$ vanishes on $\partial D_{\varepsilon}$. By the Green's identity,
		\begin{equation*}
			J^{\Omega}_{\lambda,\mu}\left(\mathbf{v},\mathbf{G}-\mathbf{v}\right) =J^{\Omega_\eps}_{\lambda,\mu}\left(\mathbf{v},\mathbf{G}-\mathbf{v}\right) + J^{D_\eps}_{\lambda,\mu}\left(\mathbf{v},\mathbf{G}-\mathbf{v}\right) = \int_{\partial \Omega} \frac{\partial \mathbf{v}}{\partial \nu} \cdot (\mathbf{G} -\mathbf{v}).
		\end{equation*}
		Since $\mathbf{v}$ is defined by a single-layer potential, $\partial \mathbf{v}/\partial \nu\rvert_{\partial \Omega}$ is a constant vector; note also $(\mathbf{G}-\mathbf{v})\rvert_{\partial \Omega}$ is an element of $ H^{\frac12}_\mathbf{R}$. The right hand side above, hence, vanishes, and we obtain $J^\Omega_{\lambda,\nu}(\mathbf{v}) = J^\Omega_{\lambda,\nu}(\mathbf{v},\mathbf{G})$. Then by the Schwarz inequality we get
		\begin{equation*} 
			J^{\Omega}_{\lambda,\mu}(\mathbf{v}) \leq J^{\Omega}_{\lambda,\mu}(\mathbf{G}) \sim \|\mathbf{G}\|^2_{H^1(\Omega)}.
		\end{equation*}
In the last inequality above, we used the Korn inequality that $J^\Omega_{\lambda,\mu}(\mathbf{G})$ is comparable with $\|\mathbf{G}\|^2_{H^1(\Omega)}$. The desired result then follows from the standard estimates $\|\mathbf{G}\|_{H^1(\Omega)}\leq C\|\mathbf{g}\|_{H^{-\frac{1}{2}}(\partial \Omega)}$.
	\end{proof}

\begin{lemma}\label{lem:phibdd}
Assume that {\upshape(A2)} holds and that $(\widetilde\lambda,\widetilde\mu)$ is admissible. Then there is a universal $C > 0$ (in particular, independent of $\widetilde\lambda,\widetilde\mu$), such that $\|\bm\phi\|_{\bS^{\rm N}} \le C\|\mathbf{g}\|_{H^{-\frac12}(\partial \Omega)}$.
\end{lemma}
\begin{proof} First, we rewrite the formula of $\bm\phi$, i.e., the second line in \eqref{boundary integral equations III}, as
\begin{equation*}
(\Lambda^{\widetilde{\lambda},\widetilde{\mu},\mathrm{i}}_{D_{\varepsilon}}-\Lambda^{\lambda,\mu,\mathrm{e}}_{D_{\varepsilon}})\mathbb{S}_{D_{\varepsilon}}^{\lambda,\mu}\left[\bm{\phi} + \bS^{-1}\mathbf{G}\right] = \Lambda^{\lambda,\mu,{\rm i}}_{D_\eps}\mathbf{G} - \Lambda^{\lambda,\mu,{\rm e}}_{D_\eps}\mathbf{G} = -\bS^{-1}\mathbf{G},
\end{equation*}
where we also used \eqref{eq:conormalG} and \eqref{eq:elastDtN}. This yields
\begin{equation*}
\bm{\phi} = \left(\left[-(\Lambda^{\widetilde{\lambda},\widetilde{\mu},\mathrm{i}}_{D_{\varepsilon}}-\Lambda^{\lambda,\mu,\mathrm{e}}_{D_{\varepsilon}})\mathbb{S}_{D_{\varepsilon}}^{\lambda,\mu}\right]^{-1} - \mathbb{I} \right)\bS^{-1}\mathbf{G}.
\end{equation*}
Combine \eqref{coer 1} with \eqref{eq:S-1G}, we get the desired estimate.
\end{proof}

We also need the following basic energy estimate for the Stokes system in $D_\eps$. Below, $L^2_0(E)$ stands for the space of $L^2$ functions on $E$ with zero mean.

\begin{lemma}\label{estimate stokes lemma}
				Assume that {\upshape(A2)} holds and that $\mu' > 0$. There exists a universal constant $C>0$ such that, for any $\mathbf{U}\in H^1(\Omega)$, the unique solution $(\mathbf{u}_{\infty},p)\in H^1(D_{\varepsilon})\times L^2_0(D_{\varepsilon})$ of
		\begin{equation}\label{estimate stokes}
			\left\{
			\begin{aligned}
				& \mathcal{L}_{\infty,\mu'}(\mathbf{u}_{\infty},p)=0 & \mathrm{in}\ D_{\varepsilon}, \\
				& \mathrm{div}\,\mathbf{u}_{\infty}=0 & \mathrm{in}\ D_{\varepsilon}, \\
				&\mathbf{u}_{\infty}|_{\partial D_{\varepsilon}}=\mathbf{U}|_{\partial D_{\varepsilon}},
			\end{aligned}\right.
		\end{equation}
satisfies the following estimates:
		\begin{equation}
			\| \nabla \mathbf{u}_{\infty}\|_{L^2(D_{\varepsilon})}\leq  C \|\nabla \mathbf{U}\|_{L^2(\Omega)} ,\quad
			\|p\|_{L^2(D_{\varepsilon})}\leq C \mu' \|\nabla \mathbf{U}\|_{L^2(\Omega)}.
		\end{equation}
	\end{lemma}
	\begin{proof}
		Choose one of the components of $D_{\varepsilon}$, say $\omega_{\varepsilon}^{\mathbf{n}}$, we define the rescaled functions on $\omega$:
		\begin{equation}
			\mathbf{v}(x):=\mathbf{u}_{\infty}(\varepsilon x+\mathbf{n}), \quad q(x):=\varepsilon p(\varepsilon x+\mathbf{n}),\quad \mathbf{V}(x)=\mathbf{U}(\varepsilon x+\mathbf{n}) \quad  \mathrm{for}\ x\in \omega,
		\end{equation}
		then $(\mathbf{v},q)\in H^1(\omega)\times L^2_0(\omega)$ satisfies that
		\begin{equation}\label{basic 1}
			\left\{
			\begin{aligned}
				& \mathcal{L}_{\infty,\mu'}(\mathbf{v},q)=0 & \mathrm{in}\ \omega, \\
				& \mathrm{div}\,\mathbf{v}=0 & \mathrm{in}\ \omega, \\
				&\mathbf{v}|_{\partial \omega}=\mathbf{V}|_{\partial \omega}.
			\end{aligned}\right.
		\end{equation}
	Since $\mathrm{div}:H_0^1(\omega) \rightarrow L^2_0(\omega)$ is onto (for instance \cite{temam2001navier} Lemma 2.4), there exists $\mathbf{Q}\in H_0^1(\omega)$ such that
	\begin{equation}\label{basic 2}
		\mathrm{div}\,\mathbf{Q}=q, \quad \|\mathbf{Q}\|_{H^1(\omega)}\leq C\|q\|_{L^2(\omega)},
	\end{equation}
	 multiply \eqref{basic 1} by $\mathbf{Q}$ and use Green's identity \eqref{eq:GreenStokes}, we get
		\begin{equation*}
			\mu' \int_{\omega} \nabla \mathbf{v}:\nabla \mathbf{Q}+\int_{\omega} q^2=0,
		\end{equation*}
	Cauchy-Schwarz and \eqref{basic 2} yields that
	\begin{equation}\label{basic 3}
	\|q\|_{L^2(\omega)}\leq C\mu'\|\nabla \mathbf{v}\|_{L^2(\omega)}.
	\end{equation}
	Now multiply \eqref{basic 1} by $\mathbf{v}-\mathbf{V}$, integrate by parts again and apply \eqref{basic 3}. We get
		\begin{equation*}
				\mu' \int_{\omega} |\nabla \mathbf{v}|^2= \int_{\omega} q\cdot \mathrm{div}\,\mathbf{V}+\mu' \int_{\omega } \nabla \mathbf{v}:\nabla \mathbf{V} \leq C\mu'\|\nabla \mathbf{v}\|_{L^2(\omega)} \| \nabla \mathbf{V}\|_{L^2(\omega_{\varepsilon}^{\mathbf{n}})}.
		\end{equation*}
	Therefore, $\|\nabla \mathbf{v}\|_{L^2(\omega)}\leq C\|\nabla \mathbf{V}\|_{L^2(\omega)}$. Moreover, this estimate is scaling invariant. Using this and a rescaled version of \eqref{basic 3}, we get on each $\eps$-cube:
\begin{equation*}
\begin{aligned}
&\|\nabla \mathbf{u}_{\infty}\|^2_{L^2(\omega^\mathbf{n}_\eps)} \le C\|\nabla \mathbf{U}\|^2_{L^2(\omega^\mathbf{n}_\eps)},\\
&\|p\|_{L^2(\omega^\mathbf{n}_\eps)}^2 = \eps^{d-2}\|q\|_{L^2(\omega)}^2 \le C(\mu')^2 \eps^{d-2} \|\nabla \mathbf{V}\|^2_{L^2(\omega)} =  C(\mu')^2\|\nabla \mathbf{U}\|^2_{L^2(\omega^\mathbf{n}_\eps)}.
\end{aligned}
\end{equation*}
Putting those estimates together, we obtain the desired conclusion.
	\end{proof}

	\subsection{Convergence rates of certain operators} In this subsection, we establish estimates for the operator norms of $(\Lambda^{\infty,\widetilde\mu,{\rm i}}_{D_\eps} - \Lambda^{\widetilde \lambda,\widetilde\mu, {\rm i}}_{D_\eps}) \mathbb{S}_{D_\eps}$ and $\Lambda^{\widetilde\lambda,\widetilde\mu,{\rm i}}_{D_\eps}\bS_{D_\eps}$, in the corresponding asymptotic setting. Those operators appear in the right hand sides of \eqref{eq:disform}, and those estimates can be viewed as the convergence rates of $\Lambda^{\widetilde \lambda,\widetilde\mu, {\rm i}}_{D_\eps}\bS_{D_\eps}$ to $\Lambda^{\infty,\widetilde\mu,{\rm i}}_{D_\eps}\bS_{D_\eps}$ in Case 1, and of $\Lambda^{\widetilde\lambda,\widetilde\mu,{\rm i}}_{D_\eps}$ to $0$ in Case 2.

	\begin{lemma} Assume that {\upshape(A2)} holds and that $(\widetilde\lambda,\widetilde\mu)$ is admissible. There exists a universal constant $C > 0$, such that 
	\begin{equation}
	\label{eq:oprate1}
	\left\|(\Lambda^{\infty,\widetilde\mu,{\rm i}} - \Lambda^{\widetilde \lambda,\widetilde\mu, {\rm i}}) \mathbb{S}_D \right\|_{\mathscr{L}(\mathscr{\cH})} \le C\frac{\widetilde\mu^2}{d\widetilde\lambda + 2\widetilde\mu}.
	\end{equation}
	\end{lemma}
	\begin{proof} First, we have proved in Proposition \ref{inverse boundedness} that $\Lambda_{D_{\varepsilon}}^{\widetilde{\lambda},\widetilde{\mu},\mathrm{i}}\mathbb{S}_{D_{\varepsilon}}^{\lambda,\mu}$ is a self-adjoint operator on $\cH$. The same argument shows $\Lambda_{D_{\varepsilon}}^{\infty,\widetilde{\mu},\mathrm{i}}\mathbb{S}_{D_{\varepsilon}}^{\lambda,\mu}$ is also self-adjoint. So, the desired estimate is equivalent to 
	\begin{equation}\label{524}
		\left|\left((\Lambda_{D_{\varepsilon}}^{\widetilde{\lambda},\widetilde{\mu},\mathrm{i}}\mathbb{S}_{D_{\varepsilon}}^{\lambda,\mu} - \Lambda_{D_{\varepsilon}}^{\infty,\widetilde{\mu},\mathrm{i}}\mathbb{S}_{D_{\varepsilon}}^{\lambda,\mu} )\bm{\phi},\bm{\phi} \right)_{\mathbb{S}^{\mathrm{N}}}\right| \leq C\frac{\widetilde\mu^2}{d\widetilde\lambda + 2\widetilde\mu} \|\bm{\phi} \|_{\mathbb{S}^{\mathrm{N}}}^2, \qquad \forall \bm{\phi}\in \cH.
	\end{equation}
	For any fixed $\bm{\phi}\in \cH$, let $\mathbf{v} = \cS^{\lambda,\mu}\bm\phi$, let $\mathbf{u}$ solves $\mathcal{L}_{\widetilde\lambda,\widetilde\mu}\mathbf{u} = 0$ in $D_\eps$ with $\mathbf{u} = \mathbf{v}$ on $\partial D_\eps$, and let $(\mathbf{u}_\infty,p)$ be the solution to \eqref{estimate stokes} with $\mathbf{U} = \mathbf{v}$. Then by definition, the inner product on the left hand side of \eqref{524} is precisely
	\begin{equation*}
	-\int_{\partial D_\eps} \left(\frac{\partial (\mathbf{u}_\infty,p)}{\partial \nu_{(\infty,\widetilde\mu)}} - \frac{\partial \mathbf{u}}{\partial \nu_{(\widetilde\lambda,\widetilde\mu)}}\right) \cdot (\bS\bm\phi) = J^{D_\eps}_{\widetilde\lambda,\widetilde\mu}(\mathbf{u}) - J^{D_\eps}_{\infty,\widetilde\mu}(\mathbf{u}_\infty).
	\end{equation*}
	On the other hand, $\|\bm\phi\|^2_{\bS^{\rm N}} = J^\Omega_{(\lambda,\mu)}(\mathbf{v})$ is comparable with $\|\nabla \mathbf{v}\|_{L^2(\Omega)}^2$. Hence, everything amounts to proving
	\begin{equation}
	\label{eq:opratekey1}
	\left|J^{D_\eps}_{(\widetilde\lambda,\widetilde\mu)}(\mathbf{u}) - J^{D_\eps}_{(\infty,\widetilde\mu)}(\mathbf{u}_\infty)\right| \le C\frac{\widetilde\mu^2}{d\widetilde\lambda + 2\widetilde\mu} \|\nabla\mathbf{v}\|_{L^2(\Omega)}^2.
	\end{equation}

	\noindent\emph{Step 1: An application of Lemma \ref{estimate stokes lemma} yields, for a universal constant $C$,}
	\begin{equation*}
	\|\nabla \mathbf{u}_\infty\|_{L^2(D_\eps)} \le C\|\nabla \mathbf{v}\|_{L^2(D_\eps)}, \quad \|p\|_{L^2(D_\eps)} \le C\widetilde\mu \|\nabla \mathbf{v}\|_{L^2(D_\eps)}.
	\end{equation*}

	\noindent\emph{Step 2: We prove that for some universal constant $C > 0$, 
	\begin{equation}
	\label{eq:opratekey2}
	\|\mathrm{div} \,\mathbf{u}\|^2_{L^2(D_\eps)} \le C\frac{d\widetilde\mu}{d\widetilde\lambda + 2\widetilde\mu} \|\nabla\mathbf{v}\|_{L^2(\Omega)}^2.
	\end{equation}
	}

\noindent To show this, using the inequality $|\mathrm{tr}(A)| \le d(A:A)$, we get
\begin{equation*}
J^{D_\eps}_{\widetilde\lambda,\widetilde\mu}(\mathbf{u}) \ge \left(\widetilde\lambda + \frac{2\widetilde\mu}{d}\right) \|\mathrm{div} \,\mathbf{u}\|^2_{L^2(D_\eps)}. 
\end{equation*}
By the Dirichlet principle \eqref{eq:diriprin}, and the fact that $\mathrm{div}\,\mathbf{u}_\infty = 0$,
\begin{equation*}
J^{D_\eps}_{\widetilde\lambda,\widetilde\mu}(\mathbf{u}) \le J^{D_\eps}_{\widetilde\lambda,\widetilde\mu}(\mathbf{u}_\infty) = 2\widetilde\mu\int_{D_\eps}|\bD(\mathbf{u}_\infty)|^2 \le 2\widetilde\mu\int_{D_\eps}|\nabla \mathbf{u}_\infty|^2.
\end{equation*}
Then \eqref{eq:opratekey2} follows from the estimates above and the one in Step 1.

\noindent\emph{Step 3: We derive a formula for the left hand side of \eqref{eq:opratekey1}}. Because $\mathbf{u}-\mathbf{u}_\infty = 0$ on $\partial D_\eps$, by Green's identities \eqref{eq:GreenLame} and \eqref{eq:GreenStokes} we get
			\begin{equation*}
				J^{D_{\varepsilon}}_{\lambda,\mu}(\mathbf{u},\mathbf{u}-\mathbf{u}_{\infty})=0, \quad J^{D_{\varepsilon}}_{\infty,\mu}(\mathbf{u}_{\infty},\mathbf{u}_{\infty}-\mathbf{u})=\int_{D_{\varepsilon}}p\cdot \mathrm{div}\,\mathbf{u},
			\end{equation*}
This and the fact that $\mathrm{div}\,\mathbf{u}_\infty = 0$ yield
			\begin{equation*}
				J^{D_{\varepsilon}}_{\lambda,\mu}(\mathbf{u})-J^{D_{\varepsilon}}_{\infty,\mu}(\mathbf{u}_{\infty})=J^{D_{\varepsilon}}_{\lambda,\mu}(\mathbf{u},\mathbf{u}_{\infty})-J^{D_{\varepsilon}}_{\infty,\mu}(\mathbf{u},\mathbf{u}_{\infty})-\int_{D_{\varepsilon}} p\cdot \mathrm{div}\,\mathbf{u}=-\int_{D_{\varepsilon}} p\cdot \mathrm{div}\,\mathbf{u}.
			\end{equation*}
The desired estimate \eqref{eq:opratekey1} then follows from the estimate of $p$ in Step 1 and that of $\mathrm{div}\,\mathbf{u}$ in Step 2. The proof is hence complete.
	\end{proof}

	\begin{lemma} Assume that {\upshape(A2)} holds and that $(\widetilde\lambda,\widetilde\mu)$ is admissible. There is a universal constant $C>0$, such that
	\begin{equation}
	\label{eq:oprate2}
	\left\|\Lambda^{\widetilde \lambda,\widetilde\mu, {\rm i}}\bS \right\|_{\mathscr{L}(\cH)} \le C(d\widetilde\lambda + 2\widetilde\mu).
	\end{equation}
	\end{lemma}
	\begin{proof} The operator is self-adjoint in $\cH$, so we need to prove, for any fixed $\bm\phi\in \cH$,
\begin{equation}
\label{eq:oprate-key2}
\left|\left(\Lambda^{\widetilde \lambda,\widetilde\mu, {\rm i}}\bS \bm\phi, \bm\phi\right)_{\bS^{\rm N}} \right| \le C(d\widetilde\lambda + 2\widetilde\mu)\|\bm\phi\|_{\bS^{\rm N}}^2.
\end{equation}
Let $\mathbf{v} = \cS_{\lambda,\mu}\bm\phi$, and let $\mathbf{u} = \cS^{\widetilde\lambda,\widetilde\mu}(\bS^{\widetilde\lambda,\widetilde\mu})^{-1} \bS_{\lambda,\mu}\bm\phi$. Then $\mathbf{u}$ satisfies $\mathcal{L}_{\widetilde\lambda,\widetilde\mu} \mathbf{u} = 0$ in $D_\eps$ and $\mathbf{u} = \mathbf{v}$ on $\partial D_\eps$. The inner product on the left hand side of \eqref{eq:oprate-key2} can then be written as
\begin{equation*}
-\int_{\partial D_\eps} \frac{\partial \mathbf{u}}{\partial \nu_{(\widetilde\lambda,\widetilde\mu)}} \cdot \bS^{\rm N}\bm\phi = -J^{D_\eps}_{(\widetilde\lambda,\widetilde\mu)}(\mathbf{u}).
\end{equation*}
By the Dirichlet principle, we get
\begin{equation*}
J^{D_\eps}_{(\widetilde\lambda,\widetilde\mu)}(\mathbf{u}) \le J^{D_\eps}_{(\widetilde\lambda,\widetilde\mu)}(\mathbf{v}) \le (d\widetilde\lambda+2\widetilde\mu)\int_{D_\eps}|\bD(\mathbf{v})|^2 \le (d\widetilde\lambda+2\widetilde\mu)\|\bD(\mathbf{v})\|_{L^2(\Omega)}^2.
\end{equation*}
This is precisely the desired estimate \eqref{eq:oprate-key2} since $\|\bD(\mathbf{v})\|^2_{L^2(\Omega)}$ is comparable with $J^\Omega_{(\lambda,\mu)}(\mathbf{v})$, i.e., $\|\bm\phi\|^2_{\bS^{\rm N}}$, due to Remark \ref{rem:equinorm}. The proof is hence complete.
\end{proof}

	\subsection{Proof of Theorem \ref{main result}} We now have everything to prove the main theorem of the paper. First, $\bm\phi$ and $\bS^{-1}\mathbf{G}$ are controlled in Lemma \ref{es G} and Lemma \ref{lem:phibdd}.

	\noindent\emph{Case 1: The incompressible inclusions limit}. Let $\widetilde\lambda\to \infty$ but $\widetilde\mu$ be fixed (or remains bounded). By \eqref{eq:oprate1}, the right hand side of the first line in \eqref{eq:disform} is of order $O(\widetilde\lambda^{-1})$. Take the inverse of $(\Lambda^{\infty,\widetilde\mu,{\rm i}}_{D_\eps} - \Lambda^{\lambda,\mu,{\rm e}}_{D_\eps})\mathbb{S}_{D_\eps}$ in that line and use the uniform in $\eps$ boundedness of the inverse provided in \eqref{eq:uinvStokes}, we get the desired result.

	\smallskip

	\noindent\emph{Case 2: The soft inclusions limit}. Let $(\widetilde\lambda,\widetilde\mu) \to (0,0)$, or equivalently, $\max(d\widetilde\lambda+2\widetilde\mu,\widetilde\mu)$ goes to zero. Using \eqref{eq:oprate2} we see that the right hand side of the second line in \eqref{eq:disform} is of order $O(|\widetilde\lambda| + \widetilde\mu)$. By Proposition \ref{prop:Lambdae}, the inverse of $\Lambda^{\lambda,\mu,{\rm e}}\bS$ uniformly bounded. The desired result then follows. 

	\smallskip

	\noindent\emph{Case 3: The hard inclusions limit}. Let $\widetilde\mu\to \infty$ but $\widetilde\lambda$ fixed (In fact, it suffices to keep $(\widetilde\lambda/\widetilde\mu, 1)$ uniformly admissible). Then the right hand side of the third line in \eqref{eq:disform} is of order $O(\widetilde\mu^{-1})$. By Proposition \ref{inverse boundedness}, the inverse of $\Lambda^{\widetilde\lambda/\widetilde\mu,1,{\rm i}}\bS$ can be bounded uniformly in $\eps$ and $\widetilde\mu$, so we can conclude.

\section*{Acknowledgments}
	The authors would like to thank Professor Long Jin for helpful discussions on pseudo-differential calculus for the layer potential operators associated to Lam\'e systems. The work of WJ is partially supported by the NSF of China under Grant No.\,11871300.
	
\appendix

\section{Some useful facts and technical tools}

In this appendix we record some important facts that are used frequently in this paper. 

\subsection{The second Korn's inequality}
\label{appendix A}

\begin{lemma}[The second Korn's inequality]
	\label{second korn}
	Let $\Omega$ be a bounded Lipschitz domain in $\mathbb{R}^d$, and let $V$ be a closed subspace of vector valued functions in $H^1(\Omega)$ such that $V\cap \mathbf{R} =\{0\}$, where $\mathbf{R}$ is the rigid displacements space. Then every $\mathbf{v}\in V$ satisfies
	\begin{equation*}
		\|\mathbf{v}\|_{H^1(\Omega)}\leq C\|\bD(\mathbf{v})\|_{L^2(\Omega)},
	\end{equation*}
	where the constant $C$ depends only on $\Omega$.
\end{lemma}

\noindent For the proof we refer to \cite{MR1195131}.

\subsection{A coercive lemma}
\label{appendix B}

The following version of Lax-Milgram theorem allowed us to prove that certain bounded linear operators are invertible by showing they are coercive.

	\begin{lemma}\label{lem:LaxMilgram}
	Let $H$ be a Hilbert space, and let $A \in \mathscr{L}(H)$ be a bounded linear operator. Suppose that there exists $\gamma>0$ such that
	\begin{equation}
		(Ax,x)_{H} \geq \gamma \|x\|_{H}^2 \quad \mathrm{for \ any\ }x\in H.
	\end{equation}
	Then $A$ has a bounded inverse, and $\|A^{-1}\|_{\mathscr{L}(H)}\leq \gamma^{-1}$.
\end{lemma}

 We then have the following result as a direct consequence.
 	\begin{lemma}
 	  \label{coercive theorem II}
 	Let $H$ be a Hilbert space and $H^*$ its dual space. Suppose $A : H \rightarrow H^*$ is a bounded linear operator, and, moreover, there exists $\gamma>0$ such that
 	\begin{equation}
 		\langle x,Ax \rangle_{H,H^*} \geq \gamma \|x\|_{H}^2 \quad \mathrm{for \ any\ }x\in H.
 	\end{equation}
 	Then $A$ has a bounded inverse, and $\|A^{-1}\|_{\mathscr{L}(H^*,H)}\leq \gamma^{-1}$.
 \end{lemma}
 \begin{proof}
 	Let $\iota: H^* \rightarrow H$ be the canonical dual isometry, i.e. for any $\phi\in H^*$ and for any $y\in H$, 
 	\begin{equation*}
 		(\iota(\phi), y)_{H}=\langle
 		y,\phi\rangle_{H,H^*}.
 	\end{equation*}
 It follows that $\iota A$ satisfies the conditions of Lemma \ref{lem:LaxMilgram}; indeed,
 	\begin{equation*} 
 		(\iota(Ax),x)_{H} =\langle x,Ax\rangle_{H,H^*} \geq \gamma \|x\|_{H}^2.
 	\end{equation*} 
 It follows that $\iota A$ has an inverse with the desired bound. Since $\iota$ is an isometry, the desired result for $A$ also follows.
 \end{proof}

\subsection{Proof of Lemma \ref{appendix lemma 1}}
\label{appendix C}

The uniqueness of $\mathbf{u}$ is clear, so it suffices to construct a solution.
	Let $\mathbf{v}\in H^1(\Omega)$ be the unique solution of
	\begin{equation}
		\mathcal{L}_{\lambda,\mu}\mathbf{v}=0 \ \ \mathrm{in}\ \Omega_{\varepsilon}, \quad \mathbf{v}|_{D_{\varepsilon}}=0, \quad \left. \frac{\partial \mathbf{v}}{\partial \nu_{(\lambda,\mu)}}\right|_{\partial \Omega}=\mathbf{g}.
	\end{equation}
For each $\mathbf{m}\in \Pi_{\varepsilon}$ and $l\in \{1,\cdots,\frac{d(d+1)}{2}\}$, let $\mathbf{v}_{\mathbf{m}}^l\in H^1(\Omega)$ be the solution of
	\begin{equation}
		\mathcal{L}_{\lambda,\mu} \mathbf{v}_{\mathbf{m}}^l=0 \ \ \mathrm{in}\ \Omega_{\varepsilon } , \quad \mathbf{v}_{\mathbf{m}}^l|_{\omega_{\mathbf{n}}^{\varepsilon}}=\delta_{\mathbf{m}\mathbf{n}}\mathbf{r}_l\quad \mathrm{for\ all}\ \mathbf{n}\in \Pi_{\varepsilon},\quad \left. \frac{\partial \mathbf{v}_{\mathbf{m}}^l}{\partial \nu_{(\lambda,\mu)}}\right|_{\partial \Omega}=0.
	\end{equation}
The existence and uniqueness, both for $\mathbf{v}$ and $\mathbf{v}_{\mathbf{m}}^l$, follow from a standard practice of weak formulation and an application of Lax-Milgram theorem. Moreover, the functions $\{\mathbf{v}^l_{\mathbf{m}}\}$ are independent as elements of $\mathcal{E}$.

For the solution to \eqref{ann}, we consider a function $\mathbf{u}$ of the form
	\begin{equation}\label{def u C}
		\mathbf{u}:=\mathbf{v}+\sum_{\mathbf{m}\in \Pi_{\varepsilon}}\sum_{l=1}^{\frac{d(d+1)}{2}} a_{\mathbf{m}}^l \mathbf{v}_{\mathbf{m}}^l,
	\end{equation}
	where $\{a_{\mathbf{m}}^l\}\subset \mathbb{R}$ are constants to be chosen. Clearly, $\mathbf{u} \in H^1(\Omega)$ already satisfies $\mathbf{u} \in \mathbf{R}$ in each component of $D_\eps$, and
	\begin{equation*}
			 \mathcal{L}_{\lambda,\mu}\mathbf{u} =0 \ \  \mathrm{in} \ \Omega_{\varepsilon},\quad  
			 \left. \frac{\partial \mathbf{u}}{\partial \nu_{(\lambda,\mu)}}\right|_{\partial \Omega}=\mathbf{g}.
	\end{equation*}
	We choose $a_{\mathbf{m}}^l$'s so that $\mathbf{u}$ also satisfies the remaining equation in \eqref{ann}, namely, 
	\begin{equation}\label{systemlinear}
		\sum_{\mathbf{m}\in \Pi_{\varepsilon}}\sum_{l=1}^{\frac{d(d+1)}{2}} a_{\mathbf{m}}^l \int_{\partial \omega^{\mathbf{n}}_{\varepsilon}} \left.\frac{\partial \mathbf{v}_{\mathbf{m}}^l}{\partial \nu_{(\lambda,\mu)}} \right|_+\cdot \mathbf{r}_j  = b^l_{\mathbf{n}} := c_{\mathbf{n}}^j -\int_{\partial \omega^{\mathbf{n}}_{\varepsilon}} \left.\frac{\partial \mathbf{v}}{\partial \nu_{(\lambda,\mu)}} \right|_+\cdot \mathbf{r}_j.
	\end{equation}
	The equation above can be viewed as a linear system of the form $A \mathbf{a} = \mathbf{b}$ where the unknown is $\mathbf{a} = (a^l_{\mathbf{m}}) \in \R^{|\Pi_\eps|d(d+1)/2}$, the right hand side vector $\mathbf{b} = (b^j_{\mathbf{n}})$ is defined above and the coefficient matrix $A = (A^{jl}_{\mathbf{n}\mathbf{m}})$ is defined by
	\begin{equation}\label{eq:stiffA}
		A^{jl}_{\mathbf{n}\mathbf{m}}= \int_{\partial \omega^{\mathbf{n}}_{\varepsilon}} \Big(\left.\frac{\partial \mathbf{v}_{\mathbf{m}}^l}{\partial \nu_{(\lambda,\mu)}} \right|_+\cdot \mathbf{r}_j \Big). 
	\end{equation}
	We need to show that this linear system has a solution.

	Let $\mathcal{X} \subset \mathbb{R}^{ |\Pi_{\varepsilon}|  d(d+1)/2}$ be the subspace defined by
	\begin{equation}
		\mathcal{X}:=\left\{(x_{\mathbf{n}}^j ) \in \mathbb{R}^{|\Pi_{\varepsilon}|\times \frac{d(d+1)}{2}} :\sum_{\mathbf{n}\in \Pi_{\varepsilon}} x_{\mathbf{n}}^j =0\quad \mathrm{for \ all\ }1\leq j \leq \frac{d(d+1)}{2} \right\},
	\end{equation}
	Clearly, $\mathrm{dim}\,\mathcal{X}=\frac{d(d+1)}{2}|\Pi_{\varepsilon}|-\frac{d(d+1)}{2}$. By the definition of $\mathbf{v}$ and the condition \eqref{concij}, we see $\mathbf{b} \in \mathcal{X}$. It suffices to show that the range of $A$ contains (actually, is) $\mathcal{X}$. First, we can check directly that the range of $A$ is contained in $\mathcal{X}$. Indeed, for any $\mathbf{p} = (p^l_{\mathbf{m}})$, we compute and get
\begin{equation*}
\sum_{\mathbf{n}\in \Pi_\eps} (A\mathbf{p})^j_{\mathbf{n}} = \sum_{\mathbf{n} \in \Pi_\eps} \int_{\partial \omega^{\mathbf{n}}_\eps} \Big(\left.\frac{\partial (p^l_{\mathbf{m}} \mathbf{v}_{\mathbf{m}}^l)}{\partial \nu_{(\lambda,\mu)}} \right|_+\cdot \mathbf{r}_j \Big) = -J^{\Omega_\eps}(\mathbf{\mathbf{w}},r_j) = 0,
\end{equation*}
where we defined $\mathbf{w} := p^l_{\mathbf{m}} \mathbf{v}^l_{\mathbf{m}}$ and the summation convention is used. We also used the fact that $\mathbf{w}$ solves the homogeneous Lam\'e system in $\Omega_\eps$ with $\partial \mathbf{w}/\partial \nu\rvert_{\partial \Omega} = 0$ and $\mathbf{w}\rvert_{\omega^{\mathbf{n}}} = p^l_{\mathbf{n}} \mathbf{r}_l$ (in particular, $\mathbf{w} \in \mathbf{R}$ in each component of $D_\eps$). 

For our purpose, it remains to show that the kernel of $A$ has dimension $d(d+1)/2$. Suppose $\mathbf{p} = (p^l_{\mathbf{m}})$ satisfies $A\mathbf{p} = 0$. Then the function $\mathbf{w} = p^l_{\mathbf{m}}\mathbf{v}^l_{\mathbf{m}}$ has the property discussed above and further satisfies $\int_{\partial \omega^{\mathbf{n}}_\eps} \partial \mathbf{w}/\partial \nu\rvert_+ \cdot \mathbf{r}_j = 0$ for all $j$ and $\mathbf{n}$. From this we get $J^{\Omega_\eps}(\omega) = 0$. Hence $\mathbf{w} \in \mathbf{R}$ in $\Omega_\eps$ and then $\mathbf{w} \in \mathbf{R}$ in $\Omega$. To summarize, we proved that $A\mathbf{p} = 0$ implies $\mathbf{w}(\mathbf{p}) := p^l_{\mathbf{m}}\mathbf{v}^l_{\mathbf{m}} \in \mathbf{R}$. It is easier to show that the reverse implication also holds, and that $\mathbf{p} \to \mathbf{w}(\mathbf{p})$ is an isomorphsim from $\R^{|\Pi_\eps|d(d+1)/2}$ to $\mathrm{span}(\mathbf{v}^l_{\mathbf{m}})$. Moreover $\mathbf{R}\subset \mathrm{span}\{\mathbf{v}^l_{\mathbf{m}}\}$ is a $d(d+1)/2$ dimensional subspace. The proof is hence complete. 

\subsection{Proof of Lemma \ref{ann dirichlet}}
\label{appendix D}

Again, the uniqueness is clear and we only need to construct a solution. The proof is very similar to the proof of Lemma \ref{appendix lemma 1} presented in the previous section, so we omit some details. 

Let $\mathbf{v}\in H^1(\Omega)$ be the unique solution of
	\begin{equation*}
		\mathcal{L}_{\lambda,\mu}\mathbf{v}=0 \ \ \mathrm{in}\ \Omega_{\varepsilon}, \quad  \mathbf{v}|_{D_{\varepsilon}}=0, \quad \mathbf{v}|_{\partial \Omega}=\mathbf{f}.
	\end{equation*}
For each $\mathbf{m} \in \Pi_\eps$ and $l = 1,2,\dots,d(d+1)/2$, let $\mathbf{v}_{\mathbf{m}}^l\in H^1(\Omega)$ be the unique solution of
	\begin{equation*}
		\mathcal{L}_{\lambda,\mu} \mathbf{v}_{\mathbf{m}}^l=0 \ \ \mathrm{in}\ \Omega_{\varepsilon}, \quad \mathbf{v}_{\mathbf{m}}^l|_{\omega_{\mathbf{n}}^{\varepsilon}}=\delta_{\mathbf{m}\mathbf{n}}\mathbf{r}_l \quad \mathrm{for\ all}\ \mathbf{n}\in \Pi_{\varepsilon} ,\quad \mathbf{v}_{\mathbf{m}}^l|_{\partial \Omega}=0.
	\end{equation*}
We seek a solution to \eqref{ann dirichlet equation} of the form 
	\begin{equation}\label{def u 22}
		\mathbf{u}:=\mathbf{v}+\sum_{\mathbf{m}\in \Pi_{\varepsilon}}\sum_{l=1}^{\frac{d(d+1)}{2}} a_{\mathbf{m}}^l \mathbf{v}_{\mathbf{m}}^l.
		\end{equation}
 As in the previous section, it suffices to find the constant vector $\mathbf{a} = (a^l_{\mathbf{m}})$ that solve the linear system $A\mathbf{a} = \mathbf{b}$, where $A$ and $\mathbf{b}$ has the same forms as in \eqref{eq:stiffA} and \eqref{systemlinear}, but with $\mathbf{v}^l_\mathbf{m}$'s and $\mathbf{v}$ redefined in this section. 
	We prove that $A$ is invertible so the linear system has a unique solution. Suppose $\mathbf{p} = (p^l_{\mathbf{m}})$ satisfies $A\mathbf{p} = 0$. Then $\mathbf{w} = p^l_\mathbf{m}\mathbf{v}^l_\mathbf{m}$ satisfies $\mathbf{w} \in \mathbf{R}$ in each component of $D_\eps$, solves the homogeneous Lam\'e system in $\Omega_\eps$, satisfies $\mathbf{w}\rvert_{\partial \Omega} = 0$, and, moreover, $\partial \mathbf{w}/\partial \nu\rvert_{+} = 0$ on $\partial D_\eps$. It follows that
	\begin{equation*}
		\mathbf{w} \in \frakH_{\mathbf{R}}^{\mathrm{D}}\cap \mathcal{S}^{\mathrm{D}}\,\mathrm{ker}\,\left( -\frac{\mathbb{I}}{2} +\mathbb{K}^{\mathrm{D},*} \right)=\{0\}.
	\end{equation*}
	Again, by the isometry of $\mathbf{p} \mapsto \mathrm{span}\{\mathbf{v}^l_{\mathbf{m}}\}$, we get $\mathbf{p} = 0$. Hence, the kernel of the square matrix $A$ is trivial. This completes the proof.

	\bibliographystyle{siam}
	\bibliography{mybib}

\begin{thebibliography}{10}

\bibitem{Agranovich1999SpectralLS}
{\sc M.~Agranovich, B.~Amosov, and M.~Levitin}, {\em Spectral problems for the
  lam{\'e} system with spectral parameter in boundary conditions on smooth or
  nonsmooth boundary}, Russian Journal of Mathematical Physics, 6 (1999).

\bibitem{ammari2008method}
{\sc H.~Ammari, P.~Garapon, H.~Kang, and H.~Lee}, {\em A method of biological
  tissues elasticity reconstruction using magnetic resonance elastography
  measurements}, Quarterly of Applied Mathematics, 66 (2008), pp.~139--175.

\bibitem{ammari2007polarization}
{\sc H.~Ammari and H.~Kang}, {\em Polarization and moment tensors: with
  applications to inverse problems and effective medium theory}, vol.~162,
  Springer Science \& Business Media, 2007.

\bibitem{ammari2013strong}
{\sc H.~Ammari, H.~Kang, K.~Kim, and H.~Lee}, {\em Strong convergence of the
  solutions of the linear elasticity and uniformity of asymptotic expansions in
  the presence of small inclusions}, Journal of Differential Equations, 254
  (2013), pp.~4446--4464.

\bibitem{ammari2009layer}
{\sc H.~Ammari, H.~Kang, and H.~Lee}, {\em Layer potential techniques in
  spectral analysis}, no.~153, American Mathematical Soc., 2009.

\bibitem{MR3769919}
{\sc K.~Ando, Y.-G. Ji, H.~Kang, K.~Kim, and S.~Yu}, {\em Spectral properties
  of the {N}eumann-{P}oincar\'{e} operator and cloaking by anomalous localized
  resonance for the elasto-static system}, European J. Appl. Math., 29 (2018),
  pp.~189--225.

\bibitem{MR3973113}
{\sc K.~Ando, H.~Kang, and Y.~Miyanishi}, {\em Elastic {N}eumann-{P}oincar\'{e}
  operators on three dimensional smooth domains: polynomial compactness and
  spectral structure}, Int. Math. Res. Not. IMRN,  (2019), pp.~3883--3900.

\bibitem{MR2399553}
{\sc L.~Baffico, C.~Grandmont, Y.~Maday, and A.~Osses}, {\em Homogenization of
  elastic media with gaseous inclusions}, Multiscale Model. Simul., 7 (2008),
  pp.~432--465.

\bibitem{bao2009gradient}
{\sc E.~S. Bao, Y.~Y. Li, and B.~Yin}, {\em Gradient estimates for the perfect
  conductivity problem}, Archive for rational mechanics and analysis, 193
  (2009), pp.~195--226.

\bibitem{MR3296149}
{\sc J.~Bao, H.~Li, and Y.~Li}, {\em Gradient estimates for solutions of the
  {L}am\'{e} system with partially infinite coefficients}, Arch. Ration. Mech.
  Anal., 215 (2015), pp.~307--351.

\bibitem{bonnetier2019homogenization}
{\sc {\'E}.~Bonnetier, C.~Dapogny, and F.~Triki}, {\em Homogenization of the
  eigenvalues of the neumann--poincar{\'e} operator}, Archive for Rational
  Mechanics and Analysis, 234 (2019), pp.~777--855.

\bibitem{bunoiu2020homogenization}
{\sc R.~Bunoiu, L.~Chesnel, K.~Ramdani, and M.~Rihani}, {\em Homogenization of
  maxwell's equations and related scalar problems with sign-changing
  coefficients}, in Annales de la Facult{\'e} des Sciences de Toulouse.
  Math{\'e}matiques., 2020.

\bibitem{MR548785}
{\sc D.~Cioranescu and J.~S.~J. Paulin}, {\em Homogenization in open sets with
  holes}, J. Math. Anal. Appl., 71 (1979), pp.~590--607.

\bibitem{MR4242872}
{\sc R.~Craster, A.~Diatta, S.~Guenneau, and H.~Hutridurga}, {\em On
  near-cloaking for linear elasticity}, Multiscale Model. Simul., 19 (2021),
  pp.~633--664.

\bibitem{dahlberg1988boundary}
{\sc B.~E. Dahlberg, C.~E. Kenig, and G.~C. Verchota}, {\em Boundary value
  problems for the systems of elastostatics in lipschitz domains}, Duke
  Mathematical Journal, 57 (1988), pp.~795--818.

\bibitem{dautray1999mathematical}
{\sc R.~Dautray and J.-L. Lions}, {\em Mathematical Analysis and Numerical
  Methods for Science and Technology: Volume 4 Integral Equations and Numerical
  Methods}, vol.~4, Springer Science \& Business Media, 1999.

\bibitem{escauriaza1993regularity}
{\sc L.~Escauriaza and J.~K. Seo}, {\em Regularity properties of solutions to
  transmission problems}, Transactions of the American Mathematical Society,
  338 (1993), pp.~405--430.

\bibitem{10.1007/BF02545747}
{\sc E.~B. Fabes, M.~Jodeit, and N.~M. Rivi{\`e}re}, {\em {Potential techniques
  for boundary value problems on $C^1$-domains}}, Acta Mathematica, 141 (1978),
  pp.~165 -- 186.

\bibitem{fabes1988dirichlet}
{\sc E.~B. Fabes, C.~E. Kenig, and G.~C. Verchota}, {\em The dirichlet problem
  for the stokes system on lipschitz domains}, Duke Mathematical Journal, 57
  (1988), pp.~769--793.

\bibitem{greenleaf2003selected}
{\sc J.~F. Greenleaf, M.~Fatemi, and M.~Insana}, {\em Selected methods for
  imaging elastic properties of biological tissues}, Annual review of
  biomedical engineering, 5 (2003), pp.~57--78.

\bibitem{MR4075336}
{\sc W.~Jing}, {\em A unified homogenization approach for the {D}irichlet
  problem in perforated domains}, SIAM J. Math. Anal., 52 (2020),
  pp.~1192--1220.

\bibitem{MR4172687}
\leavevmode\vrule height 2pt depth -1.6pt width 23pt, {\em Layer potentials for
  {L}am\'{e} systems and homogenization of perforated elastic medium with
  clamped holes}, Calc. Var. Partial Differential Equations, 60 (2021),
  pp.~Paper No. 2, 32.

\bibitem{Jing-Neumann}
\leavevmode\vrule height 2pt depth -1.6pt width 23pt, {\em Convergence rate for
  the homogenization of stationary diffusions in dilutely perforated domains
  with reflecting boundaries}, Minimax Theory and its Applications,  (to
  appear).

\bibitem{JLP-stokes}
{\sc W.~Jing, Y.~Lu, and C.~Prange}, {\em Stokes potentials and applications in
  homogenization problems in perforated domains}, in preparation.

\bibitem{MR2308861}
{\sc D.~Khavinson, M.~Putinar, and H.~S. Shapiro}, {\em Poincar\'{e}'s
  variational problem in potential theory}, Arch. Ration. Mech. Anal., 185
  (2007), pp.~143--184.

\bibitem{kupradze2012three}
{\sc V.~D. Kupradze}, {\em Three-dimensional problems of elasticity and
  thermoelasticity}, Elsevier, 2012.

\bibitem{Vooren1965OAL}
{\sc O.~A. Ladyzenskaja}, {\em Funktionalanalytische Untersuchungen der
  Navier--Stokesschen Gleichungen}, Akademie--Verlag, Berlin, 1965.

\bibitem{landau1986theory}
{\sc L.~Landau, E.~Lifshitz, A.~Kosevich, J.~Sykes, L.~Pitaevskii, and
  W.~Reid}, {\em Theory of Elasticity: Volume 7}, Course of theoretical
  physics, Elsevier Science, 1986.

\bibitem{manduca2001magnetic}
{\sc A.~Manduca, T.~E. Oliphant, M.~A. Dresner, J.~Mahowald, S.~A. Kruse,
  E.~Amromin, J.~P. Felmlee, J.~F. Greenleaf, and R.~L. Ehman}, {\em Magnetic
  resonance elastography: non-invasive mapping of tissue elasticity}, Medical
  image analysis, 5 (2001), pp.~237--254.

\bibitem{MR1195131}
{\sc O.~A. Ole\u{\i}nik, A.~S. Shamaev, and G.~A. Yosifian}, {\em Mathematical
  problems in elasticity and homogenization}, vol.~26 of Studies in Mathematics
  and its Applications, North-Holland Publishing Co., Amsterdam, 1992.

\bibitem{sakoda2004optical}
{\sc K.~Sakoda}, {\em Optical properties of photonic crystals}, vol.~80,
  Springer Science \& Business Media, 2004.

\bibitem{MR4267502}
{\sc Z.~Shen}, {\em Large-scale {L}ipschitz estimates for elliptic systems with
  periodic high-contrast coefficients}, Comm. Partial Differential Equations,
  46 (2021), pp.~1027--1057.

\bibitem{steinbach2007numerical}
{\sc O.~Steinbach}, {\em Numerical approximation methods for elliptic boundary
  value problems: finite and boundary elements}, Springer Science \& Business
  Media, 2007.

\bibitem{taylor2013partial}
{\sc M.~Taylor}, {\em Partial differential equations II: Qualitative studies of
  linear equations}, vol.~116, Springer Science \& Business Media, 2013.

\bibitem{temam2001navier}
{\sc R.~Temam}, {\em Navier-Stokes equations: theory and numerical analysis},
  vol.~343, American Mathematical Soc., 2001.

\bibitem{MR4240768}
{\sc L.~Wang, Q.~Xu, and P.~Zhao}, {\em Convergence rates for linear elasticity
  systems on perforated domains}, Calc. Var. Partial Differential Equations, 60
  (2021), pp.~Paper No. 74, 51.

\end{thebibliography}

	\end{document}